\documentclass[a4paper,reqno]{amsart}
\usepackage{subfiles}
\usepackage{upgreek,euscript}
\usepackage{amsfonts}
\usepackage{amsmath,amsthm,amssymb,colonequals}

\usepackage{indentfirst}
\usepackage{amscd}
\usepackage{setspace}
\usepackage{tikz}
\usepackage{mathrsfs}
\usepackage{float}
\usepackage{graphicx, subcaption}
\usepackage{wrapfig}
\usepackage{enumitem}
\usepackage{cancel}
\usepackage{graphics,graphicx, bezier, float, color}
\usepackage{url}
\usepackage{epic}
\usepackage{extpfeil}
\usepackage{mathdots}
\usepackage{mathtools}
\usetikzlibrary{calc,matrix,backgrounds}

\def\Id{\mathop{\rm{Id}}\nolimits}

\pgfdeclarelayer{overlay}
\pgfsetlayers{overlay,background,main}

\tikzset{circle/.style = {rounded corners,line width=1bp,color=#1}}%

\DeclareSymbolFont{bbold}{U}{bbold}{m}{n}
\DeclareSymbolFontAlphabet{\mathbbold}{bbold}

\usetikzlibrary{arrows,decorations.pathmorphing,decorations.pathreplacing,positioning,shapes.geometric,shapes.misc,decorations.markings,decorations.fractals,calc,patterns}

\tikzset{>=stealth',
	cvertex/.style={circle,draw=black,inner sep=1pt,outer sep=3pt},
	vertex/.style={circle,fill=black,inner sep=1pt,outer sep=3pt},
	star/.style={circle,fill=yellow,inner sep=0.75pt,outer sep=0.75pt},
	tvertex/.style={inner sep=1pt,font=\scriptsize},
	gap/.style={inner sep=0.5pt,fill=white}}

\usepackage{booktabs}
\usetikzlibrary{matrix,arrows,decorations.pathmorphing}

\usepackage[]{appendix}
\usepackage[english]{babel}

\usepackage[pagebackref]{hyperref}

\hypersetup{
  colorlinks   = true,          
  urlcolor     = blue,          
  linkcolor    = purple,          
  citecolor   = blue             
}

\setlist[enumerate]{format=\normalfont}

\newtheorem{theorem}{Theorem}[section]
\newtheorem{prop}[theorem]{Proposition}
\newtheorem{lemma}[theorem]{Lemma}
\newtheorem{definition}[theorem]{Definition}
\newtheorem{cor}[theorem]{Corollary}

\theoremstyle{definition}

\newtheorem{example}[theorem]{Example}

\newtheorem{remark}[theorem]{Remark}

\newtheorem{notation}[theorem]{Notation}

\numberwithin{equation}{section}

\DeclareMathOperator{\Rep}{\mathrm{Rep}}
\DeclareMathOperator{\Proj}{\mathrm{Proj}}
\DeclareMathOperator{\Ker}{\mathrm{Ker}}
\DeclareMathOperator{\QDet}{\mathrm{QDet}}
\DeclareMathOperator{\LM}{\mathrm{LM}}
\DeclareMathOperator{\LT}{\mathrm{LT}}
\DeclareMathOperator{\LCM}{\mathrm{LCM}}

\newcommand{\sfm}{\mathsf{m}}

\newcommand{\scrR}{\EuScript{R}}
\newcommand{\scrS}{\EuScript{S}}

\begin{document}
	\title[The Artin Component and Simultaneous Resolution]{The Artin Component and Simultaneous Resolution via Reconstruction Algebras of Type $A$ }
	\author{Brian Makonzi}
	\address{Brian Makonzi, Department of Mathematics, Makerere University, P.O Box 7260 Kampala, Uganda \& The Mathematics and Statistics Building, University of Glasgow, University Place, Glasgow, G12 8QQ, UK}
	\email{mckonzi@gmail.com,~2579843M@student.gla.ac.uk}
	\begin{abstract}
This paper uses noncommutative resolutions of non-Gorenstein singularities to construct classical deformation spaces, by recovering the Artin component of the deformation space of a cyclic surface singularity using only the quiver of the corresponding reconstruction algebra. The relations of the reconstruction algebra are then deformed, and the deformed relations together with variation of the GIT quotient achieve the simultaneous resolution. This extends work of Brieskorn, Kronheimer, Grothendieck, Cassens--Slodowy and Crawley-Boevey--Holland into the setting of singularities $\mathbb{C}^2/H$ with $H \leq \mathrm{GL} (2, \mathbb{C})$, and furthermore gives a prediction for what is true more generally.	
	\end{abstract}
	\maketitle
	\parindent 20pt
	\parskip 0pt
	
	\maketitle

\section{introduction}
Noncommutative resolutions control many geometric processes, especially in dimension three and for Calabi--Yau (CY) geometry \cite{NCCR}\cite{NCD}. This paper restricts to dimension two, but considers the much more general setting of rational surface singularities. These need not be CY. In the case of cyclic quotients, it extracts from a noncommutative resolution, namely the reconstruction algebra, a classical invariant, called the Artin component. Furthermore, by introducing a deformed version of the reconstruction algebra, simultaneous resolution is achieved.

 \subsection{Motivation and Background}
 When $H\leq\mathrm{SL}(2,\mathbb{C}),$ the quotient singularities $\mathbb{C}^2/H$ are exactly the Kleinian singularities (equivalently, rational double points), and these all have embedding dimension $e$=3. Grothendieck and Brieskorn \cite{Resolution} \cite{defspace} construct the deformation space for these singularities and relate it to the Weyl group $W$ of the corresponding simple simply--connected complex Lie group. The versal deformation $D \to \mathfrak{h}_\mathbb{C}/W$ of a rational double point was constructed in \cite{defspace}, and after base change via the action of the Weyl group as in the diagram below, the resulting space $Art$ resolves simultaneously \cite{defspace}.

\[
\begin{tikzpicture}
\node (A) at (0,0) {$Art$};
\node (B) at (2,0) {$D$};
\node (a) at (0,-1) {$\mathfrak{h}_\mathbb{C}$};
\node (b) at (2,-1) {$\mathfrak{h}_\mathbb{C}/W$};
\draw[->] (A)--(B);
\draw[->] (A)--(a);
\draw[->] (a)--(b);
\draw[->] (B)--(b);
\end{tikzpicture}
\]

Kronheimer \cite{kronheimer} and Cassens--Slodowy \cite[\S3]{OnKleinianSing} use the McKay quiver to construct the semiuniversal deformation of Kleinian singularities and their simultaneous resolutions, of type $A_n,~D_n,~E_6,~E_7$ and $E_8$. This was later reinterpreted by Crawley-Boevey--Holland \cite{crawley1998noncommutative} in terms of the deformed preprojective algebra.

\medskip
The deformation theory of non-Gorenstein surface quotient singularities, namely those $\mathbb{C}^2/H$ for small finite groups $H\leq\mathrm{GL}(2,\mathbb{C})$ that are not inside $\mathrm{SL}(2,\mathbb{C})$, is more complicated. Artin \cite{ArtinComp} constructed a particular component (the Artin component) which is irreducible and admits a simultaneous resolution, again after a finite base change by some appropriate Weyl group $W$. 
\[
\begin{tikzpicture}
\node (A) at (0,0) {$Art$};
\node (B) at (2,0) {$D$};
\node (a) at (0,-1) {$\mathrm{H}^1_\mathbb{C}$};
\node (b) at (2,-1) {$\mathrm{H}^1_\mathbb{C}/W$};
\draw[->] (A)--(B);
\draw[->] (A)--(a);
\draw[->] (a)--(b);
\draw[->] (B)--(b);
\end{tikzpicture}
\]

Riemenschneider \cite{deformRational} computed the Artin component $Art$ for cyclic quotient singularities, then later in \cite[\S5]{RieCyclic} he used the McKay quiver and special representations as described by Wunram \cite{Wunram} to give an alternative description.  The Artin component can be described as a factor of a polynomial ring $\mathbb{C}[\mathsf{z}]$ with respect to some quasideterminantal relations $\mathrm{QDet}(\mathsf{z}),$ but Riemenschneider's method recovers this only after ignoring a very large number of variables.  Simultaneous resolution is also not obtained using the McKay quiver perspective. 

\medskip
In this paper we use the reconstruction algebra of \cite{TypeA}, which is strictly smaller than the McKay quiver, to both construct the Artin component on the nose, and extract its simultaneous resolution.

\subsection{Main Results}
For any cyclic group $\frac{1}{r}(1,a)$, the quiver of the corresponding reconstruction algebra is recalled in \S\ref{subsec:ReconA}, and will be written $Q$. With dimension vector $\updelta=(1,\hdots,1)$, consider the co-ordinate ring of the representation variety $\mathbb{C}[\Rep(\mathbb{C}Q,\updelta)]$, which carries a natural action of $ G \colonequals \textstyle \prod_{q \in Q_0} \mathbb{C}^{\ast}$. As shown in \S\ref{ReconAlg}, $\scrR^G$ is generated by cycles. These generate a $\mathbb{C}$--algebra $\mathbb{C}[\mathsf{z}],$ and they further satisfy quasideterminantal relations (recalled in \S\ref{Qdet}) which we will denote $\mathrm{QDet}(\mathsf{z}).$ The following is our first main result.

\begin{theorem}[{\ref{thm: main1}}]\label{thm: main1 intro}
For any group $\frac{1}{r}(1,a)$, there is an isomorphism $\scrR^G \cong \frac{\mathbb{C}[\mathsf{z}]}{\mathrm{QDet}(\mathsf{z})}$. 
\end{theorem}
In particular $\scrR^G$, which is constructed using only the quiver of the reconstruction algebra, precisely gives the Artin component of $\frac{1}{r}(1,a)$.   Since the reconstruction algebra exists for \emph{any} rational surface singularity, this gives a prediction for what can be expected much more generally.

\medskip
Simultaneous resolution is then achieved by introducing the deformed reconstruction algebra (see \ref{SecRecAlg}), which generalises the work of 
Crawley-Boevey--Holland \cite{crawley1998noncommutative} on deformed preprojective algebras. In \S\ref{SimulRes}, we construct a map $\uppi\colon\scrR^G \rightarrow \Delta,$ where $\Delta$ is an affine space defined in  \eqref{def:Delta}. The following is our second main result, where $\upvartheta$ is a \emph{particular} choice of stability condition explained in \S\ref{ModuliofDeformed}. 

\begin{theorem}[{\ref{thm: main2}}]\label{thm: main2 intro}For any cyclic group $\frac{1}{r}(1,a)$, the diagram
\[
\begin{tikzpicture}
\node (A) at (0,0) {$\Rep(\mathbb{C}Q,\updelta) /\!\!\!\!/_\upvartheta \mathrm{GL}$};
\node (B) at (4,0) {$\scrR^G$};
\node (b) at (4,-2) {$\Delta$};
\draw[->] (A)-- node[above]  {} (B);
\draw[densely dotted,->] (A)-- node[below]  {$\upphi$} (b);
\draw[->] (B)-- node[right]  {$\uppi$} (b);
a\end{tikzpicture}
\]
is a simultaneous resolution of singularities, in the sense that the morphism $\upphi$ is smooth, and $\uppi$ is flat.
\end{theorem}

The smoothness of the fibres is achieved using moduli spaces of the deformed reconstruction algebra $A_{r,a,\boldsymbol{\lambda}}.$ These are introduced in \S\ref{SecRecAlg}, and may be of independent interest. As a final remark, we note in Remark \ref{genericstability} that in general the particular choice of $\upvartheta$ in Theorem \ref{thm: main2} is important, and cannot be generalised to arbitrary generic stability parameters.
\medskip

This paper is organised as follows. Section \ref{Preliminaries} recalls the reconstruction algebra associated to any cyclic subgroup of $\mathrm{GL} (2, \mathbb{C}),$ and recalls quasideterminantal form. Section \ref{RepVar} proves that the invariant representation variety associated to the quiver of this reconstruction algebra is generated by certain cycles $z_{i,j}.$ In Section \ref{ArtinComp} the Artin component is obtained. Section \ref{SimulResolution} introduces the deformed reconstruction algebra, and uses this to achieve simultaneous resolution.

\subsection*{Conventions} Throughout we work over the complex numbers $\mathbb{C}.$ For quivers, $ab$ denotes $a$ followed by $b$.

\subsection*{Acknowledgements} The author would like to thank the GRAID program for sponsoring his PhD studies, and the
ERC Consolidator Grant 101001227 (MMiMMa) for funding his visit to Glasgow. Furthermore,  he would like to thank in a special way his supervisors Michael Wemyss and David Ssevviiri for their guidance and encouragement, and the Department of Mathematics, Makerere University and the School of Mathematics \& Statistics, University of Glasgow for the hospitality and warm environment to pursue his studies.
	
	\section{preliminaries}\label{Preliminaries}
	This section recalls the reconstruction algebra of Type $A$, and introduces some combinatorics that will be used later.
	\subsection{The Reconstruction Algebra of Type  \texorpdfstring{$A$}{A}}\label{subsec:ReconA}
	Consider, for positive integers $\upalpha_{i} \geq 2,$ the following labelled Dynkin diagram of Type $A_n$
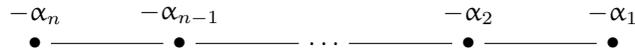
	\begin{figure}[H]
		\centering
		\begin{tikzpicture}[scale=1.9] 
		\node (A) at (-1,0) {$\bullet$};
		\node (B) at (0,0) {$\bullet$};
		\node (C) at (1,0) {$\ldots$}; 
		\node (D) at (2,0) {$\bullet$}; 
		\node (E) at (3,0) {$\bullet$}; 
		\node (W) at (-1,0.2) {$-\upalpha_n$};
		\node (X) at (0,0.2) {$-\upalpha_{n-1}$}; 
		\node (Y) at (2,0.2) {$-\upalpha_2$};
		\node (Z) at (3,0.2) {$-\upalpha_1$};
		\path[-,font=\scriptsize,>=angle 90] 
		(A) edge node[above]{} (B)
		(B) edge node[above]{} (C)
		(C) edge node[above]{} (D)
		(D) edge node[above]{} (E); 
		\end{tikzpicture}
	\end{figure}
	\noindent We call the vertex corresponding to $\upalpha_{i}$ the $i^{th}$ vertex. To this picture we associate the double quiver of the extended Dynkin quiver, with the extended vertex called the $0^{th}$ vertex:
	
	
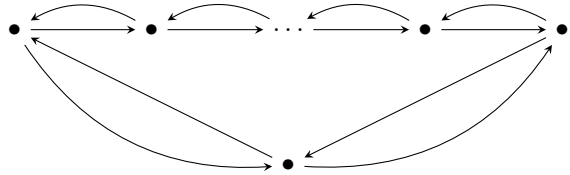
\begin{figure}[H]
		\centering
		\begin{tikzpicture}[scale=1.8] 
		\node (A) at (-1,0) {$\bullet$};
		\node (B) at (0,0) {$\bullet$}; 
		\node (C) at (1,0) {$\ldots$}; 
		\node (D) at (2,0) {$\bullet$}; 
		\node (E) at (3,0) {$\bullet$};
		\node (F) at (1,-1) {$\bullet$}; 
		\path[->,font=\scriptsize,>=angle 90,>=stealth] 
		(A) edge node[above]{} (B)
		(B) edge node[above]{} (C)
		(C) edge node[above]{} (D)
		(D) edge node[above]{} (E);
		\path[->,font=\scriptsize,>=angle 90,>=stealth] 
		(B) edge[bend right] node[above]{} (A)
		(C) edge[bend right] node[above]{} (B)
		(D) edge[bend right] node[above]{} (C)
		(E) edge[bend right] node[above]{} (D)
		(F) edge node[above]{} (A)
		(A) edge[bend right] node[above]{} (F)
		(E) edge node[above]{} (F)
		(F) edge[bend right] node[above]{} (E); 
		\end{tikzpicture}
	\end{figure}
	
	\noindent
	Denote this quiver $Q^{\prime}$, and we remark that for $n=1$ $Q^{\prime}$ is
	
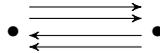
\begin{figure}[H]
		\centering
		\begin{tikzpicture}[scale=1.9] 
		\node (A) at (-1,0) {$\bullet$};
		\node (B) at (0,0) {$\bullet$};  
		\draw[transform canvas={yshift=1.2ex},->] (A) -- (B);
		\draw[transform canvas={yshift=2.2ex},->] (A) -- (B);
		\draw[transform canvas={yshift=-0.3ex},<-] (A) -- (B);    
		\draw[transform canvas={yshift=-1.3ex},<-] (A) -- (B);
		\end{tikzpicture}
	\end{figure}

\noindent	
In the case that some $\upalpha_{i}>2,$ add an additional $\upalpha_{i}-2$ arrows from the $i^{th}$ vertex to the $0^{th}$ vertex. The resulting quiver is denoted $Q,$ and we label its arrows as follows:

\medskip
\noindent
For $n=1$, we write
	\begin{itemize}
		\item $c_1, c_2$ for the two arrows from 0 to 1 in $Q^\prime$. 
		\item $a_1, a_2$ for the two arrows from 1 to 0 in $Q^\prime$.
		\item $k_1, \ldots , k_{\upalpha_1 - 2}$ for the extra arrows if $\upalpha_1 > 2.$
	\end{itemize}
For $n \geq 2$, we write the
	\begin{itemize}
		\item clockwise arrow in $Q^\prime$ from $i$ to $i-1$ as $c_{ii-1}$ (and $c_{0n}$)
		\item anticlockwise arrow in $Q^\prime$ from $i$ to $i+1$ as $a_{ii+1}$ (and $a_{n0}$)
		\item extra arrows as $k_1, \ldots , k_{\sum(\upalpha_i - 2)}$, reading from right to left (see Examples below).
	\end{itemize}
	The notation $a_{12}$ is read `anticlockwise from 1 to 2'. Below, we furthermore write $A_{ij}$ for the composition of anticlockwise paths $a$ from vertex $i$ to $j,$ and $C_{ij}$ as the composition of clockwise paths. Note that by convention $C_{ii}$ (resp.\ $A_{ii}$) is not an empty path at vertex $i$ but rather the path from $i$ to $i$ round each of the clockwise (resp. anticlockwise) arrows precisely once. Lastly, for convenience write $c_{10}\colonequals k_0$ and $a_{n0} \colonequals  k_{1+\sum(\upalpha_i-2)}$.
	
	\begin{example}\label{rec1}
		For $[\upalpha_1, \upalpha_2, \upalpha_3] = [3,2,2]$, the labelled quiver $Q$ is
		\[
		\begin{tikzpicture} [bend angle=45, looseness=1]
		\node[name=s,regular polygon, regular polygon sides=4, minimum size=3cm] at (0,0) {}; 
		\node (C1) at (s.corner 1) []  {$\scriptstyle 2$};
		\node (C2) at (s.corner 2) []  {$\scriptstyle 3$};
		\node (C3) at (s.corner 3) []  {$\scriptstyle 0$};
		\node (C4) at (s.corner 4) []  {$\scriptstyle 1$};
		\draw[<-] (C4) -- node[left]  {$\scriptstyle c_{21}$} (C1); 
		\draw[<-] (C3) -- node[below]  {$\scriptstyle k_0= c_{10}$} (C4); 
		\draw[<-] (C2) -- node[right]  {$\scriptstyle c_{03}$} (C3); 
		\draw[<-] (C1) -- node[below]  {$\scriptstyle c_{32}$} (C2);
		\draw[<-, transform canvas={yshift=0.8ex}] [green!70!black] (C3) -- node[above]  {$\scriptstyle k_1$} (C4) ;
		\draw [->,bend right] (C1) to node[gap]  {$\scriptstyle a_{23}$} (C2);
		\draw [->,bend right] (C2) to node[gap]  {$\scriptstyle k_2= a_{30}$} (C3);
		\draw [->,bend right] (C3) to node[gap]  {$\scriptstyle a_{01}$} (C4);
		\draw [->,bend right] (C4) to node[gap]  {$\scriptstyle a_{12}$} (C1);
		\end{tikzpicture}\]
	\end{example}

	\begin{example}\label{rec2}
		For $[\upalpha_1, \upalpha_2, \upalpha_3, \upalpha_4, \upalpha_5, \upalpha_6, \upalpha_7] = [2,3,2,4,3,2,2]$, the labelled quiver $Q$ is 
		
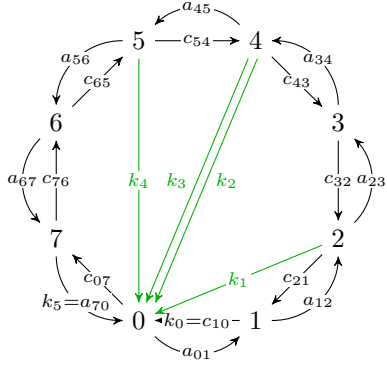
\[
		\begin{tikzpicture} [bend angle=45, looseness=1]
		\node[name=s,regular polygon, regular polygon sides=8, minimum size=4cm] at (0,0) {}; 
		\node (C1) at (s.corner 1)  {$4$};
		\node (C2) at (s.corner 2)  {$ 5$};
		\node (C3) at (s.corner 3)  {$ 6$};
		\node (C4) at (s.corner 4)  {$ 7$};
		\node (C5) at (s.corner 5)  {$ 0$};
		\node (C6) at (s.corner 6)  {$ 1$};
		\node (C7) at (s.corner 7)  {$ 2$};
		\node (C8) at (s.corner 8)  {$ 3$};    
		\draw[->] (C4) -- node[gap]  {$\scriptstyle c_{76}$} (C3); 
		\draw[->] (C3) -- node[gap]  {$\scriptstyle c_{65}$} (C2);
		\draw[->] (C2) -- node[gap]  {$\scriptstyle c_{54}$} (C1);
		\draw[->] (C1) -- node[gap]  {$\scriptstyle c_{43}$} (C8);
		\draw[->] (C8) -- node[gap]  {$\scriptstyle c_{32}$} (C7);
		\draw[->] (C7) -- node[gap]  {$\scriptstyle c_{21}$} (C6);
		\draw[->] (C6) -- node[gap]  {$\scriptstyle k_0=c_{10}$} (C5);
		\draw[->] (C5) -- node[gap]  {$\scriptstyle c_{07}$} (C4);
		\draw[->, green!70!black] (C2) -- node[gap]  {$\scriptstyle k_4$} (C5);
		\draw[->, green!70!black] (C7) -- node[gap]  {$\scriptstyle k_1$} (C5);
		\draw[->, green!70!black] (C1) -- node[left]  {$\scriptstyle k_3$} (C5);
		\draw[->, transform canvas={xshift=0.8ex}] [green!70!black] (C1) -- node[right]  {$\scriptstyle k_2$} (C5) ;
		\draw [->,bend right] (C1) to node[gap]  {$\scriptstyle a_{45}$} (C2);
		\draw [->,bend right] (C2) to node[gap]  {$\scriptstyle a_{56}$} (C3);
		\draw [->,bend right] (C3) to node[gap]  {$\scriptstyle a_{67}$} (C4);
		\draw [->,bend right] (C4) to node[gap]  {$\scriptstyle k_5=a_{70}$} (C5);
		\draw [->,bend right] (C5) to node[gap]  {$\scriptstyle a_{01}$} (C6);
		\draw [->,bend right] (C6) to node[gap]  {$\scriptstyle a_{12}$} (C7);
		\draw [->,bend right] (C7) to node[gap]  {$\scriptstyle a_{23}$} (C8);
		\draw [->,bend right] (C8) to node[gap]  {$\scriptstyle a_{34}$} (C1);
		\end{tikzpicture} \]
	\end{example}

	\subsection{Cyclic Groups and Combinatorics}
	A reconstruction algebra can be associated to any cyclic subgroup of $\mathrm{GL} (2, \mathbb{C}).$
	\begin{definition}\label{Hirzebruch-Jung}
		For $r,a \in \mathbb{N}$ with $(r,a)=1$ and $r > a$, the group $ \frac{1}{r}(1,a)$ is defined to be
		\[\frac{1}{r}(1,a) \colonequals
		\left<\upzeta\colonequals  \left( \begin{array}{cc}
		\upvarepsilon & 0  \\
		0 & \upvarepsilon^a  \\
		\end{array} \right)\right> \leq \mathrm{GL}(2,\mathbb{C}),\]
		where $\upvarepsilon$ is a primitive $r^{th}$ root of unity. The Hirzebruch--Jung continued fraction expansion of $\frac{r}{a}$ is then denoted
		\[\dfrac{r}{a} = \upalpha_1 - \dfrac{1}{\upalpha_2 - \frac{1}{\upalpha_3 - \frac{1}{(\cdots)}}} \colonequals  [\upalpha_1, \ldots , \upalpha_n]\] with each $\upalpha_i \geq  2.$ For $\frac{r}{r-a}$, the Hirzebruch--Jung expansion is written
\begin{equation}
		\dfrac{r}{r-a} = \upbeta_1 - \dfrac{1}{\upbeta_2 - \frac{1}{\upbeta_3 - \frac{1}{(\cdots)}}} \colonequals  [\upbeta_1, \ldots , \upbeta_m].\label{eqn:betas}
\end{equation}
\end{definition}
   \noindent Write $e$ for the the embedding dimension of the singularity $\mathbb{C}[x,y]^{\frac{1}{r}(1,a)}$.  Then by \cite[\S3]{deformationen} there is an equality $e = m + 2 = 3+\sum(\upalpha_i-2).$ 

To be consistent with \cite[3.5]{TypeA}, consider the $\mathbbold{i}$ and $\mathbbold{j}$-series of \eqref{eqn:betas}, which is defined to be:
\begin{equation}
\begin{array}{clll}
\mathbbold{i}_0=r & \mathbbold{i}_1=r-a & \mathbbold{i}_{t}=\upbeta_{t-1}\mathbbold{i}_{t-1}-\mathbbold{i}_{t-2} &\mbox{ for }\,2\leq t\leq m+1,\\ 
\mathbbold{j}_0=0 & \mathbbold{j}_1=1 &
\mathbbold{j}_{t}=\upbeta_{t-1}\mathbbold{j}_{t-1}-\mathbbold{j}_{t-2}& \mbox{ for }\,2\leq t\leq m+1.
\end{array}\label{eqn:iAndj}
\end{equation}
It is well known that the collection $x^{\mathbbold{i}_t}y^{\mathbbold{j}_t}$ for all $t$ such that $0\leq t\leq m+1$ generate the invariant ring \cite[Satz1]{RieInv}.	
	
	\begin{definition}[{\cite[\S2]{TypeA}}]\label{recgroup}
The reconstruction algebra $A_{r,a}$ associated to the group $\frac{1}{r}(1,a)$ is the path algebra of the quiver $Q$ in \S\textnormal{\ref{subsec:ReconA}} associated to the Hirzebruch--Jung continued fraction expansion of $\frac{r}{a},$ subject to the relations given in Definition \textnormal{\ref{DefRecAlg}} with all $\boldsymbol{\lambda}$s equal to zero.
	\end{definition}

    For our purposes, we shall not require the relations until \S\ref{SimulResolution}, and so we defer introducing them until then.
	
	\begin{example}
		Since $\frac{7}{3}=[3,2,2]$ the quiver of the reconstruction algebra $A_{7,3}$ associated to the group $\frac{1}{7}(1,3)$ is precisely the quiver in Example \ref{rec1}. The relations can be found in Example \ref{rec3}, after setting all $\boldsymbol{\lambda}$s equal to zero.
	\end{example}
	
	\begin{example}
		Since $\frac{165}{104}=[2,3,2,4,3,2,2]$ the quiver of the reconstruction algebra $A_{165,104}$ associated to the group $\frac{1}{165}(1,104)$ is precisely the quiver in Example \ref{rec2}. The relations can be found in Example \ref{rec4}, after setting all $\boldsymbol{\lambda}$s equal to zero.
	\end{example}
	
	\subsection{Quasideterminantal form}\label{QDetform}
	Consider a $2 \times n$ matrix

		\[\begin{pmatrix}
		a_1&  &a_2&  &\hdots& & a_n \\
		b_1& &b_2& &\hdots& & b_n \\ 
		\end{pmatrix}\]
	 together with $n-1$ further entries $W_1, \hdots , W_{n-1}.$ We then write these entries in the middle row, as follows.	
	\[
	X=\begin{pmatrix}
	a_1&  &a_2&  &\hdots& & a_n \\
	&W_1&&W_2&&W_{n-1}&\\
	b_1& &b_2& &\hdots& & b_n \\ 
	\end{pmatrix}\]
	
	Following Riemenschneider \cite[\S5]{RieCyclic}, consider the $2 \times 2$ \textit{quasiminors} of this $2 \times n$ \textit{quasimatrix}, which for all, $i < j$ are defined to be
	\[ a_i \cdot b_j - b_i \left( \prod_{t=i}^{j-1} W_t\right) a_j. \]
	Write $\mathrm{QDet}(X)$ for the set of all $2 \times 2$ \textit{quasiminors} of $X$.
	
	\begin{example}If
		\[X=\begin{pmatrix}
		a_1&  & a_2 & & a_3 \\
		&W_1&&W_2&\\
		b_1& & b_2 & & b_3 \\ 
		\end{pmatrix},\]
		then
		\[ \mathrm{QDet}(X) = \{a_1b_2 - b_1 W_1 a_2,~ a_1b_3- b_1W_1W_2a_3, ~a_2b_3-b_2W_2a_3 \}. \]
	\end{example}
	
	\section{The Representation Variety}\label{RepVar}
	This section considers the invariant representation variety associated to the quiver of any reconstruction algebra of Type $A$, and finds its generators in terms of cycles.
	
	\subsection{Generalities}	
	Consider the dimension vector $\updelta=(1,\hdots,1)$ and the representation variety $\Rep(\mathbb{C}Q,\updelta)$, where $Q$ is an arbitrary (finite) quiver. Here $\Rep(\mathbb{C}Q,\updelta)$ is just an affine space, and we write $\scrR\colonequals \mathbb{C}[\Rep(\mathbb{C}Q,\updelta)]$ for its co-ordinate ring, which we identify with the polynomial ring in the number of arrow variables. The co-ordinate ring carries a natural action of $ G \colonequals \textstyle \prod_{q \in Q_0} \mathbb{C}^{\ast}$ where $Q_0$ denotes the set of vertices of $Q$. The action is via conjugation, namely $ \upmu \in G = \mathbb{C}^\ast \times  \hdots \times \mathbb{C}^\ast$ acts on an arrow $p \in \scrR$ as $\upmu \cdot p = \upmu^{-1}_{t(p)} p \upmu_{h(p)}$.

	Below, we say that arrows $p_1,\hdots,p_n$ are \textit{composable} if $h(p_i) = t(p_{i+1}) ~\text{for all}~ i=1, \hdots , n-1.$

		\begin{lemma}\label{lemcyclesgen}
		If $Q$ is an arbitrary (finite) quiver, then $\mathcal{\scrR}^G$ is generated by cycles in $Q$.
	\end{lemma}
	\begin{proof} Choose a monomial $p = p_1 \hdots p_n \in \scrR,$ where ${p_i}$'s are arrows. We claim that
		$\upmu \cdot p = p ~ \text{for all} ~ \upmu \Leftrightarrow p$ is a cycle. First observe that $\upmu \cdot p 
		= (\upmu_{t(p_1)} \hdots \upmu_{t(p_n)})^{-1} p (\upmu_{h(p_1)} \hdots \upmu_{h(p_n)})$.

\noindent		
($\Leftarrow$) If $p$ is a cycle, in particular it is composable. Thus for all $\upmu \in G,$ 
		\begin{align*}
		\upmu \cdot p &= \upmu_{t(p_1)}^{-1} p_1\upmu_{h(p_1)}\upmu_{t(p_2)}^{-1}p_2\upmu_{h(p_2)} \hdots  \upmu_{t(p_n)}^{-1}p_n \upmu_{h(p_n)}\\
		&= \upmu_{t(p_1)}^{-1} \upmu_{h(p_n)} p_1 p_2  \hdots p_n\\
		&= \upmu_{t(p_1)}^{-1} \upmu_{h(p_n)} p\\
		&= p.  \tag{since $t(p_1) = h(p_n)$}
		\end{align*}
Hence $ p \in \scrR^G$.

\noindent		
($\Rightarrow$)
		Suppose that $p \in \scrR^G$ such that $\upmu \cdot p = p$ for all $\upmu$. Then $\upmu_{h(p_1)}$ must cancel some $\upmu_{t(p_i)}^{-1}$ for some $i$, so $h(p_1) = t(p_i).$ Now consider $\upmu_{h(p_i)}.$ It must cancel $\upmu_{t(p_j)}^{-1}$ for some $j$, so $h(p_i) = t(p_j).$
		Continuing like this, we can assume $p = p_1p_ip_j \hdots p_m$ where $p_1p_ip_j \hdots p_m$ is composable.
		But then $\upmu \cdot p = \upmu_{t(p_1)}^{-1} \cdot p \cdot \upmu_{h(p_m)}$ and so since $\upmu \cdot p = p$, $t(p_1) = h(p_m),$ and $p$ is a cycle.
	\end{proof}
	
	\subsection{Reconstruction Algebras}\label{ReconAlg}
	We now specialise to the case where $Q$ is the quiver of the reconstruction algebra of  $\S$\ref{subsec:ReconA}. By Lemma \ref{lemcyclesgen}, $\scrR^G$ is generated by cycles and this subsection finds a finite generating set.
     
     \medskip
     To set notation, for $h$ such that $0 \leq h \leq 1+ \sum(\upalpha_i - 2)$, write $l_h$ for the number of the vertex associated to the tail of the arrow $k_h.$ In Example \ref{rec2} above, $l_2=4,~ l_3=4$ and $l_4=5$ are associated to the tail of the arrows $k_2,~k_3$ and $k_4$ respectively.
     
     \medskip
    Consider		
\begin{equation} \label{eq1}
\begin{split}
	&\kern 11pt z_{0,0} = C_{00} \\
	\text{for}~ 1 \leq i \leq e-2 & \begin{cases}
	z_{i,0} = C_{0l_i}k_i  \\
	z_{i,j} = c_{l_i-(j-1),l_i-j}a_{l_i-j,l_i-(j-1)}&  ~\forall~   1 \leq j \leq l_i - l_{i-1}\\
	z_{i,l_i-l_{i-1}+1} = A_{0l_{i-1}}k_{i-1} & 
	\end{cases}\\
	&\kern 11pt  z_{e-1,0} = A_{00}
\end{split}
\end{equation}
		
	\begin{prop}\label{gen Z}
		For any group $\frac{1}{r}(1,a),~ \scrR^G$ is generated as a $\mathbb{C}$--algebra by the set 
		\[S=\{z_{0,0},z_{i,j},z_{e-1,0} \mid i \in \left[1,~e-2\right], ~j\in \left[0,~l_i-l_{i-1}+1\right] \}.\]
	\end{prop}
	Before proving the Proposition, we illustrate the set $S$ in the two running examples.
	
	\begin{example}\label{gen 1}
	The quiver of the reconstruction algebra associated to $\frac{1}{7}(1,3)$ is given in Example \ref{rec1}. The set $S$ is
		
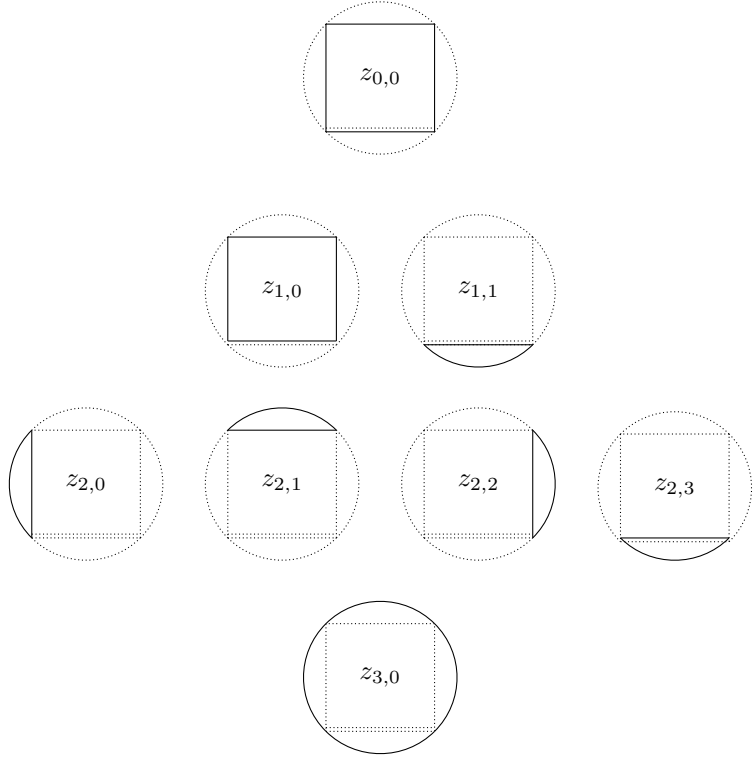
\[
		\begin{tikzpicture} [bend angle=45, looseness=1]
		\node[name=s,regular polygon, regular polygon sides=4, minimum size=2cm] at (0,0) {$z_{0,0}$}; 
		\coordinate (C1) at (s.corner 1) [] ;
		\coordinate (C2) at (s.corner 2) [] ;
		\coordinate (C3) at (s.corner 3) [] ;
		\coordinate (C4) at (s.corner 4) [] ;
		\coordinate (D3) at ($(C3)+(90:0.05)$) [] ;
		\coordinate (D4) at ($(C4)+(90:0.05)$) [] ;
		\draw[-] (C4) -- (C1); 
		\draw[-] (C3) -- (C4); 
		\draw[-] (C2) -- (C3); 
		\draw[-] (C1) -- (C2);
		\draw[densely dotted,-] [black] (D3) -- (D4) ;
		\draw [densely dotted,-,bend right] (C1) to (C2);
		\draw [densely dotted,-,bend right] (C2) to (C3);
		\draw [densely dotted,-,bend right] (C3) to (C4);
		\draw [densely dotted,-,bend right] (C4) to (C1);
		\end{tikzpicture}\]
		
		\[
		\begin{array}{cc}
		\begin{tikzpicture} [bend angle=45, looseness=1]
		\node[name=s,regular polygon, regular polygon sides=4, minimum size=2cm] at (0,0) {$z_{1,0}$}; 
		\coordinate (C1) at (s.corner 1) [] ;
		\coordinate (C2) at (s.corner 2) [] ;
		\coordinate (C3) at (s.corner 3) [] ;
		\coordinate (C4) at (s.corner 4) [] ;
		\draw[-] (D4) -- (C1); 
		\draw[densely dotted,-] (C3) -- (C4); 
		\draw[-] (C2) -- (D3); 
		\draw[-] (C1) -- (C2);
		\draw[-] [black] (D3) -- (D4) ;
		\draw [densely dotted,-,bend right] (C1) to (C2);
		\draw [densely dotted,-,bend right] (C2) to (C3);
		\draw [densely dotted,-,bend right] (C3) to (C4);
		\draw [densely dotted,-,bend right] (C4) to (C1);
		\end{tikzpicture} 
		&
		\begin{tikzpicture} [bend angle=45, looseness=1]
		\node[name=s,regular polygon, regular polygon sides=4, minimum size=2cm] at (0,0) {$z_{1,1}$}; 
		\coordinate (C1) at (s.corner 1) [] ;
		\coordinate (C2) at (s.corner 2) [] ;
		\coordinate (C3) at (s.corner 3) [] ;
		\coordinate (C4) at (s.corner 4) [] ;
		\coordinate (D3) at ($(C3)+(90:0.05)$) [] ;
		\coordinate (D4) at ($(C4)+(90:0.05)$) [] ;
		\draw[densely dotted,-] (C4) -- (C1); 
		\draw[-] (C3) -- (C4); 
		\draw[densely dotted,-] (D3) -- (D4); 
		\draw[densely dotted,-] (C2) -- (C3); 
		\draw[densely dotted,-] (C1) -- (C2);
		\draw[densely dotted,-] [black] (C3) -- (C4) ;
		\draw [densely dotted,-,bend right] (C1) to (C2);
		\draw [densely dotted,-,bend right] (C2) to (C3);
		\draw [-,bend right] (C3) to (C4);
		\draw [densely dotted,-,bend right] (C4) to (C1);
		\end{tikzpicture} 
		\end{array}
		\]
		
		\[
		\begin{array}{cccc}
		\begin{tikzpicture} [bend angle=45, looseness=1]
		\node[name=s,regular polygon, regular polygon sides=4, minimum size=2cm] at (0,0) {$z_{2,0}$}; 
		\coordinate (C1) at (s.corner 1) [];
		\coordinate (C2) at (s.corner 2) [];
		\coordinate (C3) at (s.corner 3) [];
		\coordinate (C4) at (s.corner 4) [];
		\coordinate (D3) at ($(C3)+(90:0.05)$) [] ;
		\coordinate (D4) at ($(C4)+(90:0.05)$) [] ;
		\draw[densely dotted,-] (C4) -- (C1); 
		\draw[densely dotted,-] (C3) -- (C4); 
		\draw[-] (C2) -- (C3); 
		\draw[densely dotted,-] (C1) -- (C2);
		\draw[densely dotted,-] (D3) -- (D4); 
		\draw [densely dotted,-,bend right] (C1) to (C2);
		\draw [-,bend right] (C2) to (C3);
		\draw [densely dotted,-,bend right] (C3) to (C4);
		\draw [densely dotted,-,bend right] (C4) to (C1);
		\end{tikzpicture} 
		&
		\begin{tikzpicture} [bend angle=45, looseness=1]
		\node[name=s,regular polygon, regular polygon sides=4, minimum size=2cm] at (0,0) {$z_{2,1}$}; 
		\coordinate (C1) at (s.corner 1) [] ;
		\coordinate (C2) at (s.corner 2) [] ;
		\coordinate (C3) at (s.corner 3) [] ;
		\coordinate (C4) at (s.corner 4) [] ;
		\coordinate (D3) at ($(C3)+(90:0.05)$) [] ;
		\coordinate (D4) at ($(C4)+(90:0.05)$) [] ;
		\draw[densely dotted,-] (C4) -- (C1); 
		\draw[densely dotted,-] (C3) -- (C4); 
		\draw[densely dotted,-] (C2) -- (C3); 
		\draw[-] (C1) -- (C2);
		\draw[densely dotted,-] (D3) -- (D4);
		\draw [-,bend right] (C1) to (C2);
		\draw [densely dotted,-,bend right] (C2) to (C3);
		\draw [densely dotted,-,bend right] (C3) to (C4);
		\draw [densely dotted,-,bend right] (C4) to (C1);
		\end{tikzpicture} 
		&
		\begin{tikzpicture} [bend angle=45, looseness=1]
		\node[name=s,regular polygon, regular polygon sides=4, minimum size=2cm] at (0,0) {$z_{2,2}$}; 
		\coordinate (C1) at (s.corner 1) [];
		\coordinate (C2) at (s.corner 2) [];
		\coordinate (C3) at (s.corner 3) [];
		\coordinate (C4) at (s.corner 4) [];
		\coordinate (D3) at ($(C3)+(90:0.05)$) [] ;
		\coordinate (D4) at ($(C4)+(90:0.05)$) [] ;
		\draw[-] (C4) -- (C1); 
		\draw[densely dotted,-] (C3) -- (C4); 
		\draw[densely dotted,-] (C2) -- (C3); 
		\draw[densely dotted,-] (C1) -- (C2);
		\draw[densely dotted,-] (D3) -- (D4);
		\draw [densely dotted,-,bend right] (C1) to (C2);
		\draw [densely dotted,-,bend right] (C2) to (C3);
		\draw [densely dotted,-,bend right] (C3) to (C4);
		\draw [-,bend right] (C4) to (C1);
		\end{tikzpicture} 
		&
		\begin{tikzpicture} [extended line/.style={shorten >=-#1,shorten <=-#1}, extended line/.default=1cm,bend angle=45, looseness=1]
		\node[name=s,regular polygon, regular polygon sides=4, minimum size=2cm] at (0,0) {$z_{2,3}$}; 
		\coordinate (C1) at (s.corner 1) [];
		\coordinate (C2) at (s.corner 2) [];
		\coordinate (C3) at (s.corner 3) [];
		\coordinate (C4) at (s.corner 4) [];
		\coordinate (D3) at ($(C3)+(90:0.05)$) [] ;
		\coordinate (D4) at ($(C4)+(90:0.05)$) [] ;
		\draw[densely dotted,-] (C4) -- (C1); 
		\draw[densely dotted,-] (C3) -- (C4); 
		\draw[densely dotted,-] (C2) -- (C3); 
		\draw[densely dotted,-] (C1) -- (C2);
		\draw[-] (D3) -- (D4);
		\draw [densely dotted,-,bend right] (C1) to (C2);
		\draw [densely dotted,-,bend right] (C2) to (C3);
		\draw [-,bend right] (D3) to (D4);
		\draw [densely dotted,-,bend right] (C4) to (C1);
		\end{tikzpicture} 
		\end{array}
		\]
		
		\[
		\begin{tikzpicture} [bend angle=45, looseness=1]
		\node[name=s,regular polygon, regular polygon sides=4, minimum size=2cm] at (0,0) {$z_{3,0}$}; 
		\coordinate (C1) at (s.corner 1) [];
		\coordinate (C2) at (s.corner 2) [];
		\coordinate (C3) at (s.corner 3) [];
		\coordinate (C4) at (s.corner 4) [];
		\coordinate (D3) at ($(C3)+(90:0.05)$) [] ;
		\coordinate (D4) at ($(C4)+(90:0.05)$) [] ;
		\draw[densely dotted,-] (C4) -- (C1); 
		\draw[densely dotted,-] (C3) -- (C4); 
		\draw[densely dotted,-] (C2) -- (C3); 
		\draw[densely dotted,-] (C1) -- (C2);
		\draw[densely dotted,-] (D3) -- (D4);
		\draw [-,bend right] (C1) to (C2);
		\draw [-,bend right] (C2) to (C3);
		\draw [-,bend right] (C3) to (C4);
		\draw [-,bend right] (C4) to (C1);
		\end{tikzpicture}\]
	\end{example}
	
	\begin{example}\label{gen 2}
		The quiver of the reconstruction algebra associated to $\frac{1}{165}(1,104)$ is given in Example \ref{rec2}. The set $S$ is
		
		
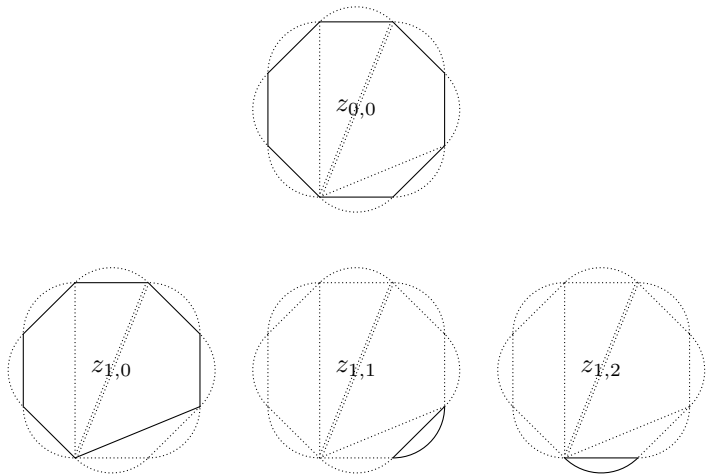
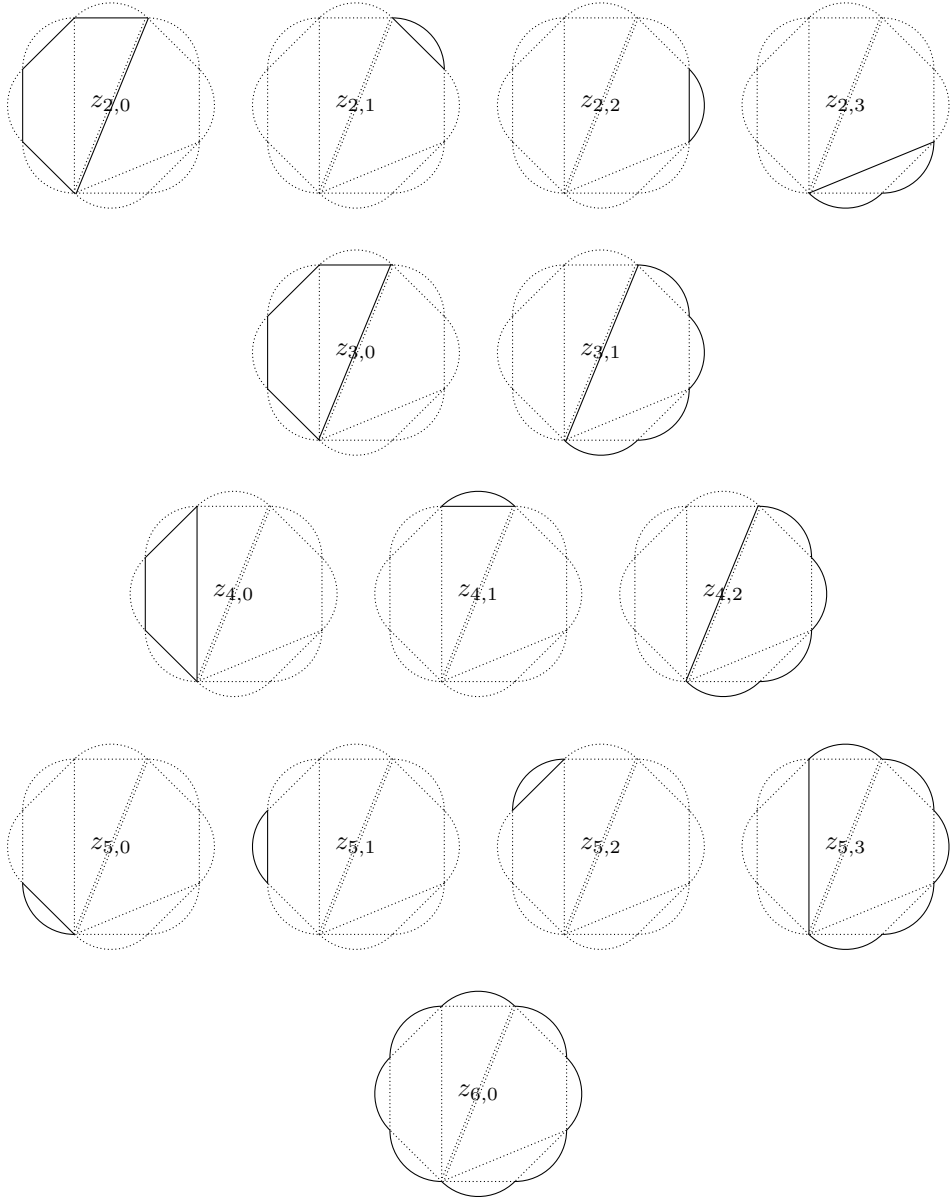
\[
		\begin{tikzpicture} [bend angle=45, looseness=1]
		\node[name=s,regular polygon, regular polygon sides=8, minimum size=2.5cm] at (0,0) {$z_{0,0}$}; 
		\coordinate (C1) at (s.corner 1)  {};
		\coordinate (k2) at ($(C1)+(-100:0.05)$)  {};
		\coordinate (k3) at ($(C1)+(-140:0.05)$)  {};
		\coordinate (C2) at (s.corner 2)  {};
		\coordinate (C3) at (s.corner 3)  {};
		\coordinate (C4) at (s.corner 4)  {};
		\coordinate (C5) at (s.corner 5)  {};
		\coordinate (k3a) at ($(s.corner 5)+(80:0.05)$)  {};
		\coordinate (k2a) at ($(s.corner 5)+(40:0.05)$)  {};
		\coordinate (C6) at (s.corner 6)  {};
		\coordinate (C7) at (s.corner 7)  {};
		\coordinate (C8) at (s.corner 8)  {};    
		\draw[-] (C4) -- (C3); 
		\draw[-] (C3) -- (C2);
		\draw[-] (C2) -- (C1);
		\draw[-] (C1) -- (C8);
		\draw[-] (C8) -- (C7);
		\draw[-] (C7) -- (C6);
		\draw[-] (C6) -- (C5);
		\draw[-] (C5) -- (C4);
		\draw[densely dotted,-] (C2) -- (C5);
		\draw[densely dotted,-] (C7) -- (C5);
		\draw[densely dotted,-] (k3) -- (k3a);
		\draw[densely dotted,-] (k2) -- (k2a) ;
		\draw [densely dotted,-,bend right] (C1) to  (C2);
		\draw [densely dotted,-,bend right] (C2) to  (C3);
		\draw [densely dotted,-,bend right] (C3) to  (C4);
		\draw [densely dotted,-,bend right] (C4) to  (C5);
		\draw [densely dotted,-,bend right] (C5) to  (C6);
		\draw [densely dotted,-,bend right] (C6) to  (C7);
		\draw [densely dotted,-,bend right] (C7) to  (C8);
		\draw [densely dotted,-,bend right] (C8) to  (C1);
		\end{tikzpicture} \]
		
		\[
		\begin{array}{ccc}
		\begin{tikzpicture} [bend angle=45, looseness=1]
		\node[name=s,regular polygon, regular polygon sides=8, minimum size=2.5cm] at (0,0) {$z_{1,0}$}; 
		\coordinate (C1) at (s.corner 1)  {};
		\coordinate (k2) at ($(C1)+(-100:0.05)$)  {};
		\coordinate (k3) at ($(C1)+(-140:0.05)$)  {};
		\coordinate (C2) at (s.corner 2)  {};
		\coordinate (C3) at (s.corner 3)  {};
		\coordinate (C4) at (s.corner 4)  {};
		\coordinate (C5) at (s.corner 5)  {};
		\coordinate (k3a) at ($(s.corner 5)+(80:0.05)$)  {};
		\coordinate (k2a) at ($(s.corner 5)+(40:0.05)$)  {};
		\coordinate (C6) at (s.corner 6)  {};
		\coordinate (C7) at (s.corner 7)  {};
		\coordinate (C8) at (s.corner 8)  {};     
		\draw[-] (C4) -- (C3); 
		\draw[-] (C3) -- (C2);
		\draw[-] (C2) -- (C1);
		\draw[-] (C1) -- (C8);
		\draw[-] (C8) -- (C7);
		\draw[densely dotted,-] (C7) -- (C6);
		\draw[densely dotted,-] (C6) -- (C5);
		\draw[-] (C5) -- (C4);
		\draw[densely dotted,-] (C2) -- (C5);
		\draw[-] (C7) -- (C5);
		\draw[densely dotted,-] (k3) -- (k3a);
		\draw[densely dotted,-] (k2) -- (k2a) ;
		\draw [densely dotted,-,bend right] (C1) to  (C2);
		\draw [densely dotted,-,bend right] (C2) to  (C3);
		\draw [densely dotted,-,bend right] (C3) to  (C4);
		\draw [densely dotted,-,bend right] (C4) to  (C5);
		\draw [densely dotted,-,bend right] (C5) to  (C6);
		\draw [densely dotted,-,bend right] (C6) to  (C7);
		\draw [densely dotted,-,bend right] (C7) to  (C8);
		\draw [densely dotted,-,bend right] (C8) to  (C1);
		\end{tikzpicture}
		&
		\begin{tikzpicture} [bend angle=45, looseness=1]
		\node[name=s,regular polygon, regular polygon sides=8, minimum size=2.5cm] at (0,0) {$z_{1,1}$}; 
		\coordinate (C1) at (s.corner 1)  {};
		\coordinate (k2) at ($(C1)+(-100:0.05)$)  {};
		\coordinate (k3) at ($(C1)+(-140:0.05)$)  {};
		\coordinate (C2) at (s.corner 2)  {};
		\coordinate (C3) at (s.corner 3)  {};
		\coordinate (C4) at (s.corner 4)  {};
		\coordinate (C5) at (s.corner 5)  {};
		\coordinate (k3a) at ($(s.corner 5)+(80:0.05)$)  {};
		\coordinate (k2a) at ($(s.corner 5)+(40:0.05)$)  {};
		\coordinate (C6) at (s.corner 6)  {};
		\coordinate (C7) at (s.corner 7)  {};
		\coordinate (C8) at (s.corner 8)  {};    
		\draw[densely dotted,-] (C4) -- (C3); 
		\draw[densely dotted,-] (C3) -- (C2);
		\draw[densely dotted,-] (C2) -- (C1);
		\draw[densely dotted,-] (C1) -- (C8);
		\draw[densely dotted,-] (C8) -- (C7);
		\draw[-] (C7) -- (C6);
		\draw[densely dotted,-] (C6) -- (C5);
		\draw[densely dotted,-] (C5) -- (C4);
		\draw[densely dotted,-] (C2) -- (C5);
		\draw[densely dotted,-] (C7) -- (C5);
		\draw[densely dotted,-] (k3) -- (k3a);
		\draw[densely dotted,-] (k2) -- (k2a) ;
		\draw [densely dotted,-,bend right] (C1) to  (C2);
		\draw [densely dotted,-,bend right] (C2) to  (C3);
		\draw [densely dotted,-,bend right] (C3) to  (C4);
		\draw [densely dotted,-,bend right] (C4) to  (C5);
		\draw [densely dotted,-,bend right] (C5) to  (C6);
		\draw [-,bend right] (C6) to  (C7);
		\draw [densely dotted,-,bend right] (C7) to  (C8);
		\draw [densely dotted,-,bend right] (C8) to  (C1);
		\end{tikzpicture}
		&
		\begin{tikzpicture} [bend angle=45, looseness=1]
		\node[name=s,regular polygon, regular polygon sides=8, minimum size=2.5cm] at (0,0) {$z_{1,2}$}; 
		\coordinate (C1) at (s.corner 1)  {};
		\coordinate (k2) at ($(C1)+(-100:0.05)$)  {};
		\coordinate (k3) at ($(C1)+(-140:0.05)$)  {};
		\coordinate (C2) at (s.corner 2)  {};
		\coordinate (C3) at (s.corner 3)  {};
		\coordinate (C4) at (s.corner 4)  {};
		\coordinate (C5) at (s.corner 5)  {};
		\coordinate (k3a) at ($(s.corner 5)+(80:0.05)$)  {};
		\coordinate (k2a) at ($(s.corner 5)+(40:0.05)$)  {};
		\coordinate (C6) at (s.corner 6)  {};
		\coordinate (C7) at (s.corner 7)  {};
		\coordinate (C8) at (s.corner 8)  {};    
		\draw[densely dotted,-] (C4) -- (C3); 
		\draw[densely dotted,-] (C3) -- (C2);
		\draw[densely dotted,-] (C2) -- (C1);
		\draw[densely dotted,-] (C1) -- (C8);
		\draw[densely dotted,-] (C8) -- (C7);
		\draw[densely dotted,-] (C7) -- (C6);
		\draw[-] (C6) -- (C5);
		\draw[densely dotted,-] (C5) -- (C4);
		\draw[densely dotted,-] (C2) -- (C5);
		\draw[densely dotted,-] (C7) -- (C5);
		\draw[densely dotted,-] (k3) -- (k3a);
		\draw[densely dotted,-] (k2) -- (k2a) ;
		\draw [densely dotted,-,bend right] (C1) to  (C2);
		\draw [densely dotted,-,bend right] (C2) to  (C3);
		\draw [densely dotted,-,bend right] (C3) to  (C4);
		\draw [densely dotted,-,bend right] (C4) to  (C5);
		\draw [-,bend right] (C5) to  (C6);
		\draw [densely dotted,-,bend right] (C6) to  (C7);
		\draw [densely dotted,-,bend right] (C7) to  (C8);
		\draw [densely dotted,-,bend right] (C8) to  (C1);
		\end{tikzpicture}
		\end{array}
		\]
		
		\[
		\begin{array}{cccc}
		\begin{tikzpicture} [extended line/.style={shorten >=-#1,shorten <=-#1}, extended line/.default=1cm,bend angle=45, looseness=1]
		\node[name=s,regular polygon, regular polygon sides=8, minimum size=2.5cm] at (0,0) {$z_{2,0}$}; 
		\coordinate (C1) at (s.corner 1)  {};
		\coordinate (k2) at ($(C1)+(-100:0.05)$)  {};
		\coordinate (k3) at ($(C1)+(-140:0.05)$)  {};
		\coordinate (C2) at (s.corner 2)  {};
		\coordinate (C3) at (s.corner 3)  {};
		\coordinate (C4) at (s.corner 4)  {};
		\coordinate (C5) at (s.corner 5)  {};
		\coordinate (k3a) at ($(s.corner 5)+(80:0.05)$)  {};
		\coordinate (k2a) at ($(s.corner 5)+(40:0.05)$)  {};
		\coordinate (C6) at (s.corner 6)  {};
		\coordinate (C7) at (s.corner 7)  {};
		\coordinate (C8) at (s.corner 8)  {};     
		\draw[-] (C4) -- (C3); 
		\draw[-] (C3) -- (C2);
		\draw[-] (C2) -- (C1);
		\draw[densely dotted,-] (C1) -- (C8);
		\draw[densely dotted,-] (C8) -- (C7);
		\draw[densely dotted,-] (C7) -- (C6);
		\draw[densely dotted,-] (C6) -- (C5);
		\draw[-] (C5) -- (C4);
		\draw[densely dotted,-] (C2) -- (C5);
		\draw[densely dotted,-] (C7) -- (C5);
		\draw[densely dotted,-] (k3) -- (k3a);
		\draw[extended line=0.05cm,-] (k2) -- (k2a) ;
		\draw [densely dotted,-,bend right] (C1) to  (C2);
		\draw [densely dotted,-,bend right] (C2) to  (C3);
		\draw [densely dotted,-,bend right] (C3) to  (C4);
		\draw [densely dotted,-,bend right] (C4) to  (C5);
		\draw [densely dotted,-,bend right] (C5) to  (C6);
		\draw [densely dotted,-,bend right] (C6) to  (C7);
		\draw [densely dotted,-,bend right] (C7) to  (C8);
		\draw [densely dotted,-,bend right] (C8) to  (C1);
		\end{tikzpicture}
		&
		\begin{tikzpicture} [bend angle=45, looseness=1]
		\node[name=s,regular polygon, regular polygon sides=8, minimum size=2.5cm] at (0,0) {$z_{2,1}$}; 
		\coordinate (C1) at (s.corner 1)  {};
		\coordinate (k2) at ($(C1)+(-100:0.05)$)  {};
		\coordinate (k3) at ($(C1)+(-140:0.05)$)  {};
		\coordinate (C2) at (s.corner 2)  {};
		\coordinate (C3) at (s.corner 3)  {};
		\coordinate (C4) at (s.corner 4)  {};
		\coordinate (C5) at (s.corner 5)  {};
		\coordinate (k3a) at ($(s.corner 5)+(80:0.05)$)  {};
		\coordinate (k2a) at ($(s.corner 5)+(40:0.05)$)  {};
		\coordinate (C6) at (s.corner 6)  {};
		\coordinate (C7) at (s.corner 7)  {};
		\coordinate (C8) at (s.corner 8)  {};    
		\draw[densely dotted,-] (C4) -- (C3); 
		\draw[densely dotted,-] (C3) -- (C2);
		\draw[densely dotted,-] (C2) -- (C1);
		\draw[-] (C1) -- (C8);
		\draw[densely dotted,-] (C8) -- (C7);
		\draw[densely dotted,-] (C7) -- (C6);
		\draw[densely dotted,-] (C6) -- (C5);
		\draw[densely dotted,-] (C5) -- (C4);
		\draw[densely dotted,-] (C2) -- (C5);
		\draw[densely dotted,-] (C7) -- (C5);
		\draw[densely dotted,-] (k3) -- (k3a);
		\draw[densely dotted,-] (k2) -- (k2a) ;
		\draw [densely dotted,-,bend right] (C1) to  (C2);
		\draw [densely dotted,-,bend right] (C2) to  (C3);
		\draw [densely dotted,-,bend right] (C3) to  (C4);
		\draw [densely dotted,-,bend right] (C4) to  (C5);
		\draw [densely dotted,-,bend right] (C5) to  (C6);
		\draw [densely dotted,-,bend right] (C6) to  (C7);
		\draw [densely dotted,-,bend right] (C7) to  (C8);
		\draw [-,bend right] (C8) to  (C1);
		\end{tikzpicture}
		&
		\begin{tikzpicture} [bend angle=45, looseness=1]
		\node[name=s,regular polygon, regular polygon sides=8, minimum size=2.5cm] at (0,0) {$z_{2,2}$}; 
		\coordinate (C1) at (s.corner 1)  {};
		\coordinate (k2) at ($(C1)+(-100:0.05)$)  {};
		\coordinate (k3) at ($(C1)+(-140:0.05)$)  {};
		\coordinate (C2) at (s.corner 2)  {};
		\coordinate (C3) at (s.corner 3)  {};
		\coordinate (C4) at (s.corner 4)  {};
		\coordinate (C5) at (s.corner 5)  {};
		\coordinate (k3a) at ($(s.corner 5)+(80:0.05)$)  {};
		\coordinate (k2a) at ($(s.corner 5)+(40:0.05)$)  {};
		\coordinate (C6) at (s.corner 6)  {};
		\coordinate (C7) at (s.corner 7)  {};
		\coordinate (C8) at (s.corner 8)  {};    
		\draw[densely dotted,-] (C4) -- (C3); 
		\draw[densely dotted,-] (C3) -- (C2);
		\draw[densely dotted,-] (C2) -- (C1);
		\draw[densely dotted,-] (C1) -- (C8);
		\draw[-] (C8) -- (C7);
		\draw[densely dotted,-] (C7) -- (C6);
		\draw[densely dotted,-] (C6) -- (C5);
		\draw[densely dotted,-] (C5) -- (C4);
		\draw[densely dotted,-] (C2) -- (C5);
		\draw[densely dotted,-] (C7) -- (C5);
		\draw[densely dotted,-] (k3) -- (k3a);
		\draw[densely dotted,-] (k2) -- (k2a) ;
		\draw [densely dotted,-,bend right] (C1) to  (C2);
		\draw [densely dotted,-,bend right] (C2) to  (C3);
		\draw [densely dotted,-,bend right] (C3) to  (C4);
		\draw [densely dotted,-,bend right] (C4) to  (C5);
		\draw [densely dotted,-,bend right] (C5) to  (C6);
		\draw [densely dotted,-,bend right] (C6) to  (C7);
		\draw [-,bend right] (C7) to  (C8);
		\draw [densely dotted,-,bend right] (C8) to  (C1);
		\end{tikzpicture}
		&
		\begin{tikzpicture} [bend angle=45, looseness=1]
		\node[name=s,regular polygon, regular polygon sides=8, minimum size=2.5cm] at (0,0) {$z_{2,3}$}; 
		\coordinate (C1) at (s.corner 1)  {};
		\coordinate (k2) at ($(C1)+(-100:0.05)$)  {};
		\coordinate (k3) at ($(C1)+(-140:0.05)$)  {};
		\coordinate (C2) at (s.corner 2)  {};
		\coordinate (C3) at (s.corner 3)  {};
		\coordinate (C4) at (s.corner 4)  {};
		\coordinate (C5) at (s.corner 5)  {};
		\coordinate (k3a) at ($(s.corner 5)+(80:0.05)$)  {};
		\coordinate (k2a) at ($(s.corner 5)+(40:0.05)$)  {};
		\coordinate (C6) at (s.corner 6)  {};
		\coordinate (C7) at (s.corner 7)  {};
		\coordinate (C8) at (s.corner 8)  {};    
		\draw[densely dotted,-] (C4) -- (C3); 
		\draw[densely dotted,-] (C3) -- (C2);
		\draw[densely dotted,-] (C2) -- (C1);
		\draw[densely dotted,-] (C1) -- (C8);
		\draw[densely dotted,-] (C8) -- (C7);
		\draw[densely dotted,-] (C7) -- (C6);
		\draw[densely dotted,-] (C6) -- (C5);
		\draw[densely dotted,-] (C5) -- (C4);
		\draw[densely dotted,-] (C2) -- (C5);
		\draw[-] (C7) -- (C5);
		\draw[densely dotted,-] (k3) -- (k3a);
		\draw[densely dotted,-] (k2) -- (k2a) ;
		\draw [densely dotted,-,bend right] (C1) to  (C2);
		\draw [densely dotted,-,bend right] (C2) to  (C3);
		\draw [densely dotted,-,bend right] (C3) to  (C4);
		\draw [densely dotted,-,bend right] (C4) to  (C5);
		\draw [-,bend right] (C5) to  (C6);
		\draw [-,bend right] (C6) to  (C7);
		\draw [densely dotted,-,bend right] (C7) to  (C8);
		\draw [densely dotted,-,bend right] (C8) to  (C1);
		\end{tikzpicture}
		\end{array}
		\]

		\[
		\begin{array}{cc}
		\begin{tikzpicture} [extended line/.style={shorten >=-#1,shorten <=-#1}, extended line/.default=1cm,bend angle=45, looseness=1]
		\node[name=s,regular polygon, regular polygon sides=8, minimum size=2.5cm] at (0,0) {$z_{3,0}$}; 
		\coordinate (C1) at (s.corner 1)  {};
		\coordinate (k2) at ($(C1)+(-100:0.05)$)  {};
		\coordinate (k3) at ($(C1)+(-140:0.05)$)  {};
		\coordinate (C2) at (s.corner 2)  {};
		\coordinate (C3) at (s.corner 3)  {};
		\coordinate (C4) at (s.corner 4)  {};
		\coordinate (C5) at (s.corner 5)  {};
		\coordinate (k3a) at ($(s.corner 5)+(80:0.05)$)  {};
		\coordinate (k2a) at ($(s.corner 5)+(40:0.05)$)  {};
		\coordinate (C6) at (s.corner 6)  {};
		\coordinate (C7) at (s.corner 7)  {};
		\coordinate (C8) at (s.corner 8)  {};     
		\draw[-] (C4) -- (C3); 
		\draw[-] (C3) -- (C2);
		\draw[-] (C2) -- (C1);
		\draw[densely dotted,-] (C1) -- (C8);
		\draw[densely dotted,-] (C8) -- (C7);
		\draw[densely dotted,-] (C7) -- (C6);
		\draw[densely dotted,-] (C6) -- (C5);
		\draw[-] (C5) -- (C4);
		\draw[densely dotted,-] (C2) -- (C5);
		\draw[densely dotted,-] (C7) -- (C5);
		\draw[extended line=0.05cm,-] (k3) -- (k3a);
		\draw[densely dotted,-] (k2) -- (k2a) ;
		\draw [densely dotted,-,bend right] (C1) to  (C2);
		\draw [densely dotted,-,bend right] (C2) to  (C3);
		\draw [densely dotted,-,bend right] (C3) to  (C4);
		\draw [densely dotted,-,bend right] (C4) to  (C5);
		\draw [densely dotted,-,bend right] (C5) to  (C6);
		\draw [densely dotted,-,bend right] (C6) to  (C7);
		\draw [densely dotted,-,bend right] (C7) to  (C8);
		\draw [densely dotted,-,bend right] (C8) to  (C1);
		\end{tikzpicture}
		&
		\begin{tikzpicture} [extended line/.style={shorten >=-#1,shorten <=-#1}, extended line/.default=1cm,bend angle=45, looseness=1]
		\node[name=s,regular polygon, regular polygon sides=8, minimum size=2.5cm] at (0,0) {$z_{3,1}$}; 
		\coordinate (C1) at (s.corner 1)  {};
		\coordinate (k2) at ($(C1)+(-100:0.05)$)  {};
		\coordinate (k3) at ($(C1)+(-140:0.05)$)  {};
		\coordinate (C2) at (s.corner 2)  {};
		\coordinate (C3) at (s.corner 3)  {};
		\coordinate (C4) at (s.corner 4)  {};
		\coordinate (C5) at (s.corner 5)  {};
		\coordinate (k3a) at ($(s.corner 5)+(80:0.05)$)  {};
		\coordinate (k2a) at ($(s.corner 5)+(40:0.05)$)  {};
		\coordinate (C6) at (s.corner 6)  {};
		\coordinate (C7) at (s.corner 7)  {};
		\coordinate (C8) at (s.corner 8)  {};     
		\draw[densely dotted,-] (C4) -- (C3); 
		\draw[densely dotted,-] (C3) -- (C2);
		\draw[densely dotted,-] (C2) -- (C1);
		\draw[densely dotted,-] (C1) -- (C8);
		\draw[densely dotted,-] (C8) -- (C7);
		\draw[densely dotted,-] (C7) -- (C6);
		\draw[densely dotted,-] (C6) -- (C5);
		\draw[densely dotted,-] (C5) -- (C4);
		\draw[densely dotted,-] (C2) -- (C5);
		\draw[densely dotted,-] (C7) -- (C5);
		\draw[densely dotted,-] (k3) -- (k3a);
		\draw[extended line=0.05cm,-] (k2) -- (k2a) ;
		\draw [densely dotted,-,bend right] (C1) to  (C2);
		\draw [densely dotted,-,bend right] (C2) to  (C3);
		\draw [densely dotted,-,bend right] (C3) to  (C4);
		\draw [densely dotted,-,bend right] (C4) to  (C5);
		\draw [-,bend right] (C5) to  (C6);
		\draw [-,bend right] (C6) to  (C7);
		\draw [-,bend right] (C7) to  (C8);
		\draw [-,bend right] (C8) to  (C1);
		\end{tikzpicture}
		\end{array}
		\]
		
		\[
		\begin{array}{ccc}
		\begin{tikzpicture} [bend angle=45, looseness=1]
		\node[name=s,regular polygon, regular polygon sides=8, minimum size=2.5cm] at (0,0) {$z_{4,0}$}; 
		\coordinate (C1) at (s.corner 1)  {};
		\coordinate (k2) at ($(C1)+(-100:0.05)$)  {};
		\coordinate (k3) at ($(C1)+(-140:0.05)$)  {};
		\coordinate (C2) at (s.corner 2)  {};
		\coordinate (C3) at (s.corner 3)  {};
		\coordinate (C4) at (s.corner 4)  {};
		\coordinate (C5) at (s.corner 5)  {};
		\coordinate (k3a) at ($(s.corner 5)+(80:0.05)$)  {};
		\coordinate (k2a) at ($(s.corner 5)+(40:0.05)$)  {};
		\coordinate (C6) at (s.corner 6)  {};
		\coordinate (C7) at (s.corner 7)  {};
		\coordinate (C8) at (s.corner 8)  {};    
		\draw[-] (C4) -- (C3); 
		\draw[-] (C3) -- (C2);
		\draw[densely dotted,-] (C2) -- (C1);
		\draw[densely dotted,-] (C1) -- (C8);
		\draw[densely dotted,-] (C8) -- (C7);
		\draw[densely dotted,-] (C7) -- (C6);
		\draw[densely dotted,-] (C6) -- (C5);
		\draw[-] (C5) -- (C4);
		\draw[-] (C2) -- (C5);
		\draw[densely dotted,-] (C7) -- (C5);
		\draw[densely dotted,-] (k3) -- (k3a);
		\draw[densely dotted,-] (k2) -- (k2a) ;
		\draw [densely dotted,-,bend right] (C1) to  (C2);
		\draw [densely dotted,-,bend right] (C2) to  (C3);
		\draw [densely dotted,-,bend right] (C3) to  (C4);
		\draw [densely dotted,-,bend right] (C4) to  (C5);
		\draw [densely dotted,-,bend right] (C5) to  (C6);
		\draw [densely dotted,-,bend right] (C6) to  (C7);
		\draw [densely dotted,-,bend right] (C7) to  (C8);
		\draw [densely dotted,-,bend right] (C8) to  (C1);
		\end{tikzpicture}
		&
		\begin{tikzpicture} [bend angle=45, looseness=1]
		\node[name=s,regular polygon, regular polygon sides=8, minimum size=2.5cm] at (0,0) {$z_{4,1}$}; 
		\coordinate (C1) at (s.corner 1)  {};
		\coordinate (k2) at ($(C1)+(-100:0.05)$)  {};
		\coordinate (k3) at ($(C1)+(-140:0.05)$)  {};
		\coordinate (C2) at (s.corner 2)  {};
		\coordinate (C3) at (s.corner 3)  {};
		\coordinate (C4) at (s.corner 4)  {};
		\coordinate (C5) at (s.corner 5)  {};
		\coordinate (k3a) at ($(s.corner 5)+(80:0.05)$)  {};
		\coordinate (k2a) at ($(s.corner 5)+(40:0.05)$)  {};
		\coordinate (C6) at (s.corner 6)  {};
		\coordinate (C7) at (s.corner 7)  {};
		\coordinate (C8) at (s.corner 8)  {};    
		\draw[densely dotted,-] (C4) -- (C3); 
		\draw[densely dotted,-] (C3) -- (C2);
		\draw[-] (C2) -- (C1);
		\draw[densely dotted,-] (C1) -- (C8);
		\draw[densely dotted,-] (C8) -- (C7);
		\draw[densely dotted,-] (C7) -- (C6);
		\draw[densely dotted,-] (C6) -- (C5);
		\draw[densely dotted,-] (C5) -- (C4);
		\draw[densely dotted,-] (C2) -- (C5);
		\draw[densely dotted,-] (C7) -- (C5);
		\draw[densely dotted,-] (k3) -- (k3a);
		\draw[densely dotted,-] (k2) -- (k2a) ;
		\draw [-,bend right] (C1) to  (C2);
		\draw [densely dotted,-,bend right] (C2) to  (C3);
		\draw [densely dotted,-,bend right] (C3) to  (C4);
		\draw [densely dotted,-,bend right] (C4) to  (C5);
		\draw [densely dotted,-,bend right] (C5) to  (C6);
		\draw [densely dotted,-,bend right] (C6) to  (C7);
		\draw [densely dotted,-,bend right] (C7) to  (C8);
		\draw [densely dotted,-,bend right] (C8) to  (C1);
		\end{tikzpicture}
		&
		\begin{tikzpicture} [extended line/.style={shorten >=-#1,shorten <=-#1}, extended line/.default=1cm,bend angle=45, looseness=1]
		\node[name=s,regular polygon, regular polygon sides=8, minimum size=2.5cm] at (0,0) {$z_{4,2}$}; 
		\coordinate (C1) at (s.corner 1)  {};
		\coordinate (k2) at ($(C1)+(-100:0.05)$)  {};
		\coordinate (k3) at ($(C1)+(-140:0.05)$)  {};
		\coordinate (C2) at (s.corner 2)  {};
		\coordinate (C3) at (s.corner 3)  {};
		\coordinate (C4) at (s.corner 4)  {};
		\coordinate (C5) at (s.corner 5)  {};
		\coordinate (k3a) at ($(s.corner 5)+(80:0.05)$)  {};
		\coordinate (k2a) at ($(s.corner 5)+(40:0.05)$)  {};
		\coordinate (C6) at (s.corner 6)  {};
		\coordinate (C7) at (s.corner 7)  {};
		\coordinate (C8) at (s.corner 8)  {};     
		\draw[densely dotted,-] (C4) -- (C3); 
		\draw[densely dotted,-] (C3) -- (C2);
		\draw[densely dotted,-] (C2) -- (C1);
		\draw[densely dotted,-] (C1) -- (C8);
		\draw[densely dotted,-] (C8) -- (C7);
		\draw[densely dotted,-] (C7) -- (C6);
		\draw[densely dotted,-] (C6) -- (C5);
		\draw[densely dotted,-] (C5) -- (C4);
		\draw[densely dotted,-] (C2) -- (C5);
		\draw[densely dotted,-] (C7) -- (C5);
		\draw[extended line=0.05cm,-] (k3) -- (k3a);
		\draw[densely dotted,-] (k2) -- (k2a) ;
		\draw [densely dotted,-,bend right] (C1) to  (C2);
		\draw [densely dotted,-,bend right] (C2) to  (C3);
		\draw [densely dotted,-,bend right] (C3) to  (C4);
		\draw [densely dotted,-,bend right] (C4) to  (C5);
		\draw [-,bend right] (C5) to  (C6);
		\draw [-,bend right] (C6) to  (C7);
		\draw [-,bend right] (C7) to  (C8);
		\draw [-,bend right] (C8) to  (C1);
		\end{tikzpicture}
		\end{array}
		\]
		
		\[
		\begin{array}{cccc}
		\begin{tikzpicture} [bend angle=45, looseness=1]
		\node[name=s,regular polygon, regular polygon sides=8, minimum size=2.5cm] at (0,0) {$z_{5,0}$}; 
		\coordinate (C1) at (s.corner 1)  {};
		\coordinate (k2) at ($(C1)+(-100:0.05)$)  {};
		\coordinate (k3) at ($(C1)+(-140:0.05)$)  {};
		\coordinate (C2) at (s.corner 2)  {};
		\coordinate (C3) at (s.corner 3)  {};
		\coordinate (C4) at (s.corner 4)  {};
		\coordinate (C5) at (s.corner 5)  {};
		\coordinate (k3a) at ($(s.corner 5)+(80:0.05)$)  {};
		\coordinate (k2a) at ($(s.corner 5)+(40:0.05)$)  {};
		\coordinate (C6) at (s.corner 6)  {};
		\coordinate (C7) at (s.corner 7)  {};
		\coordinate (C8) at (s.corner 8)  {};    
		\draw[densely dotted,-] (C4) -- (C3); 
		\draw[densely dotted,-] (C3) -- (C2);
		\draw[densely dotted,-] (C2) -- (C1);
		\draw[densely dotted,-] (C1) -- (C8);
		\draw[densely dotted,-] (C8) -- (C7);
		\draw[densely dotted,-] (C7) -- (C6);
		\draw[densely dotted,-] (C6) -- (C5);
		\draw[-] (C5) -- (C4);
		\draw[densely dotted,-] (C2) -- (C5);
		\draw[densely dotted,-] (C7) -- (C5);
		\draw[densely dotted,-] (k3) -- (k3a);
		\draw[densely dotted,-] (k2) -- (k2a) ;
		\draw [densely dotted,-,bend right] (C1) to  (C2);
		\draw [densely dotted,-,bend right] (C2) to  (C3);
		\draw [densely dotted,-,bend right] (C3) to  (C4);
		\draw [-,bend right] (C4) to  (C5);
		\draw [densely dotted,-,bend right] (C5) to  (C6);
		\draw [densely dotted,-,bend right] (C6) to  (C7);
		\draw [densely dotted,-,bend right] (C7) to  (C8);
		\draw [densely dotted,-,bend right] (C8) to  (C1);
		\end{tikzpicture}
		&
		\begin{tikzpicture} [bend angle=45, looseness=1]
		\node[name=s,regular polygon, regular polygon sides=8, minimum size=2.5cm] at (0,0) {$z_{5,1}$}; 
		\coordinate (C1) at (s.corner 1)  {};
		\coordinate (k2) at ($(C1)+(-100:0.05)$)  {};
		\coordinate (k3) at ($(C1)+(-140:0.05)$)  {};
		\coordinate (C2) at (s.corner 2)  {};
		\coordinate (C3) at (s.corner 3)  {};
		\coordinate (C4) at (s.corner 4)  {};
		\coordinate (C5) at (s.corner 5)  {};
		\coordinate (k3a) at ($(s.corner 5)+(80:0.05)$)  {};
		\coordinate (k2a) at ($(s.corner 5)+(40:0.05)$)  {};
		\coordinate (C6) at (s.corner 6)  {};
		\coordinate (C7) at (s.corner 7)  {};
		\coordinate (C8) at (s.corner 8)  {};    
		\draw[-] (C4) -- (C3); 
		\draw[densely dotted,-] (C3) -- (C2);
		\draw[densely dotted,-] (C2) -- (C1);
		\draw[densely dotted,-] (C1) -- (C8);
		\draw[densely dotted,-] (C8) -- (C7);
		\draw[densely dotted,-] (C7) -- (C6);
		\draw[densely dotted,-] (C6) -- (C5);
		\draw[densely dotted,-] (C5) -- (C4);
		\draw[densely dotted,-] (C2) -- (C5);
		\draw[densely dotted,-] (C7) -- (C5);
		\draw[densely dotted,-] (k3) -- (k3a);
		\draw[densely dotted,-] (k2) -- (k2a) ;
		\draw [densely dotted,-,bend right] (C1) to  (C2);
		\draw [densely dotted,-,bend right] (C2) to  (C3);
		\draw [-,bend right] (C3) to  (C4);
		\draw [densely dotted,-,bend right] (C4) to  (C5);
		\draw [densely dotted,-,bend right] (C5) to  (C6);
		\draw [densely dotted,-,bend right] (C6) to  (C7);
		\draw [densely dotted,-,bend right] (C7) to  (C8);
		\draw [densely dotted,-,bend right] (C8) to  (C1);
		\end{tikzpicture}
		&
		\begin{tikzpicture} [bend angle=45, looseness=1]
		\node[name=s,regular polygon, regular polygon sides=8, minimum size=2.5cm] at (0,0) {$z_{5,2}$}; 
		\coordinate (C1) at (s.corner 1)  {};
		\coordinate (k2) at ($(C1)+(-100:0.05)$)  {};
		\coordinate (k3) at ($(C1)+(-140:0.05)$)  {};
		\coordinate (C2) at (s.corner 2)  {};
		\coordinate (C3) at (s.corner 3)  {};
		\coordinate (C4) at (s.corner 4)  {};
		\coordinate (C5) at (s.corner 5)  {};
		\coordinate (k3a) at ($(s.corner 5)+(80:0.05)$)  {};
		\coordinate (k2a) at ($(s.corner 5)+(40:0.05)$)  {};
		\coordinate (C6) at (s.corner 6)  {};
		\coordinate (C7) at (s.corner 7)  {};
		\coordinate (C8) at (s.corner 8)  {};    
		\draw[densely dotted,-] (C4) -- (C3); 
		\draw[-] (C3) -- (C2);
		\draw[densely dotted,-] (C2) -- (C1);
		\draw[densely dotted,-] (C1) -- (C8);
		\draw[densely dotted,-] (C8) -- (C7);
		\draw[densely dotted,-] (C7) -- (C6);
		\draw[densely dotted,-] (C6) -- (C5);
		\draw[densely dotted,-] (C5) -- (C4);
		\draw[densely dotted,-] (C2) -- (C5);
		\draw[densely dotted,-] (C7) -- (C5);
		\draw[densely dotted,-] (k3) -- (k3a);
		\draw[densely dotted,-] (k2) -- (k2a) ;
		\draw [densely dotted,-,bend right] (C1) to  (C2);
		\draw [-,bend right] (C2) to  (C3);
		\draw [densely dotted,-,bend right] (C3) to  (C4);
		\draw [densely dotted,-,bend right] (C4) to  (C5);
		\draw [densely dotted,-,bend right] (C5) to  (C6);
		\draw [densely dotted,-,bend right] (C6) to  (C7);
		\draw [densely dotted,-,bend right] (C7) to  (C8);
		\draw [densely dotted,-,bend right] (C8) to  (C1);
		\end{tikzpicture}
		&
		\begin{tikzpicture} [bend angle=45, looseness=1]
		\node[name=s,regular polygon, regular polygon sides=8, minimum size=2.5cm] at (0,0) {$z_{5,3}$}; 
		\coordinate (C1) at (s.corner 1)  {};
		\coordinate (k2) at ($(C1)+(-100:0.05)$)  {};
		\coordinate (k3) at ($(C1)+(-140:0.05)$)  {};
		\coordinate (C2) at (s.corner 2)  {};
		\coordinate (C3) at (s.corner 3)  {};
		\coordinate (C4) at (s.corner 4)  {};
		\coordinate (C5) at (s.corner 5)  {};
		\coordinate (k3a) at ($(s.corner 5)+(80:0.05)$)  {};
		\coordinate (k2a) at ($(s.corner 5)+(40:0.05)$)  {};
		\coordinate (C6) at (s.corner 6)  {};
		\coordinate (C7) at (s.corner 7)  {};
		\coordinate (C8) at (s.corner 8)  {};    
		\draw[densely dotted,-] (C4) -- (C3); 
		\draw[densely dotted,-] (C3) -- (C2);
		\draw[densely dotted,-] (C2) -- (C1);
		\draw[densely dotted,-] (C1) -- (C8);
		\draw[densely dotted,-] (C8) -- (C7);
		\draw[densely dotted,-] (C7) -- (C6);
		\draw[densely dotted,-] (C6) -- (C5);
		\draw[densely dotted,-] (C5) -- (C4);
		\draw[-] (C2) -- (C5);
		\draw[densely dotted,-] (C7) -- (C5);
		\draw[densely dotted,-] (k3) -- (k3a);
		\draw[densely dotted,-] (k2) -- (k2a) ;
		\draw [-,bend right] (C1) to  (C2);
		\draw [densely dotted,-,bend right] (C2) to  (C3);
		\draw [densely dotted,-,bend right] (C3) to  (C4);
		\draw [densely dotted,-,bend right] (C4) to  (C5);
		\draw [-,bend right] (C5) to  (C6);
		\draw [-,bend right] (C6) to  (C7);
		\draw [-,bend right] (C7) to  (C8);
		\draw [-,bend right] (C8) to  (C1);
		\end{tikzpicture}
		\end{array}
		\]
		
		\[
		\begin{tikzpicture} [bend angle=45, looseness=1]
		\node[name=s,regular polygon, regular polygon sides=8, minimum size=2.5cm] at (0,0) {$z_{6,0}$}; 
		\coordinate (C1) at (s.corner 1)  {};
		\coordinate (k2) at ($(C1)+(-100:0.05)$)  {};
		\coordinate (k3) at ($(C1)+(-140:0.05)$)  {};
		\coordinate (C2) at (s.corner 2)  {};
		\coordinate (C3) at (s.corner 3)  {};
		\coordinate (C4) at (s.corner 4)  {};
		\coordinate (C5) at (s.corner 5)  {};
		\coordinate (k3a) at ($(s.corner 5)+(80:0.05)$)  {};
		\coordinate (k2a) at ($(s.corner 5)+(40:0.05)$)  {};
		\coordinate (C6) at (s.corner 6)  {};
		\coordinate (C7) at (s.corner 7)  {};
		\coordinate (C8) at (s.corner 8)  {};    
		\draw[densely dotted,-] (C4) -- (C3); 
		\draw[densely dotted,-] (C3) -- (C2);
		\draw[densely dotted,-] (C2) -- (C1);
		\draw[densely dotted,-] (C1) -- (C8);
		\draw[densely dotted,-] (C8) -- (C7);
		\draw[densely dotted,-] (C7) -- (C6);
		\draw[densely dotted,-] (C6) -- (C5);
		\draw[densely dotted,-] (C5) -- (C4);
		\draw[densely dotted,-] (C2) -- (C5);
		\draw[densely dotted,-] (C7) -- (C5);
		\draw[densely dotted,-] (k3) -- (k3a);
		\draw[densely dotted,-] (k2) -- (k2a) ;
		\draw [-,bend right] (C1) to  (C2);
		\draw [-,bend right] (C2) to  (C3);
		\draw [-,bend right] (C3) to  (C4);
		\draw [-,bend right] (C4) to  (C5);
		\draw [-,bend right] (C5) to  (C6);
		\draw [-,bend right] (C6) to  (C7);
		\draw [-,bend right] (C7) to  (C8);
		\draw [-,bend right] (C8) to  (C1);
		\end{tikzpicture} \]
		
	\end{example}

		With the above notation set, the proof of Proposition \ref{gen Z} is a relatively simple induction. In what follows, for two paths $p,q \in \mathbb{C}Q,$  we write $p \sim q$ if $p=q$ in $\scrR\colonequals \mathbb{C}[\Rep(\mathbb{C}Q,\updelta)]$ where $\updelta=(1,\hdots,1).$
		\begin{proof}
			By Lemma \ref{lemcyclesgen}, $\scrR^G$ is generated by cycles. Hence consider a cycle $p,$ then the proof is complete if we show that $p$ is generated by elements in $S.$ We induct on the lengths of cycles, since all cycles of length two (the $ac$'s) are already in the generating set.

\smallskip			
For any vertex $v$, consider a non trivial cycle $p$, then it must leave the vertex. According to the quiver, there are three options:

\smallskip
\noindent
\emph{Case 1.} The path $p$ starts with a $k$ arrow ($p = k_tp^{\prime}).$ Since $p$ is a cycle then $p^{\prime}:0 \rightarrow v,$ so we have the following subcases:
		\begin{enumerate}
				\item [(a)] $p^{\prime}$ starts clockwise. If $p^{\prime}$ moves in the clockwise direction indefinitely to vertex $v$ ($p^{\prime}= C_{0v}p^{\prime\prime}$), then $p = k_t C_{0v} p^{\prime\prime} \sim z p^{\prime\prime}$
					and by induction $p \in \left<S\right>$. Hence we can assume that, at some stage $p^{\prime}$ stops travelling clockwise before vertex $v$. At that stage, either we continue anticlockwise so $$p = k_t C_{0w}a_{ww+1} p^{\prime\prime} = k_t C_{0w+1}\underbrace{c_{w+1w}a_{ww+1}}_{z}p^{\prime\prime} \sim z (\text{cycles of length smaller than}~ p),$$
					or we continue via some $k_j$ so 
					$$p = k_t C_{0w} k_j p^{\prime\prime} = k_t \underbrace{C_{0w} k_j}_{z} p^{\prime\prime} \sim z (\text{cycles of length smaller than}~ p).$$ In either case, by induction $p \in \left<S\right>.$				
					\item [(b)] $p^{\prime}$ starts anticlockwise. This subcase is similar to (a), interchanging the clockwise paths and the anticlockwise paths.
				\end{enumerate}
\noindent
\emph{Case 2.} The path $p$ starts with a clockwise arrow, so $p = c_{vv-1}p^{\prime}$. Since $p$ is a cycle then $p^{\prime}\colon v-1 \rightarrow v.$ If $p^{\prime}$ continues clockwise indefinitely, then we can write $p=C_{vv}p^{\prime\prime}$ $\sim$ $z_{0,0}p^{\prime\prime},$ and by induction we are done. Otherwise, at some stage $p^\prime$ stops travelling clockwise and we can write $p=C_{vw}p^\prime$ for some $p^\prime: w \rightarrow v.$ According to the quiver, there are two options.
\begin{enumerate}
	\item [(a)] $p^\prime$ starts with an anticlockwise arrow ($p^\prime = ap^{\prime\prime}$), so then
					\[
					p = C_{vw}a_{ww+1} p^{\prime\prime}= C_{vw+1}\underbrace{(c_{w+1w}a_{ww+1})}_{z} p^{\prime\prime} \sim z (\text{cycles of length smaller than}~ p),\]
					thus by induction, $p \in \left<S \right>.$
					\item [(b)] $p^\prime$ starts with a $k$ arrow ($p^\prime = kp^{\prime\prime}$), and we repeat a similar procedure as in Case 1 applied to $p^{\prime}$.	By induction, $p \in \left< S \right>.$	
				\end{enumerate} 
\noindent
\emph{Case 3.} The path $p$ starts with an anticlockwise arrow. This is very similar to Case 2, after interchanging the clockwise and the anticlockwise arrows. 
\end{proof}

	\section{The Artin Component}\label{ArtinComp}
	This section recovers the Artin component directly from the quiver of the reconstruction algebras, using the representation variety. 
	
	\subsection{QDet and First Properties}\label{Qdet}
	By Riemenschneider duality (see e.g.\ \cite[2.11]{TypeA}), for all $t$ such that $1\leq t\leq m$  there is an equality $\upbeta_t = l_t-l_{t-1}+2$. Set
\[
	s_t= \begin{cases}
	\upbeta_t -1 & ~ \text{if} ~1 \leq t \leq m\\
	0 &  ~   \text{if} ~ t = m+1,
	\end{cases}
\]         
Recalling the notation in $\S$\ref{QDetform}, consider the description of the Artin component of $\frac{1}{r}(1,a)$ due to Riemenschneider \cite{RieCyclic}, which in its \textit{quasideterminantal} form is as follows:
	\vspace*{.25cm}
	\[\scriptsize\begin{pmatrix}
	z_{0,0}& & z_{1,0} && z_{2,0}& \cdots & &z_{m,0}\\
	&z_{1,{s_1}-1}\cdot \ldots \cdot z_{1,1}&& z_{2,{s_2}-1}\cdot \ldots \cdot z_{2,1}&&& z_{m,{s_m}-1}\cdot \ldots \cdot z_{m,1}&\\
	z_{1,s_1} & & z_{2,s_2}&& z_{3,s_3}& \cdots & &z_{m+1,s_{m+1}}
	\end{pmatrix}.\]
	\normalsize
	As in $\S$\ref{QDetform}, $\mathrm{QDet}(\mathsf{z})$ is defined to be the set of all \textit{quasiminors} of the above matrix.
	\begin{example}
		The Artin component of the group $\frac{1}{7}(1,3)$ in Example \ref{gen 1} has quasideterminantal form
		\[ \begin{pmatrix}
		z_{0,0}& & z_{1,0} && z_{2,0}\\
		&&&z_{2,2}z_{2,1} & \\
		z_{1,1} & & z_{2,3}&& z_{3,0} 
		\end{pmatrix},\]
		thus $\mathrm{QDet}(\mathsf{z})$ is the set 
		\[ \{ z_{0,0}z_{2,3} - z_{1,0}z_{1,1},~ z_{0,0}z_{3,0} - z_{2,0}z_{2,1}z_{2,2}z_{1,1},~ z_{1,0}z_{3,0}- z_{2,0}z_{2,1}z_{2,2}z_{2,3}\}.
		\]
	\end{example}
	
	\begin{example}
		The Artin component of the group $\frac{1}{165}(1,104)$ in Example \ref{gen 2} has  quasideterminantal form
		\[ \begin{pmatrix}
		z_{0,0}&&  z_{1,0} && z_{2,0}&z_{3,0} &&z_{4,0}&&z_{5,0}\\
		&z_{1,1}&&z_{2,2}z_{2,1}&&  &z_{4,1}&&z_{5,2}z_{5,1} &\\
		z_{1,2}  && z_{2,3}&& z_{3,1} &z_{4,2}&&z_{5,3}&&z_{6,0}
		\end{pmatrix},\]
		and in this case $\mathrm{QDet}(\mathsf{z})$ consists of 15 relations.
	\end{example}
	For the group $\frac{1}{r}(1,a)$, recall from  $\S$\ref{ReconAlg} that $\scrR^G$ is constructed only from the quiver of the reconstruction algebra. Consider the polynomial ring  $\mathbb{C}[\mathsf{z}]$ which has as variables elements in the set $S$ of Proposition \ref{gen Z}. There is a natural homomorphism
  \[
\mathbb{C}[\mathsf{z}] \xrightarrow{\varphi} \scrR^G, 
\]
defined by sending $z_{i,j}$ to the corresponding cycle in \eqref{eq1}.
\begin{prop}\label{SurjHom}
For any group $\frac{1}{r}(1,a)$, the homomorphism $\varphi \colon \mathbb{C}[\mathsf{z}] \rightarrow \scrR^G $ is surjective, and $\mathrm{QDet}(\mathsf{z})$ belongs to the kernel.
\end{prop}
	
	\begin{proof}
		Surjectivity follows from Proposition \ref{gen Z}. We just need to show that the quasiminors are sent to zero.  An arbitrary quasiminor is determined by
		\begin{enumerate}
			\item[$\bullet$] First choosing $z_{i,0},~0 \leq i \leq m-1.$
			\item[$\bullet$] Then choosing $z_{j,s_j},~i+2 \leq j \leq m+1.$
		\end{enumerate}
		With these choices,
		\begin{align*}
		\varphi(z_{i,0}z_{j,s_j}) &= C_{0l_i}k_i \cdot A_{0l_{i+1}}k_{i+1}\\
		&=C_{0l_{i+1}}(c_{l_{i+1}l_{i+1}-1} \hdots c_{l_i+1l_i})k_i \cdot A_{0l_i}(a_{l_il_i+1} \hdots a_{{l_{i+1}-1}l_{i+1}}) k_{i+1} \tag{since elements in $\mathbb{C}[\mathsf{z}]$ commute}\\
		&= A_{0l_i}k_i \cdot (c_{l_{i+1}l_{i+1}-1} a_{{l_{i+1}-1}l_{i+1}} \hdots  c_{l_i+1l_i}a_{l_il_i+1}) \cdot C_{0l_{i+1}}k_{i+1}\\
		&= A_{0l_i}k_i  \left(\displaystyle\prod_{p = 1}^{l_{i+1}-l_i}c_{l_{i+1}-(p-1)l_{i+1}-p } a_{l_{i+1}-pl_{i+1}-(p-1)}\right)C_{0l_{i+1}}k_{i+1}\\
		&= \varphi \left(z_{i+1,s_{i+1}} \left(\displaystyle\prod_{k = i+1}^{j-1}z_{k,1}\hdots z_{k,s_{k}-1 } \right)z_{j-1,0} \right)\\
		&= \varphi \left(z_{i+1,s_{i+1}} \left(\displaystyle\prod_{k = i+1}^{j-1}z_{k,s_{k}-1}\hdots z_{k,1 } \right)z_{j-1,0} \right).
		\end{align*}
		
		\noindent
		This shows that the quasiminor relation
		\[z_{i,0}z_{j,s_j} = z_{i+1,s_{i+1}} \left(\displaystyle\prod_{k = i+1}^{j-1}z_{k,s_{k}-1}\hdots z_{k,1 } \right)z_{j-1,0}\]
		 belongs to the kernel of $\varphi,$ as required.
	\end{proof}

	The remainder of this section will prove that $\mathrm{QDet}(\mathsf{z})$ generates the kernel, but this involves significant work.

	\subsection{Toric ideals generalities}\label{sec: toric gen}
	To compute the kernel of the homomorphism $\varphi$ in Proposition \ref{SurjHom}, we will rely on its description as a toric ideal of $\mathbb{C}[\mathsf{z}],$ as explained in \cite[\S4]{GrobnerSturmfel}. 
		
	Let $\mathcal{A} = \{a_1, a_2, \ldots , a_n\} \subset \mathbb{Z}^d\backslash \{0\},$ where each $a_i$ is considered as a column vector, and consider the Laurent polynomial ring $k[t^{\pm1}]\colonequals  k[t_1, \ldots , t_d, t_1^{-1}, \ldots , t_d^{-1}]$. Set $A = [a_1 a_2 \ldots a_n] \in \mathbb{Z}^{d \times n}$ to be the corresponding $d \times n$ matrix, and consider the map
	\begin{align*}
	 k[x] &\rightarrow k[t^{\pm1}] \\
	x_i &\mapsto t^{a_i}.
	\end{align*}
	
	The toric ideal of $\mathcal{A}$, denoted by $I_\mathcal{A}$, is by definition the kernel. It is possible to compute this using an elimination method, however this is computationally hard in general. A more efficient algorithm to compute $I_\mathcal{A}$ is given in \cite[Algorithm 12.3]{GrobnerSturmfel}, and proceeds as follows:
	
	\begin{enumerate}
		\item [(1)] Find any lattice spanning set $L$ for $\ker(A)_{\mathbb{Z}}$.
		\item [(2)] Consider the ideal $I_L \colonequals  (x^{{u}^+} - x^{{u}^-}\mid u \in L)$, and compute the saturation of $I_L$, $(I_L : (x_1x_2 \hdots x_n)^{\infty})$	with respect to the indeterminates $x_1,\ldots ,x_n.$ Then
		\[(I_L: (x_1 \hdots x_n)^\infty) = I_\mathcal{A}.\]
		Part $(2)$ is the most difficult step.		
	\end{enumerate} 
	
    \subsection{Step 1: Lattice Spanning Set}This section explains how to view the homomorphism $\varphi\colon \mathbb{C}[\mathsf{z}]\rightarrow \scrR^G$ in the toric language of the previous section, then in Corollary \ref{cor: ker Z 1} computes a lattice spanning set for the kernel.

    \begin{example}\label{ex:1311}
    	For the group $\frac{1}{3}(1,1)$,  the homomorphism $\varphi\colon\mathbb{C}[\mathsf{z}]\twoheadrightarrow \scrR^G$ sends $z_{0,0} \mapsto c_1a_1$, $z_{3,0} \mapsto c_2k_1$, and
    	\[
    	\begin{array}{rclcrclcrcl}
    	z_{1,0} &\mapsto& c_1a_2&&z_{2,0} &\mapsto& c_1k_1\\
    	z_{1,1} &\mapsto& c_2a_1&&z_{2,1} &\mapsto& c_2a_2.\\
    	\end{array}
    	\] 
    	Each of $z_{0,0}, z_{1,0}, z_{1,1}, z_{2,0}, z_{2,1},$ and $z_{3,0}$ gives rise to a column vector, where the entries in the column corresponding to $z_{i,j}$ record the exponents of the variables $ k_1, a_2, a_1, c_1$ and $c_2$ that appear in the (monomial) image of $z_{i,j}$ under the map $\varphi.$ Hence

    	\[
	M =
    	\begin{tikzpicture}[baseline=(current bounding box.center)] 
    	\begin{scope}[xscale=0.7,yscale=0.5]
    	\node at (1,-.8) {$k_1$};
    	\node at (1,-1.8) {$a_2$};
    	\node at (1,-2.8) {$a_1$};
    	\node at (1,-3.8) {$c_1$};
    	\node at (1,-4.8) {$c_2$};
    	\end{scope}
    	\end{tikzpicture}
    	\left(
    	\begin{tikzpicture}[baseline=(current bounding box.center)] 
    	\begin{scope}[xscale=0.7,yscale=0.5]
    	\foreach \x in {1,...,6}
    	\foreach \y in {1,...,5}
    	{
    		\ifnum\x<\y
    		\node (my-\y-\x) at (\x,-\y) {$\phantom 0$};
    		\else
    		\node (my-\y-\x) at (\x,-\y) {$\phantom 0$};
    		\fi
    	}
    	\node (my-1-1) at (1,-1) {$1$};
    	\node (my-1-2) at (1,-2) {$0$};
    	\node (my-1-3) at (1,-3) {$0$};
    	\node (my-1-4) at (1,-4) {$0$};
    	\node (my-1-5) at (1,-5) {$1$};
    	\node (my-2-1) at (2,-1) {$0$};
    	\node (my-2-2) at (2,-2) {$1$};
    	\node (my-2-3) at (2,-3) {$0$};
    	\node (my-2-4) at (2,-4) {$0$};
    	\node (my-2-5) at (2,-5) {$1$};
    	\node (my-3-1) at (3,-1) {$0$};
    	\node (my-3-2) at (3,-2) {$0$};
    	\node (my-3-3) at (3,-3) {$1$};
    	\node (my-3-4) at (3,-4) {$0$};
    	\node (my-3-5) at (3,-5) {$1$};
    	\node (my-4-1) at (4,-1) {$0$};
    	\node (my-4-2) at (4,-2) {$0$};
    	\node (my-4-3) at (4,-3) {$1$};
    	\node (my-4-4) at (4,-4) {$1$};
    	\node (my-4-5) at (4,-5) {$0$};
    	\node (my-5-1) at (5,-1) {$0$};
    	\node (my-5-2) at (5,-2) {$1$};
    	\node (my-5-3) at (5,-3) {$0$};
    	\node (my-5-4) at (5,-4) {$1$};
    	\node (my-5-5) at (5,-5) {$0$};
    	\node (my-6-1) at (6,-1) {$1$};
    	\node (my-6-2) at (6,-2) {$0$};
    	\node (my-6-3) at (6,-3) {$0$};
    	\node (my-6-4) at (6,-4) {$1$};
    	\node (my-6-5) at (6,-5) {$0$};
    	\end{scope}
    	\draw[rounded corners] ($(my-1-1)+(-0.2,0.2)$) rectangle ($(my-3-3)+(0.2,-0.2)$);
    	\draw[rounded corners] ($(my-4-1)+(-0.2,0.2)$) rectangle ($(my-3-6)+(0.2,-0.2)$);
    	\draw[rounded corners] ($(my-1-4)+(-0.2,0.2)$) rectangle ($(my-3-4)+(0.2,-0.2)$);
    	\draw[rounded corners] ($(my-1-5)+(-0.2,0.2)$) rectangle ($(my-3-5)+(0.2,-0.2)$);
    	\draw[rounded corners] ($(my-4-4)+(-0.2,0.2)$) rectangle ($(my-6-4)+(0.2,-0.2)$);
    	\draw[rounded corners] ($(my-4-5)+(-0.2,0.2)$) rectangle ($(my-6-5)+(0.2,-0.2)$);
    	\end{tikzpicture}
    	\right)
    	\hspace*{-13em}
    	\begin{tikzpicture}[baseline=(current bounding box.center)] 
    	\begin{scope}[xscale=0.7,yscale=0.35]
    	\node at (0.5,-1) {$z_{3,0}$};
    	\node at (1.5,-1) {$z_{2,1}$};
    	\node at (2.5,-1) {$z_{1,1}$};
    	\node at (3.5,-1) {$z_{0,0}$};
    	\node at (4.5,-1) {$z_{1,0}$};
    	\node at (5.5,-1) {$z_{2,0}$};
    	\node at (7.5,-10) {\phantom 0};
    	\end{scope}
    	\end{tikzpicture}\kern 3pt.
    	\]
    	
    	The kernel of the map $\varphi$ is by construction, the toric ideal of the matrix $M.$
    \end{example}
    
    \begin{notation}\label{not: order1}In the case $\frac{1}{r}(1,1),$ in a similar way to Example \ref{ex:1311} each $z_{i,j}$ gets mapped under $\varphi$ to a monomial in the arrows, and thus we can build a matrix $M$ where the columns record the exponents. To do this requires us to fix an order on the columns and rows, which we do now. Consider the following diagram.
    	\[
    	\begin{tikzpicture}[xscale=0.8,yscale=0.6]
    	\node (A) at (0,0) {$z_{0,0}$};
    	\node (B) at (2,0) {$z_{1,0}$};
    	\node (C) at (4,0) {$\hdots$};
    	\node (D) at (6,0) {$z_{m-1,0}$};
    	\node (E) at (8,0) {$z_{m,0}$};
    	\node (a) at (0,-1) {$z_{1,1}$};
    	\node (b) at (2,-1) {$z_{2,1}$};
    	\node (c) at (4,-1) {$\hdots$};
    	\node (d) at (6,-1) {$z_{m,1}$};
    	\node (e) at (8,-1) {$z_{m+1,0}$};
    	\draw[green!50!black,rounded corners,->] (8,-1.2) -- (0,-1.2) -- (0,0.2) -- (8,0.2);
    	\end{tikzpicture}
    	\]
    	Following the arrow, we label the columns $1, \ldots, 2r$ of the matrix $M$ by 
    	\[z_{m+1,0}, ~z_{m,1}, \ldots,~z_{1,1},~z_{0,0}, \ldots, ~z_{m-1,0}, ~z_{m,0}. \] 
    	and the rows of $M$ by $k_\ell,\hdots,k_1,a_2,a_1,c_1,c_2$.  With this ordering,
    	\[
    	\begin{tikzpicture}[>=stealth,baseline=(current bounding box.center)] 
    	\begin{scope}[xscale=0.35,yscale=0.35]
    	\foreach \x in {1,...,12}
    	\foreach \y in {1,...,7}
    	{
    		\ifnum\x<\y
    		\node (my-\y-\x) at (\x,-\y) {$\phantom 0$};
    		\else
    		\node (my-\y-\x) at (\x,-\y) {$\phantom 0$};
    		\fi
    	}
    	\end{scope}
    	
    	\draw[<->] ($(my-1-1)+(-0.14,0.21)$) -- node[above] {$\scriptstyle 2r$} ($(my-1-12)+(0.14,0.21)$);
    	\draw[<->] ($(my-1-1)+(-0.21,0.14)$) -- node[left] {$ M= ~ \scriptstyle r$} ($(my-6-1)+(-0.21,-0.14)$);
    	\draw[rounded corners] ($(my-1-1)+(-0.14,0.14)$) rectangle ($(my-6-6)+(0.14,-0.14)$);
    	\node at ($(my-3-3)+(0.21,-0.14)$) {$\Id_r$};
    	\draw[rounded corners] ($(my-1-7)+(-0.14,0.14)$) rectangle ($(my-6-12)+(0.14,-0.14)$);
    	\node at ($(my-3-9)+(0.21,-0.14)$) {$\Id^*_r$};
    	
    	
    	\node at ($(my-6-1)+(0,-0.3)$) {$\scriptstyle 0$};
    	\node at ($(my-6-1)+(0,-0.6)$) {$\scriptstyle 1$};
    	\draw[densely dotted] ($(my-6-1)+(0,-0.45)$) -- ($(my-6-6)+(0,-0.45)$);
    	\node at ($(my-6-6)+(0,-0.3)$) {$\scriptstyle 0$};
    	\node at ($(my-6-6)+(0,-0.6)$) {$\scriptstyle 1$};
    	\node at ($(my-6-7)+(0,-0.3)$) {$\scriptstyle 1$};
    	\node at ($(my-6-7)+(0,-0.6)$) {$\scriptstyle 0$};
    	\draw[densely dotted] ($(my-6-7)+(0,-0.45)$) -- ($(my-6-12)+(0,-0.45)$);
    	\node at ($(my-6-12)+(0,-0.3)$) {$\scriptstyle 1$};
    	\node at ($(my-6-12)+(0,-0.6)$) {$\scriptstyle 0$};
    	
    	\end{tikzpicture}
    	\]
    \end{notation} \noindent
    where $\Id_r$ is the $r \times r$ identity matrix, and $\Id^*_r$ is the anti-diagonal identity matrix.
    
    \medskip
    For the general case $\frac{1}{r}(1,a)$ with $a\neq 1$, there is also a matrix $M$ whose entries are similarly the powers of the variables.  To describe the matrix $M$ requires us to set notation, which we do now.
    
    \begin{notation}\label{not: order2}
    	Consider the following diagram.
    	\[
    	\left(
    	\begin{tikzpicture}[baseline=-\the\dimexpr\fontdimen22\textfont2\relax]
    	\matrix (m)[matrix of math nodes]
    	{
    		z_{0,0}& & z_{1,0} && z_{2,0}& \phantom{\cdots}&z_{m-1,0}& &z_{m,0}\\
    		&z_{1,{s_1}-1}\ldots z_{1,1}&& z_{2,{s_2}-1}\ldots z_{2,1}&&\cdots&& z_{m,{s_m}-1} \ldots z_{m,1}&\\
    		z_{1,s_1} & & z_{2,s_2}&& z_{3,s_3}& &z_{m,s_m}& &z_{m+1,s_{m+1}}\\
    	};
    	\draw[->,rounded corners] (m-3-7.south) --node[below] {$ \scriptstyle 1$} (m-3-1.south west)--(m-3-1.north west) -- (m-2-8.south east) -- (m-1-9.south west)-- (m-1-9.south);
    	\draw[->,rounded corners] (m-1-1.north west) --node[above] {$ \scriptstyle 2$} (m-1-1.north east);
    	\draw[->,rounded corners] (m-3-9.south west) --node[below] {$ \scriptstyle 3$} (m-3-9.south east);
    	\draw[->,rounded corners] (m-1-3.north) --node[above] {$ \scriptstyle 4$} (m-1-7.north);
    	\end{tikzpicture}
    	\right)
    	\]
    	Following the above arrows as numbered, we label the columns $1, \ldots , \ell+n+1$ of $M$ by
    	\[
    	z_{m,s_m},\ldots, z_{1,s_1}, z_{1,{s_1}-1}, \ldots , z_{1,1}, z_{2,{s_2}-1}, \ldots , z_{2,1},\ldots, z_{m,{s_m}-1}, \ldots , z_{m,1}, z_{m,0} 
    	\]
    	Then column $\ell+n+2$ will be labelled $z_{0,0}$, column $\ell+n+3$ labelled $z_{m+1,s_{m+1}}$, and columns $\ell+n+4 , \ldots , 2\ell+n+3$ will be labelled $ z_{1,0},\ldots, z_{m-1,0}$.
    	
    	We next specify the labelling of the rows of $M$. The first $\ell$ rows will be $k_\ell, \ldots , k_1$, then the next rows labelled $a_{01}, \ldots, a_{n0},$ then the next rows  $c_{on}, \ldots, c_{10}$. 
    \end{notation}
    
    \begin{example}\label{map1}
    	For the group $\frac{1}{7}(1,2)$,  the homomorphism $\varphi\colon\mathbb{C}[\mathsf{z}]\twoheadrightarrow \scrR^G$ sends $z_{0,0} \mapsto c_{02}c_{21}c_{10}$, $z_{4,0} \mapsto a_{01}a_{12}a_{20}$, and
    	\[
    	\begin{array}{rclcrclcrcl}
    	z_{1,0} &\mapsto& c_{02}c_{21}k_1&&z_{2,0} &\mapsto& c_{02}c_{21}k_2&&z_{3,0} &\mapsto& c_{02}a_{20}\\
    	z_{1,1} &\mapsto& c_{10}a_{01}&&z_{2,1} &\mapsto& a_{01}k_1&&z_{3,1} &\mapsto& c_{21}a_{12}\\  
    	&&&&&&&&z_{3,2} &\mapsto& a_{01}k_2.\\ 
    	\end{array}
    	\] 
    	The exponents of $z_{0,0}, z_{1,0}, z_{1,1}, z_{2,0}, z_{2,1}, z_{3,0}, z_{3,1}, z_{3,2}$ and $z_{4,0}$ lead to the column vectors	with each entry of any corresponding column vector being the power of the variables $ k_2, k_1, a_{01}, a_{12}, a_{20}, c_{02}, c_{21}$ and $c_{10}$	respectively. Hence
       	
    	\[M =
    	\begin{tikzpicture}[baseline=(current bounding box.center)] 
    	\begin{scope}[xscale=0.7,yscale=0.5]
    	\node at (1,-.8) {$k_2$};
    	\node at (1,-1.8) {$k_1$};
    	\node at (1,-2.8) {$a_{01}$};
    	\node at (1,-3.8) {$a_{12}$};
    	\node at (1,-4.8) {$a_{20}$};
    	\node at (1,-5.8) {$c_{02}$};
    	\node at (1,-6.8) {$c_{21}$};
    	\node at (1,-7.8) {$c_{10}$};
    	\end{scope}
    	\end{tikzpicture}
    	\left(
    	\begin{tikzpicture}[baseline=(current bounding box.center)] 
    	\begin{scope}[xscale=0.7,yscale=0.5]
    	\foreach \x in {1,...,9}
    	\foreach \y in {1,...,8}
    	{
    		\ifnum\x<\y
    		\node (my-\y-\x) at (\x,-\y) {$\phantom 0$};
    		\else
    		\node (my-\y-\x) at (\x,-\y) {$\phantom 0$};
    		\fi
    	}
    	\node (my-1-1) at (1,-1) {$1$};
    	\node (my-1-2) at (1,-2) {$0$};
    	\node (my-1-3) at (1,-3) {$1$};
    	\node (my-1-4) at (1,-4) {$0$};
    	\node (my-1-5) at (1,-5) {$0$};
    	\node (my-1-6) at (1,-6) {$0$};
    	\node (my-1-7) at (1,-7) {$0$};
    	\node (my-1-8) at (1,-8) {$0$};
    	\node (my-2-1) at (2,-1) {$0$};
    	\node (my-2-2) at (2,-2) {$1$};
    	\node (my-2-3) at (2,-3) {$1$};
    	\node (my-2-4) at (2,-4) {$0$};
    	\node (my-2-5) at (2,-5) {$0$};
    	\node (my-2-6) at (2,-6) {$0$};
    	\node (my-2-7) at (2,-7) {$0$};
    	\node (my-2-8) at (2,-8) {$0$};
    	\node (my-3-1) at (3,-1) {$0$};
    	\node (my-3-2) at (3,-2) {$0$};
    	\node (my-3-3) at (3,-3) {$1$};
    	\node (my-3-4) at (3,-4) {$0$};
    	\node (my-3-5) at (3,-5) {$0$};
    	\node (my-3-6) at (3,-6) {$0$};
    	\node (my-3-7) at (3,-7) {$0$};
    	\node (my-3-8) at (3,-8) {$1$};
    	\node (my-4-1) at (4,-1) {$0$};
    	\node (my-4-2) at (4,-2) {$0$};
    	\node (my-4-3) at (4,-3) {$0$};
    	\node (my-4-4) at (4,-4) {$1$};
    	\node (my-4-5) at (4,-5) {$0$};
    	\node (my-4-6) at (4,-6) {$0$};
    	\node (my-4-7) at (4,-7) {$1$};
    	\node (my-4-8) at (4,-8) {$0$};
    	\node (my-5-1) at (5,-1) {$0$};
    	\node (my-5-2) at (5,-2) {$0$};
    	\node (my-5-3) at (5,-3) {$0$};
    	\node (my-5-4) at (5,-4) {$0$};
    	\node (my-5-5) at (5,-5) {$1$};
    	\node (my-5-6) at (5,-6) {$1$};
    	\node (my-5-7) at (5,-7) {$0$};
    	\node (my-5-8) at (5,-8) {$0$};
    	\node (my-6-1) at (6,-1) {$0$};
    	\node (my-6-2) at (6,-2) {$0$};
    	\node (my-6-3) at (6,-3) {$0$};
    	\node (my-6-4) at (6,-4) {$0$};
    	\node (my-6-5) at (6,-5) {$0$};
    	\node (my-6-6) at (6,-6) {$1$};
    	\node (my-6-7) at (6,-7) {$1$};
    	\node (my-6-8) at (6,-8) {$1$};
    	\node (my-7-1) at (7,-1) {$0$};
    	\node (my-7-2) at (7,-2) {$0$};
    	\node (my-7-3) at (7,-3) {$1$};
    	\node (my-7-4) at (7,-4) {$1$};
    	\node (my-7-5) at (7,-5) {$1$};
    	\node (my-7-6) at (7,-6) {$0$};
    	\node (my-7-7) at (7,-7) {$0$};
    	\node (my-7-8) at (7,-8) {$0$};
    	\node (my-8-1) at (8,-1) {$0$};
    	\node (my-8-2) at (8,-2) {$1$};
    	\node (my-8-3) at (8,-3) {$0$};
    	\node (my-8-4) at (8,-4) {$0$};
    	\node (my-8-5) at (8,-5) {$0$};
    	\node (my-8-6) at (8,-6) {$1$};
    	\node (my-8-7) at (8,-7) {$1$};
    	\node (my-8-8) at (8,-8) {$0$};
    	\node (my-9-1) at (9,-1) {$1$};
    	\node (my-9-2) at (9,-2) {$0$};
    	\node (my-9-3) at (9,-3) {$0$};
    	\node (my-9-4) at (9,-4) {$0$};
    	\node (my-9-5) at (9,-5) {$0$};
    	\node (my-9-6) at (9,-6) {$1$};
    	\node (my-9-7) at (9,-7) {$1$};
    	\node (my-9-8) at (9,-8) {$0$};
    	\end{scope}
    	\draw[rounded corners] ($(my-1-1)+(-0.2,0.2)$) rectangle ($(my-2-2)+(0.2,-0.2)$);
    	\draw[rounded corners] ($(my-1-3)+(-0.2,0.2)$) rectangle ($(my-2-5)+(0.2,-0.2)$);
    	\draw[rounded corners] ($(my-1-6)+(-0.2,0.2)$) rectangle ($(my-2-8)+(0.2,-0.2)$);
    	\draw[rounded corners] ($(my-3-1)+(-0.2,0.2)$) rectangle ($(my-5-2)+(0.2,-0.2)$);
    	\draw[rounded corners] ($(my-3-3)+(-0.2,0.2)$) rectangle ($(my-5-5)+(0.2,-0.2)$);
    	\draw[rounded corners] ($(my-3-6)+(-0.2,0.2)$) rectangle ($(my-5-8)+(0.2,-0.2)$);
    	\draw[rounded corners] ($(my-8-1)+(-0.2,0.2)$) rectangle ($(my-9-2)+(0.2,-0.2)$);
    	\draw[rounded corners] ($(my-8-3)+(-0.2,0.2)$) rectangle ($(my-9-5)+(0.2,-0.2)$);
    	\draw[rounded corners] ($(my-8-6)+(-0.2,0.2)$) rectangle ($(my-9-8)+(0.2,-0.2)$);
    	\end{tikzpicture}
    	\right)
    	\hspace*{-19em}
    	\begin{tikzpicture}[baseline=(current bounding box.center)] 
    	\begin{scope}[xscale=0.7,yscale=0.5]
    	\node at (0.5,-1) {$z_{3,2}$};
    	\node at (1.5,-1) {$z_{2,1}$};
    	\node at (2.5,-1) {$z_{1,1}$};
    	\node at (3.5,-1) {$z_{3,1}$};
    	\node at (4.5,-1) {$z_{3,0}$};
    	\node at (5.5,-1) {$z_{0,0}$};
    	\node at (6.5,-1) {$z_{4,0}$};
    	\node at (7.5,-1) {$z_{1,0}$};
    	\node at (8.5,-1) {$z_{2,0}$};
    	\node at (7.5,-10) {\phantom 0};
    	\end{scope}
    	\end{tikzpicture}\kern 3pt.
    	\]
    \end{example}
    
    With the above ordering of the columns and  rows, we now give a general block decomposition of $M$ which explains the boxes in Example \ref{map1}. 
    \begin{lemma}
    	With the ordering on rows and columns as in Notation~\textnormal{\ref{not: order2}},
    	\[
    	\begin{tikzpicture}[>=stealth,baseline=(current bounding box.center)] 
    	\begin{scope}[xscale=0.5,yscale=0.5]
    	\foreach \x in {1,...,14}
    	\foreach \y in {1,...,11}
    	{
    		\ifnum\x<\y
    		\node (my-\y-\x) at (\x,-\y) {$\phantom 0$};
    		\else
    		\node (my-\y-\x) at (\x,-\y) {$\phantom 0$};
    		\fi
    	}
    	\end{scope}
    	\draw[<->] ($(my-1-1)+(-0.2,0.3)$) -- node[above] {$\scriptstyle \ell$} ($(my-1-3)+(0.2,0.3)$);
    	\draw[<->] ($(my-1-1)+(-0.3,0.2)$) -- node[left] {$\scriptstyle \ell$} ($(my-3-1)+(-0.3,-0.2)$);
    	\draw[rounded corners] ($(my-1-1)+(-0.2,0.2)$) rectangle ($(my-3-3)+(0.2,-0.2)$);
    	\node at (my-2-2) {$\Id_\ell$};
    	\draw[rounded corners] ($(my-4-1)+(-0.2,0.2)$) rectangle ($(my-7-3)+(0.2,-0.2)$);
    	\draw[<->] ($(my-4-1)+(-0.3,0.2)$) -- node[left] {$M= ~~\scriptstyle n+1$} ($(my-7-1)+(-0.3,-0.2)$);
    	\node at ($(my-5-2)+(0,-0.25)$) {$A$};
    	\draw[rounded corners] ($(my-8-1)+(-0.2,0.2)$) rectangle ($(my-11-3)+(0.2,-0.2)$);
    	\draw[<->] ($(my-8-1)+(-0.3,0.2)$) -- node[left] {$ \scriptstyle n+1$} ($(my-11-1)+(-0.3,-0.2)$);
    	\node at ($(my-9-2)+(0,-0.25)$) {$0$};
    	\draw[<->] ($(my-1-4)+(-0.2,0.3)$) -- node[above] {$ \scriptstyle n+1$} ($(my-1-7)+(0.2,0.3)$);
    	\draw[rounded corners] ($(my-1-4)+(-0.2,0.2)$) rectangle ($(my-3-7)+(0.2,-0.2)$);
    	\node at ($(my-2-5)+(0.25,0)$) {$0$};
    	\draw[rounded corners] ($(my-4-4)+(-0.2,0.2)$) rectangle ($(my-7-7)+(0.2,-0.2)$);
    	\node at ($(my-5-5)+(0.25,-0.25)$) {$\Id_{n+1}$};
    	\draw[rounded corners] ($(my-8-4)+(-0.2,0.2)$) rectangle ($(my-11-7)+(0.2,-0.2)$);
    	\node at ($(my-9-5)+(0.25,-0.25)$) {$\Id^*_{n+1}$};
    	\node at ($(my-1-8)+(0.1,0)$) {$\scriptstyle 0$};
    	\node at ($(my-1-8)+(-0.1,0)$) {$\scriptstyle 0$};
    	\draw[densely dotted] (my-1-8) -- (my-3-8);
    	\node at ($(my-3-8)+(0.1,0)$) {$\scriptstyle 0$};
    	\node at ($(my-3-8)+(-0.1,0)$) {$\scriptstyle 0$};
    	\node at ($(my-4-8)+(0.1,0)$) {$\scriptstyle 1$};
    	\node at ($(my-4-8)+(-0.1,0)$) {$\scriptstyle 0$};
    	\draw[densely dotted] (my-4-8) -- (my-7-8);
    	\node at ($(my-7-8)+(0.1,0)$) {$\scriptstyle 1$};
    	\node at ($(my-7-8)+(-0.1,0)$) {$\scriptstyle 0$};
    	\node at ($(my-8-8)+(0.1,0)$) {$\scriptstyle 0$};
    	\node at ($(my-8-8)+(-0.1,0)$) {$\scriptstyle 1$};
    	\draw[densely dotted] (my-8-8) -- (my-11-8);
    	\node at ($(my-11-8)+(0.1,0)$) {$\scriptstyle 0$};
    	\node at ($(my-11-8)+(-0.1,0)$) {$\scriptstyle 1$};
    	\draw[<->] ($(my-1-9)+(-0.2,0.3)$) -- node[above] {$\scriptstyle \ell$} ($(my-1-11)+(0.2,0.3)$);
    	\draw[rounded corners] ($(my-1-9)+(-0.2,0.2)$) rectangle ($(my-3-11)+(0.2,-0.2)$);
    	\node at (my-2-10) {$\Id^*_\ell$};
    	\draw[rounded corners] ($(my-4-9)+(-0.2,0.2)$) rectangle ($(my-7-11)+(0.2,-0.2)$);
    	\node at ($(my-5-10)+(0,-0.25)$) {$0$};
    	\draw[rounded corners] ($(my-8-9)+(-0.2,0.2)$) rectangle ($(my-11-11)+(0.2,-0.2)$);
    	\node at ($(my-9-10)+(0,-0.25)$) {$B$};
    	\end{tikzpicture}
    	\]
    \end{lemma}
    \begin{proof}
    	In the $z$'s, the $k$'s only appear as illustrated below.
    	\[
    	\begin{tikzpicture}[baseline=-\the\dimexpr\fontdimen22\textfont2\relax]
    	\matrix (m)[matrix of math nodes]
    	{
    		z_{0,0}& & z_{1,0} && z_{2,0}& \phantom{\cdots}&z_{m-1,0}& &z_{m,0}\\
    		&z_{1,{s_1}-1}\ldots z_{1,1}&& z_{2,{s_2}-1}\ldots z_{2,1}&&\cdots&& z_{m,{s_m}-1} \ldots z_{m,1}&\\
    		z_{1,s_1} & & z_{2,s_2}&& z_{3,s_3}& &z_{m,s_m}& &z_{m+1,s_{m+1}}\\
    	};
    	\draw[rounded corners] (m-3-7.south east)--node[below]{$\scriptstyle k_\ell$}(m-3-7.south west)--($(m-3-7.south west)+(0,1.55)$)--($(m-3-7.south east)+(0,1.55)$)--cycle;;
    	\draw[rounded corners] (m-3-5.south east)--node[below]{$\scriptstyle k_2$}(m-3-5.south west)--($(m-3-5.south west)+(0,1.55)$)--($(m-3-5.south east)+(0,1.55)$)--cycle;;
    	\draw[rounded corners] (m-3-3.south east)--node[below]{$\scriptstyle k_1$}(m-3-3.south west)--($(m-3-3.south west)+(0,1.55)$)--($(m-3-3.south east)+(0,1.55)$)--cycle;;
    	\end{tikzpicture}
    	\]
    	Due to the ordering on rows and columns, the first $l$ rows of $M$ are thus
    	\[
    	\begin{tikzpicture}[>=stealth,baseline=(current bounding box.center)] 
    	\begin{scope}[xscale=0.5,yscale=0.5]
    	\foreach \x in {1,...,14}
    	\foreach \y in {1,...,3}
    	{
    		\ifnum\x<\y
    		\node (my-\y-\x) at (\x,-\y) {$\phantom 0$};
    		\else
    		\node (my-\y-\x) at (\x,-\y) {$\phantom 0$};
    		\fi
    	}
    	\end{scope}
    	\draw[<->] ($(my-1-1)+(-0.2,0.3)$) -- node[above] {$\scriptstyle \ell$} ($(my-1-3)+(0.2,0.3)$);
    	\draw[<->] ($(my-1-1)+(-0.3,0.2)$) -- node[left] {$\scriptstyle \ell$} ($(my-3-1)+(-0.3,-0.2)$);
    	\draw[rounded corners] ($(my-1-1)+(-0.2,0.2)$) rectangle ($(my-3-3)+(0.2,-0.2)$);
    	\node at (my-2-2) {$\Id_\ell$};

    	\draw[<->] ($(my-1-4)+(-0.2,0.3)$) -- node[above] {$ \scriptstyle n+1$} ($(my-1-7)+(0.2,0.3)$);
    	\draw[rounded corners] ($(my-1-4)+(-0.2,0.2)$) rectangle ($(my-3-7)+(0.2,-0.2)$);
    	\node at ($(my-2-5)+(0.25,0)$) {$0$};
    	
    	\node at ($(my-1-8)+(0.1,0)$) {$\scriptstyle 0$};
    	\node at ($(my-1-8)+(-0.1,0)$) {$\scriptstyle 0$};
    	\draw[densely dotted] (my-1-8) -- (my-3-8);
    	\node at ($(my-3-8)+(0.1,0)$) {$\scriptstyle 0$};
    	\node at ($(my-3-8)+(-0.1,0)$) {$\scriptstyle 0$};
    	
    	\draw[<->] ($(my-1-9)+(-0.2,0.3)$) -- node[above] {$\scriptstyle \ell$} ($(my-1-11)+(0.2,0.3)$);
    	\draw[rounded corners] ($(my-1-9)+(-0.2,0.2)$) rectangle ($(my-3-11)+(0.2,-0.2)$);
    	\node at (my-2-10) {$\Id^*_\ell$};

    	\end{tikzpicture}
    	\]
    	where $\Id_\ell$ is the $\ell \times \ell$ identity matrix, and $\Id^*_\ell$ is the anti-diagonal identity matrix. Similarly, in the $z$'s, the $a$'s only appear in the following region.
    	
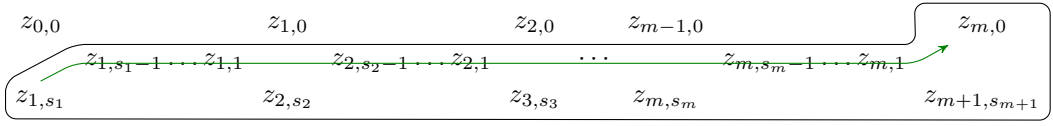
\[
    	\begin{tikzpicture}[baseline=-\the\dimexpr\fontdimen22\textfont2\relax]
    	\matrix (m)[matrix of math nodes]
    	{
    		z_{0,0}& & z_{1,0} && z_{2,0}& \phantom{\cdots}&z_{m-1,0}& &z_{m,0}\\
    		&z_{1,{s_1}-1}\ldots z_{1,1}&& z_{2,{s_2}-1}\ldots z_{2,1}&&\cdots&& z_{m,{s_m}-1} \ldots z_{m,1}&\\
    		z_{1,s_1} & & z_{2,s_2}&& z_{3,s_3}& &z_{m,s_m}& &z_{m+1,s_{m+1}}\\
    	};
    	\draw[rounded corners] (m-3-9.south east)--(m-3-1.south west)--(m-3-1.north west)--(m-2-2.north west)--(m-2-8.north east)--($(m-3-9.south west)+(0,1.55)$)--($(m-3-9.south east)+(0,1.55)$)--cycle;;
    	\draw[green!50!black,->,rounded corners] (m-3-1.north)--(m-2-2.west) -- (m-2-8.east) -- (m-1-9);
    	\end{tikzpicture}
    	\]
    	Furthermore, along the green arrow, out of all the $a$'s, the first $z_{1,s_1}$ contains only $a_{01}$, the second entry $z_{1,1}$ contains only $a_{12}$, etc until the last entry $z_{m,0}$ on the green line, which contains only $a_{n0}$.  It follows that the next $n+1$ rows of $M$ are 
    	\[
    	\begin{tikzpicture}[>=stealth,baseline=(current bounding box.center)] 
    	\begin{scope}[xscale=0.5,yscale=0.5]
    	\foreach \x in {1,...,14}
    	\foreach \y in {3,...,7}
    	{
    		\ifnum\x<\y
    		\node (my-\y-\x) at (\x,-\y) {$\phantom 0$};
    		\else
    		\node (my-\y-\x) at (\x,-\y) {$\phantom 0$};
    		\fi
    	}
    	\end{scope}
    	
    	
    	\draw[<->] ($(my-4-1)+(-0.2,0.3)$) -- node[above] {$\scriptstyle \ell$} ($(my-4-3)+(0.2,0.3)$);

    	\draw[rounded corners] ($(my-4-1)+(-0.2,0.2)$) rectangle ($(my-7-3)+(0.2,-0.2)$);
    	\draw[<->] ($(my-4-1)+(-0.3,0.2)$) -- node[left] {$\scriptstyle n+1$} ($(my-7-1)+(-0.3,-0.2)$);
    	\node at ($(my-5-2)+(0,-0.25)$) {$A$};
    	
    	
    	\draw[<->] ($(my-4-4)+(-0.2,0.3)$) -- node[above] {$ \scriptstyle n+1$} ($(my-4-7)+(0.2,0.3)$);
    	\draw[rounded corners] ($(my-4-4)+(-0.2,0.2)$) rectangle ($(my-7-7)+(0.2,-0.2)$);
    	\node at ($(my-5-5)+(0.25,-0.25)$) {$\Id_{n+1}$};
    	
    	;
    	\node at ($(my-4-8)+(0.1,0)$) {$\scriptstyle 1$};
    	\node at ($(my-4-8)+(-0.1,0)$) {$\scriptstyle 0$};
    	\draw[densely dotted] (my-4-8) -- (my-7-8);
    	\node at ($(my-7-8)+(0.1,0)$) {$\scriptstyle 1$};
    	\node at ($(my-7-8)+(-0.1,0)$) {$\scriptstyle 0$};
    	
    	\draw[rounded corners] ($(my-4-9)+(-0.2,0.2)$) rectangle ($(my-7-11)+(0.2,-0.2)$);
    	\node at ($(my-5-10)+(0,-0.25)$) {$0$};
    	\draw[<->] ($(my-4-9)+(-0.2,0.3)$) -- node[above] {$\scriptstyle \ell$} ($(my-4-11)+(0.2,0.3)$);
    	\end{tikzpicture}
    	\]
    	for some matrix $A$ (see Remark~\ref{rem:A and B} below).

    	Lastly, in a very similar way the only place the $c$'s exist in the $z$'s are in the following region
    	
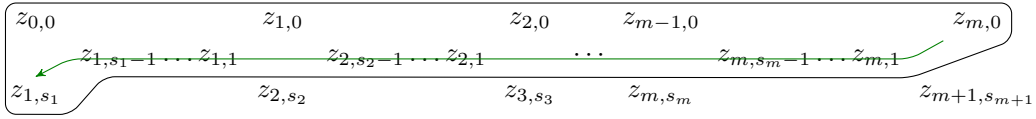
\[
    	\begin{tikzpicture}[baseline=-\the\dimexpr\fontdimen22\textfont2\relax]
    	\matrix (m)[matrix of math nodes]
    	{
    		z_{0,0}& & z_{1,0} && z_{2,0}& \phantom{\cdots}&z_{m-1,0}& &z_{m,0}\\
    		&z_{1,{s_1}-1}\ldots z_{1,1}&& z_{2,{s_2}-1}\ldots z_{2,1}&&\cdots&& z_{m,{s_m}-1} \ldots z_{m,1}&\\
    		z_{1,s_1} & & z_{2,s_2}&& z_{3,s_3}& &z_{m,s_m}& &z_{m+1,s_{m+1}}\\
    	};
    	\draw[rounded corners] (m-1-1.north west)--(m-1-9.north east)--(m-1-9.south east) -- (m-2-8.south east) -- ($(m-3-1.south)+(0.9,0.5)$)-- (m-3-1.south east)--($(m-1-1.north west)+(0,-1.475)$)--cycle;;
    	\draw[green!50!black,<-,rounded corners] (m-3-1.north)--(m-2-2.west) -- (m-2-8.east) -- (m-1-9);
    	\end{tikzpicture}
    	\]
    	where again following the green line, among all the $c$'s, the first $z_{m,0}$ contains only $c_{0n}$, the second entry contains only $c_{nn-1}$, etc until the last entry $z_{1,s_1}$ on the green line, which contains only $c_{10}$.  It follows that the next $n+1$ rows of $M$ are 
    	\[
    	\begin{tikzpicture}[>=stealth,baseline=(current bounding box.center)] 
    	\begin{scope}[xscale=0.5,yscale=0.5]
    	\foreach \x in {1,...,14}
    	\foreach \y in {7,...,11}
    	{
    		\ifnum\x<\y
    		\node (my-\y-\x) at (\x,-\y) {$\phantom 0$};
    		\else
    		\node (my-\y-\x) at (\x,-\y) {$\phantom 0$};
    		\fi
    	}
    	\end{scope}
    	\draw[rounded corners] ($(my-8-1)+(-0.2,0.2)$) rectangle ($(my-11-3)+(0.2,-0.2)$);
    	\draw[<->] ($(my-8-1)+(-0.3,0.2)$) -- node[left] {$ \scriptstyle n+1$} ($(my-11-1)+(-0.3,-0.2)$);
    	\node at ($(my-9-2)+(0,-0.25)$) {$0$};
    	\draw[<->] ($(my-8-1)+(-0.2,0.3)$) -- node[above] {$\scriptstyle \ell$} ($(my-8-3)+(0.2,0.3)$);
    	\draw[<->] ($(my-8-4)+(-0.2,0.3)$) -- node[above] {$ \scriptstyle n+1$} ($(my-8-7)+(0.2,0.3)$);
    	\draw[rounded corners] ($(my-8-4)+(-0.2,0.2)$) rectangle ($(my-11-7)+(0.2,-0.2)$);
    	\node at ($(my-9-5)+(0.25,-0.25)$) {$\Id^*_{n+1}$};
    	\node at ($(my-8-8)+(0.1,0)$) {$\scriptstyle 0$};
    	\node at ($(my-8-8)+(-0.1,0)$) {$\scriptstyle 1$};
    	\draw[densely dotted] (my-8-8) -- (my-11-8);
    	\node at ($(my-11-8)+(0.1,0)$) {$\scriptstyle 0$};
    	\node at ($(my-11-8)+(-0.1,0)$) {$\scriptstyle 1$};
    	\draw[<->] ($(my-8-9)+(-0.2,0.3)$) -- node[above] {$\scriptstyle \ell$} ($(my-8-11)+(0.2,0.3)$);
    	\draw[rounded corners] ($(my-8-9)+(-0.2,0.2)$) rectangle ($(my-11-11)+(0.2,-0.2)$);
    	\node at ($(my-9-10)+(0,-0.25)$) {$B$};
    	\end{tikzpicture}
    	\]
    	for some matrix $B$.  The result follows.
    \end{proof}
    
    \begin{remark}\label{rem:A and B}
    	Although not required, it is possible to explicitly describe both the matrices $A$ and $B$.  For $A$, there are $\upbeta_1 - 1,~\upbeta_2 - 2,~\upbeta_3 - 2 \ldots, \upbeta_{m-1}-2, \upbeta_m -1$ rows each containing $\{1,1, \ldots,1, 1, 1\},~\{1,1, \ldots,1, 1, 0\}, ~\{1,1, \ldots,1, 0, 0\}, \ldots, \{0,0, \ldots,0, 0, 0\}$	respectively.
    	
    	For $B$, there are $\upbeta_m - 1,~\upbeta_{m-1} - 2,~\upbeta_{m-2} - 2, \ldots, \upbeta_2- 2, \upbeta_1-1$ rows, each containing
    	$\{1,1, \ldots,1, 1, 1\},~\{1,1, \ldots,1, 1, 0\}, ~\{1,1, \ldots,1, 0, 0\}, \ldots, \{0,0, \ldots,0, 0, 0\}$ respectively.
    \end{remark}

    Now consider the $2+ \sum \upbeta_i = 2\ell + n + 3$ square matrix
    
    \[
    \begin{tikzpicture}[>=stealth,baseline=(current bounding box.center)] 
    \begin{scope}[xscale=0.35,yscale=0.35]
    \foreach \x in {1,...,11}
    \foreach \y in {1,...,11}
    {
    	\ifnum\x<\y
    	\node (my-\y-\x) at (\x,-\y) {$\phantom 0$};
    	\else
    	\node (my-\y-\x) at (\x,-\y) {$\phantom 0$};
    	\fi
    }
    \end{scope}
    \draw[<->] ($(my-1-1)+(-0.14,0.21)$) -- node[above] {$\scriptstyle \ell+n+2$} ($(my-1-8)+(0.14,0.21)$);
    \draw[<->] ($(my-1-1)+(-0.21,0.14)$) -- node[left] {$ Q= ~~ \scriptstyle\ell+n+2$} ($(my-8-1)+(-0.21,-0.14)$);
    \draw[rounded corners] ($(my-1-1)+(-0.14,0.14)$) rectangle ($(my-8-8)+(0.14,-0.14)$);
    \node at (my-5-5) {$\Id_{\ell+n+2}$};
    \draw[rounded corners] ($(my-9-1)+(-0.14,0.14)$) rectangle ($(my-11-8)+(0.14,-0.14)$);
    \draw[<->] ($(my-9-1)+(-0.21,0.14)$) -- node[left] {$\scriptstyle \ell+1$} ($(my-11-1)+(-0.21,-0.14)$);
    \node at ($(my-9-4)+(0.1,-0.4)$) {$0$};
    \draw[<->] ($(my-1-9)+(-0.14,0.21)$) -- node[above] {$\scriptstyle \ell+1$} ($(my-1-11)+(0.14,0.21)$);
    \draw[<->] ($(my-1-11)+(0.21,0.14)$) -- node[right] {$\scriptstyle 2\ell+n+3$} ($(my-11-11)+(0.21,-0.14)$);
    \draw[rounded corners] ($(my-1-9)+(-0.14,0.14)$) rectangle ($(my-11-11)+(0.14,-0.14)$);
    \node at (my-6-10) {$K$};
    \end{tikzpicture}
    \]
    where $K$ is the $(2 \ell + n + 3) \times (\ell + 1)$ matrix				\[
    \begin{tikzpicture}[>=stealth,baseline=(current bounding box.center)] 
    \begin{scope}[xscale=0.5,yscale=0.5]
    \foreach \x in {1,...,6}
    \foreach \y in {1,...,12}
    {
    	\ifnum\x<\y
    	\node (my-\y-\x) at (\x,-\y) {$\phantom 0$};
    	\else
    	\node (my-\y-\x) at (\x,-\y) {$\phantom 0$};
    	\fi
    }
    \end{scope}
    \draw[<->] ($(my-1-4)+(0.3,0.2)$) -- node[right] {$\scriptstyle \ell$} ($(my-3-4)+(0.3,-0.2)$);
    \draw[draw=none] ($(my-4-1)+(-0.3,0.2)$) -- node[left] {$K=$} ($(my-7-1)+(-0.3,-0.2)$);
    \draw[<->] ($(my-4-4)+(0.3,0.2)$) -- node[right] {$\scriptstyle n+1$} ($(my-7-4)+(0.3,-0.2)$);
    \draw[<->] ($(my-10-4)+(0.3,0.2)$) -- node[right] {$ \scriptstyle \ell$} ($(my-12-4)+(0.3,-0.2)$);
    \node at ($(my-1-1)+(0.1,0)$) {$\scriptstyle 0$};
    \draw[densely dotted] ($(my-1-1)+(0.1,-0.2)$) -- ($(my-3-1)+(0.1,0.2)$);
    \node at ($(my-3-1)+(0.1,0)$) {$\scriptstyle 0$};
    \node at (my-4-1) {$\scriptstyle -1$};
    \draw[densely dotted] ($(my-4-1)+(0.1,-0.2)$) -- ($(my-7-1)+(0.1,0.2)$);
    \node at (my-7-1) {$\scriptstyle -1$};
    \node at ($(my-8-1)+(0.1,0.02)$) {$\scriptstyle 1$};
    \node at ($(my-8-2)+(-0.1,0)$) {$\scriptstyle -1$};
    \draw[densely dotted] ($(my-8-1)+(0.7,0)$) -- (my-8-4);
    \node at (my-8-4) {$\scriptstyle -1$};
    \node at ($(my-9-1)+(0.1,0)$) {$\scriptstyle 1$};
    \node at ($(my-10-1)+(0.1,0)$) {$\scriptstyle 0$};
    \node at (my-9-2) {$\scriptstyle 0$};
    \draw[densely dotted] ($(my-9-1)+(0.7,0)$) -- (my-9-4);
    \node at ($(my-9-4)+(0.1,0)$) {$\scriptstyle 0$};
    \draw[densely dotted] ($(my-10-1)+(0.1,-0.2)$) -- ($(my-12-1)+(0.1,0.2)$);
    \node at ($(my-12-1)+(0.1,0)$) {$\scriptstyle 0$};
    \draw[<->] ($(my-1-2)+(-0.2,0.3)$) -- node[above] {$\scriptstyle \ell$} ($(my-1-4)+(0.2,0.3)$);
    \draw[rounded corners] ($(my-1-2)+(-0.2,0.2)$) rectangle ($(my-3-4)+(0.2,-0.2)$);
    \node at (my-2-3) {$-\Id^*_\ell$};
    \draw[rounded corners] ($(my-4-2)+(-0.2,0.2)$) rectangle ($(my-7-4)+(0.2,-0.2)$);
    \node at ($(my-5-3)+(0,-0.2)$) {$V$};
    \draw[rounded corners] ($(my-10-2)+(-0.2,0.2)$) rectangle ($(my-12-4)+(0.2,-0.2)$);
    \node at (my-11-3) {$\Id_\ell$};
    \end{tikzpicture}
    \]
    and the matrix $V$  has $\upbeta_1 - 1,~\upbeta_2 - 2,~\upbeta_3 - 2, \ldots, \upbeta_{m-1}-2, \upbeta_m -1$ rows, each containing $\{1,1, \ldots,1, 1, 1\},~\{0,1, \ldots,1, 1, 1\}, ~\{0,0, 1,\ldots,1, 1\}, \ldots, \{0,0, \ldots,0, 0, 0\}$ respectively.  The matrix $K$ encodes the $\QDet$ relations starting from $z_{00}$, namely
    \begin{align*}
    z_{0,0}z_{m+1,s_{m+1}} &= z_{1,{s_1}} \cdot z_{1,{s_1}-1}\ldots z_{1,1} \cdot z_{2,{s_2}-1}\ldots z_{2,1} \ldots z_{m,{s_m}-1} \ldots z_{m,1} \cdot z_{m,0}\\
    z_{0,0}z_{2,s_2} &=  z_{1,{s_1}} \cdot z_{1,{s_1}-1}\ldots z_{1,1} \cdot z_{1,0}\\
    z_{0,0}z_{3,s_3} &=  z_{1,{s_1}} \cdot z_{1,{s_1}-1}\ldots z_{1,1} \cdot z_{2,{s_2}-1}\ldots z_{2,1} \cdot z_{2,0}\\
    & \vdots\\		             
    z_{0,0}z_{m,s_{m}} &= z_{1,{s_1}} \cdot z_{1,{s_1}-1}\ldots z_{1,1} \cdot z_{2,{s_2}-1}\ldots z_{2,1} \ldots z_{m-1,s_{m-1}-1} \ldots z_{m-1,1} \cdot z_{m-1,0}.
    \end{align*}
    
    \begin{example}\label{Example4.10}
    	Continuing Example \ref{map1}, here $\upbeta_1=2$ and $\upbeta_2=3$, so the associated matrix $K$ is  
    	\[ 
    	K = \left(\!\!\begin{array}{rrr}
    	0 &0&-1\\
    	0 &-1&0\\
    	-1&1&1\\
    	-1&0&0\\
    	-1&0&0\\
    	1&-1&-1\\
    	1&0&0\\
    	0&1&0\\
    	0&0&1
    	\end{array}      \right)
    	.\]
    \end{example}
    
    \begin{lemma}\label{lem:Q inv and MQ}
    	$Q$ is invertible and further 	
    	\[
    	\begin{tikzpicture}[>=stealth,baseline=(current bounding box.center)] 
    	\begin{scope}[xscale=0.5,yscale=0.5]
    	\foreach \x in {1,...,14}
    	\foreach \y in {1,...,11}
    	{
    		\ifnum\x<\y
    		\node (my-\y-\x) at (\x,-\y) {$\phantom 0$};
    		\else
    		\node (my-\y-\x) at (\x,-\y) {$\phantom 0$};
    		\fi
    	}
    	\end{scope}
    	\draw[<->] ($(my-1-1)+(-0.2,0.3)$) -- node[above] {$\scriptstyle \ell$} ($(my-1-3)+(0.2,0.3)$);
    	\draw[<->] ($(my-1-1)+(-0.3,0.2)$) -- node[left] {$\scriptstyle \ell$} ($(my-3-1)+(-0.3,-0.2)$);
    	\draw[rounded corners] ($(my-1-1)+(-0.2,0.2)$) rectangle ($(my-3-3)+(0.2,-0.2)$);
    	\node at (my-2-2) {$\Id_\ell$};
    	\draw[rounded corners] ($(my-4-1)+(-0.2,0.2)$) rectangle ($(my-7-3)+(0.2,-0.2)$);
    	\draw[<->] ($(my-4-1)+(-0.3,0.2)$) -- node[left] {$MQ= ~~\scriptstyle n+1$} ($(my-7-1)+(-0.3,-0.2)$);
    	\node at ($(my-5-2)+(0,-0.25)$) {$A$};
    	\draw[rounded corners] ($(my-8-1)+(-0.2,0.2)$) rectangle ($(my-11-3)+(0.2,-0.2)$);
    	\draw[<->] ($(my-8-1)+(-0.3,0.2)$) -- node[left] {$ \scriptstyle n+1$} ($(my-11-1)+(-0.3,-0.2)$);
    	\node at ($(my-9-2)+(0,-0.25)$) {$0$};
    	\draw[<->] ($(my-1-4)+(-0.2,0.3)$) -- node[above] {$ \scriptstyle n+1$} ($(my-1-7)+(0.2,0.3)$);
    	\draw[rounded corners] ($(my-1-4)+(-0.2,0.2)$) rectangle ($(my-3-7)+(0.2,-0.2)$);
    	\node at ($(my-2-5)+(0.25,0)$) {$0$};
    	\draw[rounded corners] ($(my-4-4)+(-0.2,0.2)$) rectangle ($(my-7-7)+(0.2,-0.2)$);
    	\node at ($(my-5-5)+(0.25,-0.25)$) {$\Id_{n+1}$};
    	\draw[rounded corners] ($(my-8-4)+(-0.2,0.2)$) rectangle ($(my-11-7)+(0.2,-0.2)$);
    	\node at ($(my-9-5)+(0.25,-0.25)$) {$\Id^*_{n+1}$};
    	\node at ($(my-1-8)+(0.1,0)$) {$\scriptstyle 0$};
    	\node at ($(my-1-8)+(-0.1,0)$) {$\scriptstyle 0$};
    	\node at ($(my-3-8)+(0.1,0)$) {$\scriptstyle 0$};
    	\node at ($(my-3-8)+(-0.1,0)$) {$\scriptstyle 0$};
    	\node at ($(my-4-8)+(0.1,0)$) {$\scriptstyle 0$};
    	\node at ($(my-4-8)+(-0.1,0)$) {$\scriptstyle 0$};
    	\node at ($(my-7-8)+(0.1,0)$) {$\scriptstyle 0$};
    	\node at ($(my-7-8)+(-0.1,0)$) {$\scriptstyle 0$};
    	\node at ($(my-8-8)+(0.1,0)$) {$\scriptstyle 0$};
    	\node at ($(my-8-8)+(-0.1,0)$) {$\scriptstyle 1$};
    	\node at ($(my-11-8)+(0.1,0)$) {$\scriptstyle 0$};
    	\node at ($(my-11-8)+(-0.1,0)$) {$\scriptstyle 1$};
    	\draw[<->] ($(my-1-9)+(-0.2,0.3)$) -- node[above] {$\scriptstyle \ell$} ($(my-1-11)+(0.2,0.3)$);
    	\draw[rounded corners] ($(my-1-9)+(-0.2,0.2)$) rectangle ($(my-3-11)+(0.2,-0.2)$);
    	\node at (my-2-10) {$0$};
    	\draw[rounded corners] ($(my-4-9)+(-0.2,0.2)$) rectangle ($(my-7-11)+(0.2,-0.2)$);
    	\node at ($(my-5-10)+(0,-0.25)$) {$0$};
    	\draw[rounded corners] ($(my-8-9)+(-0.2,0.2)$) rectangle ($(my-11-11)+(0.2,-0.2)$);
    	\node at ($(my-9-10)+(0,-0.25)$) {$0$};
    	\draw[densely dotted] ($(my-1-8)+(0,0.33)$) --  ($(my-11-8)+(0,-0.33)$);
    	\end{tikzpicture}
    	\]
    \end{lemma}
    \begin{proof}
    	$Q$ is invertible since any unitriangular matrix has determinant one. For the second statement, since $\QDet \subseteq \Ker_\mathbb{Z}$, it follows  that $MK=0$.  This justifies the last $\ell+1$ columns above.  The first $\ell+n+2$ columns are clear, since multiplying $M$ on the right by the unit matrix $\Id_{\ell+n+2}$ with zero underneath picks out the first  $\ell+n+2$ columns of $M$ only.  Thus the first  $\ell+n+2$ columns of $M$ are the first  $\ell+n+2$ columns above.
    \end{proof}

    \begin{cor}\label{cor: ker Z 1}
    	$\Ker_\mathbb{Z}$ is generated by the columns of $K$.
    \end{cor}
    \begin{proof}
    	By the form of $MQ$ in Lemma~\ref{lem:Q inv and MQ}, it is clear that it is possible to obtain Smith Normal Form from $MQ$ using only row operations.  This gives an invertible matrix $R$ for which 
    	\[
    	\begin{tikzpicture}[>=stealth,baseline=(current bounding box.center)] 
    	\begin{scope}[xscale=0.35,yscale=0.35]
    	\foreach \x in {1,...,11}
    	\foreach \y in {1,...,11}
    	{
    		\ifnum\x<\y
    		\node (my-\y-\x) at (\x,-\y) {$\phantom 0$};
    		\else
    		\node (my-\y-\x) at (\x,-\y) {$\phantom 0$};
    		\fi
    	}
    	\end{scope}
    	\draw[<->] ($(my-1-1)+(-0.14,0.21)$) -- node[above] {$\scriptstyle \ell+n+2$} ($(my-1-8)+(0.14,0.21)$);
    	\draw[<->] ($(my-1-1)+(-0.21,0.14)$) -- node[left] {$ RMQ= ~~ \scriptstyle\ell+n+2$} ($(my-8-1)+(-0.21,-0.14)$);
    	\draw[rounded corners] ($(my-1-1)+(-0.14,0.14)$) rectangle ($(my-8-8)+(0.14,-0.14)$);
    	\node at (my-5-5) {$\Id_{\ell+n+2}$};
    	\draw[rounded corners] ($(my-9-1)+(-0.14,0.14)$) rectangle ($(my-11-8)+(0.14,-0.14)$);
    	\draw[<->] ($(my-9-1)+(-0.21,0.14)$) -- node[left] {$\scriptstyle n$} ($(my-11-1)+(-0.21,-0.14)$);
    	\node at ($(my-9-4)+(0.1,-0.4)$) {$0$};
    	\draw[<->] ($(my-1-9)+(-0.14,0.21)$) -- node[above] {$\scriptstyle \ell+1$} ($(my-1-11)+(0.14,0.21)$);
    	\draw[rounded corners] ($(my-1-9)+(-0.14,0.14)$) rectangle ($(my-11-11)+(0.14,-0.14)$);
    	\node at (my-6-10) {$0$};
    	\end{tikzpicture}
    	\] 
    	It follows from Smith Normal Form that $\Ker_\mathbb{Z}$ is generated by the last $\ell+1$ columns of $Q$, which are precisely the columns of $K$.
    \end{proof}
    
    \begin{remark}
    	In the case $a=1,$ equivalently for the groups $\frac{1}{r}(1,1),$ consider the matrix $Q$ defined as 
    	\[
    	\begin{tikzpicture}[>=stealth,baseline=(current bounding box.center)] 
    	\begin{scope}[xscale=0.35,yscale=0.35]
    	\foreach \x in {1,...,10}
    	\foreach \y in {1,...,10}
    	{
    		\ifnum\x<\y
    		\node (my-\y-\x) at (\x,-\y) {$\phantom 0$};
    		\else
    		\node (my-\y-\x) at (\x,-\y) {$\phantom 0$};
    		\fi
    	}
    	\end{scope}
    	\draw[<->] ($(my-1-1)+(-0.14,0.21)$) -- node[above] {$\scriptstyle r$} ($(my-1-5)+(0.14,0.21)$);
    	\draw[<->] ($(my-1-1)+(-0.21,0.14)$) -- node[left] {$\scriptstyle r$} ($(my-5-1)+(-0.21,-0.14)$);
    	\draw[rounded corners] ($(my-1-1)+(-0.14,0.14)$) rectangle ($(my-5-5)+(0.14,-0.14)$);
    	\node at ($(my-2-2)+(0.35,-0.35)$) {$\Id_{r}$};
    	\node at ($(my-3-1)+(-1,-0.9)$) {$Q=$};
    	\draw[rounded corners] ($(my-6-1)+(-0.14,0.14)$) rectangle ($(my-10-5)+(0.14,-0.14)$);
    	\draw[<->] ($(my-6-1)+(-0.21,0.14)$) -- node[left] {$\scriptstyle r$} ($(my-10-1)+(-0.21,-0.14)$);
    	\node at ($(my-7-2)+(0.35,-0.35)$) {$0$};
    	
    	\draw[<->] ($(my-1-7)+(0,0.21)$) -- node[above] {$\scriptstyle r-1$} ($(my-1-10)+(0.21,0.21)$);
    	\draw[rounded corners] ($(my-1-7)+(0,0.14)$) rectangle ($(my-10-10)+(0.21,-0.14)$);
    	\node at ($(my-6-9)+(-0.1,0.2)$) {$K$};
    	\node at ($(my-1-6)+(0.1,0)$) {$\scriptstyle 0$};
    	\draw[densely dotted] ($(my-1-6)+(0.1,-0.2)$) -- ($(my-4-6)+(0.1,0.2)$);
    	\node at ($(my-4-6)+(0.1,0)$) {$\scriptstyle 0$};
    	\node at ($(my-5-6)+(0,0)$) {$\scriptstyle -1$};

    	\node at ($(my-6-6)+(0.1,0)$) {$\scriptstyle 1$};
    	\node at ($(my-7-6)+(0.1,0)$) {$\scriptstyle 0$};
    	\draw[densely dotted] ($(my-7-6)+(0.1,-0.2)$) -- ($(my-10-6)+(0.1,0.2)$);
    	\node at ($(my-10-6)+(0.1,0)$) {$\scriptstyle 0$};
    	\end{tikzpicture}
    	\begin{tikzpicture}[>=stealth,baseline=(current bounding box.center)] 
    	\begin{scope}[xscale=0.35,yscale=0.35]
    	\foreach \x in {1,...,10}
    	\foreach \y in {1,...,10}
    	{
    		\ifnum\x<\y
    		\node (my-\y-\x) at (\x,-\y) {$\phantom 0$};
    		\else
    		\node (my-\y-\x) at (\x,-\y) {$\phantom 0$};
    		\fi
    	}
    	\end{scope}
    	\draw[<->] ($(my-1-1)+(-0.14,0.21)$) -- node[above] {$\scriptstyle r-1$} ($(my-1-4)+(0.14,0.21)$);
    	\draw[<->] ($(my-1-4)+(0.21,0.14)$) -- node[right] {$\scriptstyle r-1$} ($(my-4-4)+(0.21,-0.14)$);
    	\draw[rounded corners] ($(my-1-1)+(-0.14,0.14)$) rectangle ($(my-4-4)+(0.14,-0.14)$);
    	\node at ($(my-2-2)+(0.2,-0.2)$) {$\Id^*_{r-1}$};
    	\node at ($(my-3-1)+(-1.5,-0.9)$) {$\quad\mbox{with }\,\, K=$};
    	\node at ($(my-5-2)+(-0.3,0)$) {$\scriptstyle 1$};
    	\draw[densely dotted] ($(my-5-2)+(-0.1,0)$) -- (my-5-4);
    	\node at ($(my-6-2)+(-0.4,0)$) {$\scriptstyle -1$};
    	\node at (my-5-4) {$\scriptstyle 1$};
    	\draw[densely dotted] ($(my-6-2)+(-0.1,0)$) -- ($(my-6-4)+(-0.4,0)$);
    	\node at ($(my-6-4)+(-0.1,0)$) {$\scriptstyle -1$};
    	\draw[rounded corners] ($(my-7-1)+(-0.14,0.14)$) rectangle ($(my-10-4)+(0.14,-0.14)$);
    	\draw[<->] ($(my-7-4)+(0.21,0.14)$) -- node[right] {$\scriptstyle r-1$} ($(my-10-4)+(0.21,-0.14)$);
    	\node at ($(my-7-2)+(0.2,-0.5)$) {$\Id_{r-1}$};
    	\end{tikzpicture}.
    	\] 
    	This gives Smith Normal Form, in a similar way to Lemma~\ref{lem:Q inv and MQ}, with 
    	\[
    	\begin{tikzpicture}[>=stealth,baseline=(current bounding box.center)] 
    	\begin{scope}[xscale=0.35,yscale=0.35]
    	\foreach \x in {1,...,12}
    	\foreach \y in {1,...,11}
    	{
    		\ifnum\x<\y
    		\node (my-\y-\x) at (\x,-\y) {$\phantom 0$};
    		\else
    		\node (my-\y-\x) at (\x,-\y) {$\phantom 0$};
    		\fi
    	}
    	\end{scope}
    	\draw[<->] ($(my-1-1)+(-0.14,0.21)$) -- node[above] {$\scriptstyle 2r$} ($(my-1-10)+(0.14,0.21)$);
    	\draw[<->] ($(my-1-10)+(0.24,0.21)$) -- node[right] {$\scriptstyle r+2$} ($(my-11-10)+(0.24,0.21)$);
    	\node at ($(my-6-1)+(-1,0.14)$) {$MQ=$};
    	\draw[rounded corners] ($(my-1-1)+(-0.14,0.14)$) rectangle ($(my-9-6)+(0.14,-0.14)$);
    	\node at ($(my-5-3)+(0.24,-0.14)$) {$\Id_{r+1}$};
    	\draw[rounded corners] ($(my-1-7)+(-0.14,0.14)$) rectangle ($(my-10-10)+(0.14,-0.14)$);
    	\node at ($(my-3-8)+(0.24,-0.85)$) {$0$};
    	\node at ($(my-9-1)+(0,-0.3)$) {$\scriptstyle 1$};
    	\draw[densely dotted] ($(my-9-1)+(0.25,-0.3)$) -- ($(my-9-6)+(-0.65,-0.3)$);
    	\node at ($(my-9-6)+(0,-0.31)$) {$\scriptstyle -1$};
    	\node at ($(my-9-6)+(-0.35,-0.3)$) {$\scriptstyle 1$};	
    	\end{tikzpicture}
    	\]
    	In particular, this shows that Corollary~\ref{cor: ker Z 1} also holds for $a=1$.
    \end{remark}
    
    Returning to the notation of \S\ref{sec: toric gen}, write $L$ for a spanning set for the kernel of $\varphi_\mathbb{Z}$, which by Corollary~\ref{cor: ker Z 1} can be taken to be the columns of the above matrix $K$.  As calibration, and again in the notation of \S\ref{sec: toric gen}, the associated $I_L$ in Example \ref{Example4.10} is
    \[ 
    I_L = (z_{0,0}z_{4,0} - z_{1,1}z_{3,1}z_{3,0} , ~z_{0,0}z_{2,1} - z_{1,0}z_{1,1}, ~z_{0,0}z_{3,2} - z_{2,0}z_{1,1} ).
    \]
    
    Now we saturate the ideal $I_L$, in general.
	
	\subsection{Step 2: Saturation}
	According to \cite[\S1]{CToricI}, $I_L$ can be saturated by first introducing a new indeterminate $t,$ then calculating a Gr\"obner basis of $H \colonequals  I_L + (tx_1x_2 \ldots x_n - 1),$ then afterwards eliminating the variable $t$. However, this approach makes the ideal inhomogeneous. Instead, following \cite[\S1]{CToricI}, we introduce a \emph{homogeneous} variable $u$ whose degree is equal to the sum of the degrees of the variables $x_1,\ldots ,x_n$, and then calculate the Gr\"obner basis of the ideal $H\colonequals  I_L + (x_1x_2 \ldots x_n - u) $ using the graded reverse lexicographic order.
	
	Most importantly, the two main benefits of this approach are:
	\begin{enumerate}
		\item [(a)]If $J$ is an ideal such that $I_L \subseteq J \subseteq I_M$, then instead of saturating $I_L$, we may saturate $J$, since $(I_L:(x_1 \hdots x_n)^\infty) = I_M = (J:(x_1 \hdots x_n)^\infty),$ see \cite[\S1]{CToricI}. 
		\item [(b)] Often we do not need to saturate ideals with respect to all the indeterminates, in our case we will find a much smaller subset.
	\end{enumerate}
	
	\begin{lemma}\label{satBig}
		$I_L \subseteq \mathrm{QDet}(\mathsf{z}) \subseteq I_M.$	
	\end{lemma}
	
	\begin{proof}
	Since $I_L$ are some of the $\mathrm{QDet}(\mathsf{z})$ relations starting with $z_{0,0}$ only, then $I_L \subseteq \mathrm{QDet}(\mathsf{z})$. By Proposition \ref{SurjHom}, $\mathrm{QDet}(\mathsf{z}) \subseteq I_M.$
	\end{proof}
	We will therefore saturate $\mathrm{QDet}(\mathsf{z})$ instead of $I_L,$ writing this $(\mathrm{QDet}:P^\infty)$, where $P$ is the product of all the $z_{ij}$ variables. The ideal $(\mathrm{QDet}: P^\infty)$ will be obtained by calculating the DegRevLex--Gr\"obner basis of the ideal $H^{\prime} = \mathrm{QDet}(\mathsf{z}) + (P- u) $.

		The set of the monomials $N$ in $\mathbb{C}[\mathsf{z}]$ is a basis of $\mathbb{C}[\mathsf{z}]$, considered as a vector space over $\mathbb{C}.$ So any nonzero polynomial $f \in \mathbb{C}[\mathsf{z}]$ is given as the linear combination $f = \sum_{m \in S}\upmu_m m$ of monomials, where $S \subset N,$ $S$ is finite, and $\upmu_m$ are all nonzero constants. Set $x^a \colonequals  x_1^{a_1} \ldots x_n^{a_n}$ and $x^b \colonequals  x_1^{b_1} \ldots x_n^{b_n}$ with $a = (a_1, \ldots, a_n)$,
	$b = (b_1, \ldots, b_n)$.

	\begin{definition}\cite[\S1]{maclagan2005comp}
		A term order $\succ$ on $S$ is a total order on the monomials of
		$S$ such that
		\begin{enumerate}
			\item [1.]  $x^a \succ x^b$ implies that $x^ax^c \succ x^bx^c$ for all $c \in \mathbb{N}^n$, and
			
			\item [2.] $x^a \succ x^0 = 1$ for all $a \in \mathbb{N}^n \backslash \{0\}.$
		\end{enumerate}
		
	\end{definition}
There are various different term orders on $S,$ with respect to a fixed ordering of the variables, such as $x_1 \succ x_2 \succ \hdots \succ x_n$.  In the \emph{lexicographic (lex) order}, $x^a \succ x^b$ if and only if the first nonzero entry in the vector $a - b$ is positive. 
	
	\begin{example}
		If $x \succ y \succ z,$ then with respect to lexicographic order,
		\[x^4 \succ x^2 y^2 \succ x^2yz \succ xy^3 .\]	
	\end{example}

	Further, if the polynomial ring $\mathbb{C}[x_1, \hdots, x_n]$ is graded, there are additional term orderings. Suppose $\mathbb{C}[x_1, \hdots, x_n]$ is graded by $(d_1, \hdots , d_n)$ where $deg(x_i) = d_i.$ Set $|a|\colonequals  \sum a_id_i$ and $|b|\colonequals  \sum b_id_i$. In the \emph{graded reverse lexicographic order,} $x^a \succ x^b$ if and only if either $|a| >|b|$, or $|a| = |b|$ and the last
	nonzero entry in the vector $a - b$ is negative.

	\begin{example}\label{exammonomial}
		If $x \succ y \succ z \succ w$ and weight vector $(1, 1, 1, 1)$, with respect to the graded reverse lexicographic order                                      
		$$x^2y^2z^3w  \succ x^2y^2z^2w^2  \succ  xz^4  \succ x^3.$$ The total degree comes first and the lower power of $w$ breaks the tie between the two monomials of degree 8.
	\end{example}
	
	When a monomial order $\succ$ has been chosen, the leading monomial of $f = \sum_{m \in S}\upmu_m m $ is the largest $m \in S$ with respect to $\succ$. The leading coefficient is the corresponding $\upmu_m,$ and the leading term is $\upmu_m  m.$ 
	\begin{example}
		For the polynomial ring $\mathbb{C}[t,a,b_1,b_2,c_1,c_2,c_3,d]$ with the weighting vector (13, 5, 4, 4, 3, 3, 3, 5), consider the polynomial $ac_3 - b_1b_2$. Since it is homogeneous, we look for the lower power of $c_3$ to break the tie, thus $b_1b_2  \succ ac_3$ and therefore $-b_1b_2$ is the leading term.
		
	\end{example}

Recall that $\mathrm{QDet}(\mathsf{z})$ consists of the quasiminors of the matrix
		\scriptsize\[ 
		\begin{pmatrix}
		z_{0,0}& & z_{1,0} && z_{2,0}& \cdots & &z_{m,0}\\
		&z_{1,1}\cdot \ldots \cdot z_{1,{s_1}-1}&& z_{2,1}\cdot \ldots \cdot z_{2,{s_2}-1}&&& z_{m,1}\cdot \ldots \cdot z_{m,{s_m}-1}&\\
		z_{1,s_1} & & z_{2,s_2}&& z_{3,s_3}& \cdots & &z_{m+1,s_{m+1}}
		\end{pmatrix},\]
		\normalsize		
and we are aiming to compute a Gr\"obner basis of the ideal $(\mathrm{QDet}:P^\infty)$, where $P$ is the product of all the $z_{ij}$ variables. The next Lemma allows to replace the full product $P$ with a smaller product.

\begin{lemma}\label{Lemma: replace P by E}
$(\mathrm{QDet}: P^\infty)=(\mathrm{QDet}: E^\infty)$ where $E= z_{0,0}z_{2,s_2}z_{3,s_3} \hdots z_{m+1,s_{m+1}}$.
\end{lemma}
\begin{proof}
	Since $E$ contains only some of the variables $z_{i,j},$ and $P$ contains them all, write $P=EG.$ The claim is that $(\mathrm{QDet}: (EG)^\infty)=(\mathrm{QDet}: E^\infty).$ But this follows from \cite[2.6(1)]{CToricI} provided we can show that $G$ is invertible in the localisation
	\begin{equation}
	(\mathbb{C}[\mathsf{z}]/\mathrm{QDet})_E = \mathbb{C}[\mathsf{z}]_E/\mathrm{QDet}_E. 
	\label{eqn:localisation}
	\end{equation}
	By definition $G$ contains all the $z_{i,j}$ which are not in $E.$ Now for the quasiminors in $\mathrm{QDet}(\mathsf{z}),$ if we invert $E$ we invert all variables $z_{0,0},z_{2,s_2},z_{3,s_3}, \hdots ,z_{m+1,s_{m+1}}$ in $E,$ this implies that all variables in the left hand side monomials of the quasiminor relations starting from $z_{0,0},$ namely
	\begin{align*}
	z_{0,0}z_{2,s_2} &=  z_{1,{s_1}} \cdot z_{1,{s_1}-1}\ldots z_{1,1} \cdot z_{1,0}\\
	z_{0,0}z_{3,s_3} &=  z_{1,{s_1}} \cdot z_{1,{s_1}-1}\ldots z_{1,1} \cdot z_{2,{s_2}-1}\ldots z_{2,1} \cdot z_{2,0}\\
	& \vdots\\		             
	z_{0,0}z_{m,s_{m}} &= z_{1,{s_1}} \cdot z_{1,{s_1}-1}\ldots z_{1,1} \cdot z_{2,{s_2}-1}\ldots z_{2,1} \ldots z_{m-1,s_{m-1}-1} \ldots z_{m-1,1} \cdot z_{m-1,0}\\
	z_{0,0}z_{m+1,s_{m+1}} &= z_{1,{s_1}} \cdot z_{1,{s_1}-1}\ldots z_{1,1} \cdot z_{2,{s_2}-1}\ldots z_{2,1} \ldots z_{m,{s_m}-1} \ldots z_{m,1} \cdot z_{m,0}
	\end{align*}
	are invertible modulo $\mathrm{QDet}(\mathsf{z}).$ But this implies that all the variables in the right hand side monomials become invertible in $\eqref{eqn:localisation}.$ But the monomials in the right hand side contain all variables, hence $G$ is invertible in $\eqref{eqn:localisation},$ as required.
\end{proof}

	\begin{example}
			For the group $\frac{1}{7}(1,2),$ $\mathrm{QDet}(\mathsf{z})$ consists of the quasiminors of the matrix 
			\[\begin{pmatrix}
			\textcolor{magenta}{z_{0,0}}&  &z_{1,0}&&z_{2,0}&&z_{3,0} \\
			&&&&&z_{3,1}\\
			z_{1,1}& &\textcolor{magenta}{z_{2,1}}&&\textcolor{magenta}{z_{3,2}}&&\textcolor{magenta}{z_{4,0}}\\ 
			\end{pmatrix}.\]
			We saturate $\mathrm{QDet}(\mathsf{z})$ with respect to  $E = z_{0,0}z_{2,1}z_{3,2}z_{4,0}$, which is only the coloured $z$'s.
			
	\end{example}

 The kernel of $\varphi,$ which is the toric ideal $I_M,$ is thus obtained from the saturation $(\QDet:P^\infty)=(\QDet:E^{\infty})$ of Lemma~\ref{Lemma: replace P by E}, which in turn will be obtained by eliminating $u$ in the Gr\"obner basis of the ideal $H=\QDet+ (E-u)$. 

\begin{definition}\label{def_spolynomials}
Let $f,g \in \mathbb{C}[\mathsf{z}]$ be nonzero polynomials.
\begin{enumerate}
\item  Write $\LM (f)$, $\LM (g)$ for the leading monomial of $f$ and $g$ respectively, and $\LT(f)$, $\LT(g)$ for the leading terms (i.e.\ with coefficients). Define $\upgamma = \LCM (f,g)$ to be the least common multiple of the monomials $\LM (f)$ and $\LM (g)$.
\item The S-polynomial of $f$ and $g$ is the combination
\[  
S (f, g) = \left(\frac{\upgamma}{\LT (f)}\right) f - \left(\frac{\upgamma}{\LT (g)}\right) g .
\]				
\end{enumerate}
\end{definition}
 
Recall that  $H=\QDet+ (E-u)$ is generated by the quasiminors $f_{ij}$, together with $f\colonequals E-u$.  We next grade the polynomial ring $\mathbb{C}[u,\mathsf{z}]$.    Recalling the $\mathbbold{i}$ and $\mathbbold{j}$ series in \eqref{eqn:iAndj}, for any $i$ such that $0\leq i\leq m+1$, we declare
\[
\mathrm{deg}(z_{i,j})\colonequals \mathbbold{i}_i+\mathbbold{j}_i ,
\]
which does not depend on $j$. The variable $u$ is graded so that the equation $E-u$ is homogeneous, thus 
\[
\mathrm{deg}(u)\colonequals \mathbbold{i}_0+\mathbbold{j}_0 +\sum_{t=2}^{m+1}(\mathbbold{i}_t+\mathbbold{j}_t).
\]

\begin{example}
Consider $\frac{1}{7}(1,2),$ and $f_{12}  \colonequals  z_{0,0}z_{2,1} - z_{1,0}z_{1,1}$, $f_{34}  \colonequals  z_{2,0}z_{4,0} - z_{3,0}z_{3,1}z_{3,2}$. We calculate $S(f_{12},f_{34})$ with respect to DegRevLex order, to calibrate the reader. The degree of both terms in $f_{12}$ is twelve, so the leading term is thus $- z_{1,0}z_{1,1},$ since the last nonzero entry in 
\[
(0,1,1,0,0,0,0,0,0) - (1,0,0,1,0,0,0,0,0)
\]
is negative. Similarly, the leading term of $f_{34}$ is $- z_{3,0}z_{3,1}z_{3,2},$ since the degree of both terms in $f_{34}$ is twelve, and the last nonzero entry in
\[(0,0,0,0,0,1,1,1,0) - (0,0,0,1,0,0,0,0,1)\] 
is negative. Thus $S(f_{12},f_{34})$ equals

\begin{align*}
&\frac{z_{1,0}z_{1,1}z_{3,0}z_{3,1}z_{3,2}}{- z_{1,0}z_{1,1}}f_{12}
-
\frac{z_{1,0}z_{1,1}z_{3,0}z_{3,1}z_{3,2}}{-z_{3,0}z_{3,1}z_{3,2}}f_{34}	\\ 
&= -z_{3,0}z_{3,1}z_{3,2} (z_{0,0}z_{2,1} - z_{1,0}z_{1,1}) + z_{1,0}z_{1,1} (z_{2,0}z_{4,0} - z_{3,0}z_{3,1}z_{3,2})\\
&= -z_{3,0}z_{3,1}z_{3,2}z_{0,0}z_{2,1} + z_{1,0}z_{1,1}z_{2,0}z_{4,0}.
\end{align*}

\end{example}

\medskip
\noindent
To ease notation in the following Proposition, as in \S\ref{QDetform} write
\scriptsize\[ 
\begin{pmatrix}
z_{0,0}& & z_{1,0} && z_{2,0}& \cdots & &z_{m,0}\\
&z_{1,1}\cdot \ldots \cdot z_{1,{s_1}-1}&& z_{2,1}\cdot \ldots \cdot z_{2,{s_2}-1}&&& z_{m,1}\cdot \ldots \cdot z_{m,{s_m}-1}&\\
z_{1,s_1} & & z_{2,s_2}&& z_{3,s_3}& \cdots & &z_{m+1,s_{m+1}}
\end{pmatrix}\]
\normalsize	
as
\[
\begin{pmatrix}
a_1&  &a_2&  &\hdots& & a_{m+1} \\
&W_1&&W_2&&W_{m}&\\
b_1& &b_2& &\hdots& & b_{m+1}\\ 
\end{pmatrix}.\]
Further, for any $i<j$ set $\sfm_{[i,j]}\colonequals \prod_{t=i}^{j}W_t,$ where as above $W_t = z_{t,1}\cdot \ldots \cdot z_{t,{s_t}-1}.$

\begin{prop}\label{Sreduced}
With respect to the DegRevLex order on $\mathbb{C}[u,\mathsf{z}]$, 
\[	
S(f_{ij},f_{k\ell})=\begin{cases}
-b_k \sfm_{[k,\ell-1]} a_\ell a_ib_j + b_i \sfm_{[i,j-1]}a_ja_kb_\ell &\mbox{if }i<j<k<\ell,  \\
 -b_k \sfm_{[j,\ell-1]} a_\ell a_ib_j + b_i \sfm_{[i,k-1]}a_kb_\ell a_j&\mbox{if }i<k \leq j<\ell, \\
-\sfm_{[j,\ell-1]} a_\ell a_kb_j+a_ja_kb_\ell &\mbox{if }i=k<j<\ell, \\
 b_i\sfm_{[i,k-1]}a_kb_\ell -b_ka_ib_\ell  &\mbox{if }i<k < j = \ell, \\
 b_i \sfm_{[i,k-1]} \sfm_{[\ell,j-1]}a_j a_kb_\ell - b_k a_\ell  a_ib_j   &\mbox{if }i<k<\ell <j.
\end{cases}
\]		
Furthermore, for any $i,j$
\[
\begin{array}{c}
S(f_{ij},f) = -ua_ib_j + b_i \sfm_{[i,j-1]} a_jE.
\end{array}
\]
\end{prop}
\begin{proof}
In the case  $i<j<k<\ell $, the $S$-polynomial $S(f_{ij},f_{k\ell })$ equals
\begin{align*}
	&\frac{b_i \sfm_{[i,j-1]}a_j \cdot b_k \sfm_{[k,\ell-1]} a_\ell  }{-b_i \sfm_{[i,j-1]}a_j}f_{ij}
-
	\frac{b_i \sfm_{[i,j-1]}a_j \cdot b_k \sfm_{[k,\ell-1]} a_\ell }{-b_k \sfm_{[k,\ell-1]}a_\ell}f_{k\ell }	\\ 
	&= -b_k \sfm_{[k,\ell-1]} a_\ell f_{ij} + b_i \sfm_{[i,j-1]}a_j f_{k\ell }\\
	&= 
    -b_k \sfm_{[k,\ell-1]} a_\ell \left(a_ib_j - b_i \sfm_{[i,j-\ell]}a_j\right) 
    + b_i \sfm_{[i,j-1]}a_j \left(a_kb_\ell  - b_k \sfm_{[k,\ell-1]}a_\ell \right)\\
	&= -b_k \sfm_{[k,\ell-1]} a_\ell a_ib_j + b_i \sfm_{[i,j-1]}a_ja_kb_\ell.
	\end{align*}
All other cases are similar.  For the final claim, the $S$-polynomial $S(f_{ij},f) $ equals
\begin{align*}
&\frac{ub_i \sfm_{[i,j-1]}a_j }{-b_i \sfm_{[i,j-1]} a_j}f_{ij}-
	\frac{ub_i \sfm_{[i,j-1]}a_j }{-u}f\\
	&=-u f_{ij} + b_i \sfm_{[i,j-1]} a_j f\\
	&=-u \left(a_ib_j - b_i \sfm_{[i,j-1]}a_j\right) + b_i \sfm_{[i,j-1]} a_j (E-u)\\
	&=-ua_ib_j + b_i \sfm_{[i,j-1]} a_jE.\qedhere
	\end{align*} 	
\end{proof}

\begin{definition}
A polynomial $f$ is reducible by $g$ to $r$, written $f \xrightarrow{g} r$, if $\LM (g)$ divides some monomial $m$ in $f$ and 
\[
r = f - \dfrac{\upmu_m m}{\LT (g)} \cdot g.
\] 
We say this is lead reducible if $\LM (g) \mid \LM (f)$, and 
\[
r = f - \dfrac{\LT (f)}{\LT (g)} \cdot g. 
\]
\end{definition}

\begin{definition}
	A polynomial $f$ is reducible or lead reducible by a set $G= \{g_1, \ldots , g_s\},$ denoted by 	$f \xrightarrow{G} r,$ if
\[
	f=f_1 \xrightarrow{g_{i_1}} f_2 \xrightarrow{g_{i_2}} \ldots \xrightarrow{g_{i_m}} f_m = r,
\]
and if $r$ cannot be reduced any further, then we call $r$ the normal form or remainder of $f$ modulo $G.$ 
\end{definition}
For multivariate polynomials, the remainder is not unique and this leads us to the Gr\"obner basis theory.  We will compute the Gr\"obner basis of $H= \QDet+(E-u)$ using Buchberger's algorithm.  Write $\scrS$ for the set of generators of  $\QDet$ given by all the quasiminors $f_{ij}$, together with $f=E-u$.

\begin{example}
For the group $\frac{1}{7}(1,2)$, with matrix
\[
\begin{pmatrix}
			\textcolor{magenta}{z_{0,0}}&  &z_{1,0}&&z_{2,0}&&z_{3,0} \\
			&&&&&z_{3,1}\\
			z_{1,1}& &\textcolor{magenta}{z_{2,1}}&&\textcolor{magenta}{z_{3,2}}&&\textcolor{magenta}{z_{4,0}}\\ 
\end{pmatrix},\]
the ideal $H=\QDet + (E-u)$ is generated by
\begin{align*}
			f_{12} & \colonequals  z_{0,0}z_{2,1} - z_{1,1}z_{1,0} &   f_{24} & \colonequals z_{1,0}z_{4,0} - z_{2,1}z_{3,1}z_{3,0}\\
			f_{13} & \colonequals  z_{0,0}z_{3,2} - z_{1,1}z_{2,0} & f_{34} & \colonequals  z_{2,0}z_{4,0} - z_{3,2}z_{3,1}z_{3,0}\\  
			f_{14} & \colonequals  z_{0,0}z_{4,0} - z_{1,1}z_{3,1}z_{3,0} &   f      & \colonequals  z_{0,0}z_{2,1}z_{3,2}z_{4,0} - u,\\ 
			f_{23} & \colonequals  z_{1,0}z_{3,2} - z_{2,1}z_{2,0} 
\end{align*}
and so $\scrS=\{ f_{12}, f_{13}, f_{14}, f_{23}, f_{24}, f_{34}, f\}$.
\end{example}

\begin{cor}\label{cor:reduce}
The S-polynomials in Proposition~\textnormal{\ref{Sreduced}} are reduced to zero by the set $\scrS$.
\end{cor}
\begin{proof}
In the case $i<j<k<\ell $, by \ref{Sreduced} $S(f_{ij},f_{k\ell }) = -b_k \sfm_{[k,\ell-1]} a_\ell a_ib_j + b_i \sfm_{[i,j-1]}a_ja_kb_\ell$, which has leading term $-b_k \sfm_{[k,\ell-1]} a_\ell a_ib_j$.  This leading term is divisible by $ \LT (f_{k\ell }) $ so 
	\begin{align*}
	S(f_{ij},f_{k\ell }) & \xrightarrow{f_{k\ell }} S(f_{ij},f_{k\ell }) - (a_ib_j) f_{k\ell } = b_i \sfm_{[i,j-1]}a_ja_kb_\ell  - a_ib_ja_kb_\ell. 
	\end{align*} 
   The leading term of the right hand side is $b_i \sfm_{[i,j-1]}a_ja_kb_\ell $, which is divisible by $ \LT (f_{ij})$, and thus 
   \begin{align*}
   S(f_{ij},f_{k\ell }) & \xrightarrow{f_{k\ell }}  b_i \sfm_{[i,j-1]}a_ja_kb_\ell  - a_ib_ja_kb_\ell  \xrightarrow{f_{ij}} 0.
   \end{align*} 
The next four cases in Proposition~\ref{Sreduced} are very similar, and are summarised by
\[
\begin{array}{ll}
S(f_{ij},f_{k\ell }) \xrightarrow{f_{j\ell }} b_i \sfm_{[i,k-1]}a_ja_kb_\ell  - a_jb_\ell a_ib_k \xrightarrow{f_{ik}} 0& \mbox{if }i<k \leq j<\ell \\
S(f_{ij},f_{k\ell }) \xrightarrow{f_{j\ell }} 0&\mbox{if }i=k<j<\ell \\
S(f_{ij},f_{k\ell }) \xrightarrow{f_{ik}} 0&\mbox{if }i<k < j = \ell \\
S(f_{ij},f_{k\ell }) \xrightarrow{f_{ik}} -a_\ell b_ja_ib_k + a_ib_kb_\ell \sfm_{[\ell,j-1]}a_j \xrightarrow{f_{\ell j}} 0&\mbox{if }i<k<\ell <j.
\end{array} 
\]  
Furthermore, the final case $S(f_{ij},f) =-ua_ib_j + b_i \sfm_{[i,j-1]} a_jE$ has leading term $-ua_ib_j$. This is divisible by $ \LT (f) $, and so 
\begin{align*}
S(f_{ij},f) & \xrightarrow{f} S(f_{ij},f) - (a_ib_j)f = b_i \sfm_{[i,j-1]} a_jE - a_ib_jE.
\end{align*}
The leading term of the right hand side is $b_i \sfm_{[i,j-1]} a_jE$, which is divisible by $ \LT (f_{ij})$, and  thus
\begin{align*}
S(f_{ij},f) & \xrightarrow{f} b_i \sfm_{[i,j-1]} a_jE - a_ib_jE
 \xrightarrow{f_{ij}} 0.\qedhere
\end{align*} 
\end{proof}

\begin{cor}\label{Sgrobnerbasis}
	$\scrS$ is a Gr\"obner basis for $\QDet+(E-u)$.
\end{cor}

\begin{proof}
Since by Corollary~\ref{cor:reduce} all the $S$-polynomials between elements of $\scrS$ reduce to $0$ modulo $\scrS$, this follows as an immediate consequence of Buchberger's criterion \cite[\S2]{Criterion}.
\end{proof}

\subsection{Recovering the Artin component}
For any group $\frac{1}{r}(1,a)$, the quiver of the reconstruction algebra is denoted $Q$.  Recall from \S \ref{RepVar} that $\updelta=(1,\hdots,1)$, and further $\scrR\colonequals \mathbb{C}[\Rep(\mathbb{C}Q,\updelta)]$ carries a natural action of $ G \colonequals \textstyle \prod_{q \in Q_0} \mathbb{C}^{\ast}$. The following shows that $\scrR^G,$ which is constructed using only the quiver of the reconstruction algebra, is precisely the Artin component of $\frac{1}{r}(1,a)$.

\begin{theorem}\label{thm: main1}
For any group $ \frac{1}{r}(1,a)$, there is an isomorphism $\scrR^G \cong \frac{\mathbb{C}[\mathsf{z}]}{\mathrm{QDet}(\mathsf{z})}$. 
\end{theorem}
\begin{proof}
By Proposition~\ref{SurjHom} there is a surjective homomorphism $\mathbb{C}[\mathsf{z}] \xrightarrow{\varphi} \scrR^G$. By \cite[\S4]{GrobnerSturmfel}, the kernel of $\varphi$ is a toric ideal $I_M$ of $\mathbb{C}[\mathsf{z}].$  By Corollary \ref{cor: ker Z 1}, the columns of $K$ are a spanning set $L$ for the kernel $\varphi_\mathbb{Z}$ so $I_M = (I_L:P^\infty).$ By Corollary \ref{satBig},  $(I_L: P^\infty)=(\mathrm{QDet}: P^\infty)$, and further $(\mathrm{QDet}:P^\infty) = (\mathrm{QDet}:E^\infty)$ by Lemma \ref{Lemma: replace P by E}. As explained above Definition \ref{def_spolynomials}, the toric ideal $I_M$ is thus obtained from eliminating $u$ from a Gr\"obner basis of $\QDet+(E-u),$ and thus by Corollary \ref{Sgrobnerbasis} by eliminating $u$ from $\scrS$. Therefore,
$I_M = (\QDet:E^\infty)= \scrS \cap \mathbb{C}[\mathsf{z}] = \mathrm{QDet}(\mathsf{z})$. 
\end{proof}

\section{Simultaneous Resolution}\label{SimulResolution}
In this section, the deformed reconstruction algebra is introduced, and is used to achieve simultaneous resolution. 

\subsection{The Deformed Reconstruction Algebra} \label{SecRecAlg}
In what follows, write $l_\wp$ for the number of the vertex associated to the tail of the arrow $k_\wp,$ and set
$d_\wp = l_\wp - l_{\wp-1}. $  Recall that by convention $k_0 = c_{10}$ and $k_{e-2} = a_{n0}.$
\begin{definition}\label{DefRecAlg}
	Given $r,a \in \mathbb{N}$ with $r > a > 1$ such that $(r,a)=1,$ and scalars $\boldsymbol{\lambda} \in \mathbb{C}^{\oplus \upbeta_1 } \oplus \ldots \oplus \mathbb{C}^{\oplus \upbeta_{e-2} },$ write $\boldsymbol{\lambda} = (\boldsymbol{\lambda}_1,\boldsymbol{\lambda}_2, \ldots, \boldsymbol{\lambda}_{e-2})$ with $\boldsymbol{\lambda}_i = (\lambda_{i\upbeta_{i}-1},  \ldots, \lambda_{i1}, \lambda_{i0} ).$ Then the deformed reconstruction algebra $A_{r,a, \boldsymbol{\lambda}}$ is defined to be the path algebra of the quiver $Q$ associated to the Hirzebruch--Jung continued fraction expansion of $\frac{r}{a},$ subject to the following relations (which below, we refer to as the step $i$ relations) for all $i$ such that $1 \leq i \leq e-2.$
    
    \medskip
	\noindent
	If $d_{i} = 0,$  then
	\begin{align*}
	k_{i}C_{0l_{i}} - k_{i-1}A_{0l_{i-1}} &= \lambda_{i,1}\\
	A_{0l_{i-1}}k_{{i-1}} - C_{0l_{i}}k_{i} &=  \lambda_{i,0}.
	\end{align*}
	If $d_{i} > 0,$ then
	\begin{align*}
	k_{i} C_{0l_{i}} - c_{l_{i}l_{i}-1} a_{l_{i}-1l_{i}} &=   \lambda_{i,\upbeta_{i}-1}\\
	a_{l_{i}-1l_{i}} c_{l_{i}l_{i}-1} - c_{l_{i}-1l_{i}-2} a_{l_{i}-2l_{i}-1} &=   \lambda_{i,\upbeta_{i}-2}\\
	&\vdots\\
	a_{l_{i-1}l_{i-1}+1} c_{l_{i-1}+1 l_{i-1}} - k_{{i-1}}A_{0l_{i-1}} &=   \lambda_{i,1}\\
	A_{0l_{i-1}}k_{{i-1}} - C_{0l_{i}} k_{i} &=  \lambda_{i,0}.
	\end{align*} 
\end{definition}

To simplify, write
\begin{equation}
\Delta := \{\boldsymbol{\lambda} \in \mathbb{C}^{\oplus \upbeta_1 } \oplus \ldots \oplus \mathbb{C}^{\oplus \upbeta_{e-2}} \mid \sum_{j=0}^{\upbeta_{i}-1} \lambda_{i,j} =0, ~\forall~ i = 1, \ldots, e-2\}\label{def:Delta}.
\end{equation}
Below we will be most interested in the case where the parameters $\boldsymbol{\lambda}$ in Definition \ref{DefRecAlg} belongs to $\Delta$. This will correspond to the case $\boldsymbol{\lambda}\cdot \updelta = 0$ in \cite{crawley1998noncommutative}, equivalently to the case $t=0$ in symplectic reflection algebras \cite{SRA}.

\begin{remark}
	Let $r > 1,$ and $a=1,$ and consider scalars $\boldsymbol{\lambda} \in (\mathbb{C}^{\oplus 2})^{\oplus e-2}.$ Then the deformed reconstruction algebra $A_{r,1, \boldsymbol{\lambda}}$ is defined to be the path algebra of the quiver $Q$ for $n=1$, $\upalpha_1 = r$ in \S\ref{subsec:ReconA}, subject to the following relations  		
\[\begin{array}{rcl}
a_{2}c_{1} - a_{1}c_{2} = \lambda_{1,1} & and & c_{1}a_{2} - c_{2}a_{1} = \lambda_{1,0}\\
k_{1}c_{1} - a_{2}c_{2} = \lambda_{2,1} & and & c_{1}k_{1} - c_{2}a_{2}= \lambda_{2,0}\\
k_{i-1}c_{1} - k_{i-2}c_{2} = \lambda_{i,1} & and & c_{1}k_{i-1} - c_{2}k_{i-2} = \lambda_{{i},0} ~\forall ~3 \leq i \leq e-2.
\end{array}\]	
\end{remark}

\begin{example}\label{rec3}
	In the case $\boldsymbol{\lambda} \in \Delta,$ the reconstruction algebra of Type $A_{7,3,\boldsymbol{\lambda}}$ associated to $[3,2,2]$ is the path algebra of the quiver in Example \ref{rec1} subject to the relations	
	\begin{align*}
	k_1 C_{01}  &=  c_{10} a_{01} + \lambda_{11} &  a_{30} c_{03} &= c_{32} a_{23} + \lambda_{23}  \\
	C_{01} k_1 &= a_{01} c_{10} - \lambda_{11}  &  a_{23} c_{32}  &= c_{21} a_{12} +  \lambda_{22} \\
	&                  &   a_{12} c_{21} &= k_1 a_{01} + \lambda_{21} \\
	&                     & a_{01} k_1 &=  c_{03} a_{30} - \scriptstyle \sum_{j=1}^{3}\lambda_{2,j}.
	\end{align*}
\end{example}

\begin{example} \label{rec4}
    In the case $\boldsymbol{\lambda} \in \Delta,$ the reconstruction algebra of Type $A_{165,107,\boldsymbol{\lambda}}$ associated to  $[2,3,2,4,3,2,2]$ is the path algebra of the quiver in Example \ref{rec2} subject to the relations
	\begin{align*}
	k_1 C_{02} &= c_{21} a_{12} + \lambda_{12}  & k_3 C_{04} &= k_2 A_{04} + \lambda_{31}   &  	a_{70} c_{07} &= c_{76} a_{67} + \lambda_{53} \\
	a_{12} c_{21} &= c_{10} a_{01}  + \lambda_{11}     &  A_{04} k_2 &= C_{04} k_3 - \lambda_{31}  &  a_{67} c_{76} &= c_{65} a_{56} + \lambda_{52} \\
	a_{01} c_{10} &= C_{02} k_1 - \scriptstyle \sum_{j=1}^{2}\lambda_{1,j}  &     &  & a_{56} c_{65} &= k_4 A_{05} + \lambda_{51} &\\
	&   & k_4 C_{05}  &= c_{54} a_{45} + \lambda_{42}&   A_{05} k_4 &= c_{07} a_{70} - \scriptstyle \sum_{j=1}^{3}\lambda_{5,j}.\\
	k_2 C_{04} &= c_{43} a_{34} + \lambda_{23} &a_{45} c_{54} &=  k_3A_{04} + \lambda_{41} &\\ a_{34} c_{43} &= c_{32} a_{23} + \lambda_{22}   &	A_{04} k_3 &= C_{05} k_4 - \scriptstyle \sum_{j=1}^{2}\lambda_{4,j} & \\
	a_{23} c_{32} &= k_1 A_{02} + \lambda_{21}&   &  & \\
	A_{02} k_1 &= C_{04} k_2 - \scriptstyle \sum_{j=1}^{3}\lambda_{2,j}&   &  &  
	\end{align*}                               
\end{example}

\subsection{Moduli of Deformed Reconstruction Algebras}\label{ModuliofDeformed}
With respect to the ordering of the vertices as in Section 2, fix for the rest of this paper the dimension vector $\updelta = (1,1, \ldots, 1),$ and fix the generic King stability condition $\upvartheta = (-n,1, \ldots, 1).$ Recall that
\[ \Rep(A_{r,a,\boldsymbol{\lambda}},\updelta) /\!\!\!\!/_\upvartheta \mathrm{GL} := \Proj \left(\displaystyle\bigoplus_{n \geq 0} \mathbb{C} [\Rep(A_{r,a,\boldsymbol{\lambda}},\updelta) ] ^ {G, \upvartheta^n}\right).\]

\begin{remark}
If $\boldsymbol{\lambda} \notin \Delta,$ then $\Rep(A_{r,a,\boldsymbol{\lambda}},\updelta) = \emptyset.$ Indeed, given $\boldsymbol{\lambda} \notin \Delta,$ some $\sum_{j=0}^{\upbeta_{i}-1} \lambda_{i,j} \neq 0.$ Now if $M \in \Rep(A_{r,a,\boldsymbol{\lambda}},\updelta),$ then its linear maps between vertices are scalars, which have to satisfy the relations for $A_{r,a,\boldsymbol{\lambda}}.$ Now scalars commute, and thus summing the step $i$ relations gives $\sum_{j=0}^{\upbeta_{i}-1} \lambda_{i,j} = 0,$ which is a contradiction. This is why below we always assume that $\boldsymbol{\lambda} \in \Delta.$ 
\end{remark}

\begin{definition}\label{OpenCharts}
	Let $\boldsymbol{\lambda} \in \Delta,$ and $a > 1.$ For $0 \leq t \leq n,$ define the open set $W_t$ in $\Rep(A_{r,a,\boldsymbol{\lambda}},\updelta) /\!\!\!\!/_\upvartheta \mathrm{GL}$ as follows: $W_0$ is defined by the condition $C_{01} \neq 0,$ $W_n$ by the condition $A_{0n}\neq 0,$ and for $1 \leq t \leq n-1,$ $W_t$ is defined by the conditions $C_{0t+1} \neq 0$ and $A_{0t} \neq 0.$ In the degenerate case when $a=1,$ define the open set $W_1$ by the  condition $a_1 \neq 0,$ and $W_2$ by the  condition $a_2 \neq 0.$
\end{definition}

As in {\cite[4.3]{TypeA}}, $\{W_t \mid 0 \leq l \leq n \}$ forms an open cover of $\Rep(A_{r,a,\boldsymbol{\lambda}},\updelta) /\!\!\!\!/_\upvartheta \mathrm{GL}.$

\begin{prop}\label{Opencover} For any $A_{r,a,\boldsymbol{\lambda}}$ with $a > 1$ and $\boldsymbol{\lambda} \in \Delta,$  the following statements hold
	\begin{enumerate}
		\item [(1) ]Each representation in $W_0$ is determined by $(c_{10},a_{01}) \in \mathbb{C}^2.$
		\item [(2) ]Each representation in $W_t$ is determined by $(c_{t+1t},a_{tt+1}) \in \mathbb{C}^2.$
		\item [(3) ] Each representation in $W_n$ is determined by $(c_{0n}, a_{n0}) \in \mathbb{C}^2.$
	\end{enumerate}
	Thus every open set $W_t$ in the cover is just affine space $\mathbb{A}^2.$

\end{prop}

\begin{proof}
	(1) As in {\cite[4.3]{TypeA}}, we can set $c_{0n} = c_{nn-1} = \ldots = c_{21} = 1.$ First, consider the Step 1 relations. 
		
		\smallskip
		If $d_1 = 0,$ then the relations become
		\begin{align*}
		k_1 - c_{10}a_{01} &= \lambda_{1,1}\\
		a_{01}c_{10} - k_1 &=  - \lambda_{1,1}.
		\end{align*}
		Since $a_{01}, c_{10}, k_1$ are scalars, the bottom follows from the top and $k_1$ is in terms of $(c_{10},a_{01})$ with no further relations between $c_{10}$ and $a_{01}.$ If $d_1 > 0,$ then
		\begin{align*}
		k_1 - a_{l_1-1l_1} &=   \lambda_{1,\upbeta_{1}-1}\\
		a_{l_1-1l_1} - a_{l_1-2l_1-1} &=   \lambda_{1,\upbeta_{1}-2}\\
		&\vdots\\
		a_{12} - c_{10}a_{01} &=   \lambda_{1,1}\\
		a_{10}c_{10} - k_1 &=   - \sum_{j=1}^{\upbeta_{1}-1}\lambda_{1,j}.
		\end{align*} 
		The last relation follows by summing the other relations. It is furthermore clear that $k_1$ and all the anticlockwise arrows between vertex 1 and $l_1$ are determined by $(c_{10},a_{01}).$ 
		
		\medskip
		\noindent
		By induction, we can assume that all the anticlockwise arrows between vertex 0 and $l_i$ are determined by $(c_{10}, a_{01})$, as are $k_1, \ldots, k_i$ and furthermore the Step $1, \ldots, i$ relations hold with no further relations between $c_{10}$ and $a_{01}$.
		
		\medskip
		\noindent
		We next establish the induction step, by considering the Step $i+1$ relations. If $d_{i+1}=0$ then the Step $i+1$ relations become
		\begin{align*}
		k_{i+1} - k_iA_{0l_i} &= \lambda_{i+1,1}\\
		A_{0l_i}k_i - k_{i+1}&= -\lambda_{i+1,1}.
		\end{align*}
		The bottom comes from the top and $k_{i+1}$ is in terms of $A_{0l_i}$ and $k_i,$ which by induction are determined by $(c_{10}, a_{01}).$ If $d_{i+1} > 0,$ then
		\begin{align*}
		k_{i+1} - a_{l_{i+1}-1l_{i+1}} &=   \lambda_{i+1,\upbeta_{i+1}-1}\\
		a_{l_{i+1}-1l_{i+1}} -  a_{l_{i+1}-2l_{i+1}-1}  &=  \lambda_{i+1,\upbeta_{i+1}-2}\\
		&\vdots\\	
		a_{l_i,l_i+1} - k_iA_{0l_i} &=  \lambda_{i+1,1}\\
		A_{0l_i}k_i - k_{i+1} &=   - \sum_{j=1}^{\upbeta_{i+1}-1}\lambda_{i+1,j}.
		\end{align*}
		
		\noindent
		The last relation follows by summing the other relations. It is furthermore clear that $k_{i+1}$ and all the anticlockwise arrows between vertex $l_i$ and $l_{i+1}$ are determined by $(c_{10},a_{01}).$ Thus by induction, all arrows are determined by $(c_{10}, a_{01}) \in \mathbb{C}^2.$ 
		
    \medskip
	\noindent
	(2) As in {\cite[4.3]{TypeA}}, we can set $c_{0n} = \ldots = c_{t+2t+1} = 1  =  a_{01} = \ldots = a_{t-1t}$  and show that all the arrows are determined by $(c_{t+1t},a_{tt+1}).$ Let $s-1: = \text{max} \{ j \mid l_j \leq t \},$ and $s : = \text{min} \{ j \mid l_j \geq t+1 \}.$	We start with the anticlockwise direction from vertex $l_s$ to vertex $0$, and then clockwise from vertex $l_{s-1}$ to vertex $0$ in the diagram below.
	
	
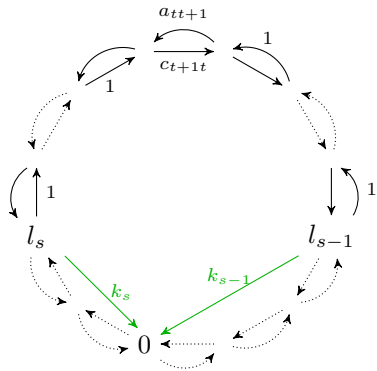
\[
	\begin{tikzpicture} [bend angle=45, looseness=1]
	\node[name=s,regular polygon, regular polygon sides=12, minimum size=4cm] at (0,0) {}; 
	\node (C1) at (s.corner 1)  {};
	\node (C2) at (s.corner 2)  {};
	\node (C3) at (s.corner 3)  {};
	\node (C4) at (s.corner 4)  {};
	\node (C5) at (s.corner 5)  {$l_s$};
	\node (C6) at (s.corner 6)  {};
	\node (C7) at (s.corner 7)  {$ 0$};
	\node (C8) at (s.corner 8)  {}; 
	\node (C9) at (s.corner 9)  {};
	\node (C10) at (s.corner 10)  {$l_{s-1}$};
	\node (C11) at (s.corner 11)  {};
	\node (C12) at (s.corner 12)  {};   
	\draw[->] (C5) -- node[right]  {$\scriptstyle 1$} (C4);
	\draw[densely dotted,->] (C4) -- node[]  {} (C3); 
	\draw[->] (C3) -- node[below]  {$\scriptstyle 1$} (C2);
	\draw[->] (C2) -- node[below]  {$\scriptstyle c_{t+1t}$} (C1);
	\draw[->] (C1) -- node  {} (C12);
	\draw[densely dotted,->] (C12) -- node  {} (C11);
	\draw[->] (C11) -- node  {} (C10);
	\draw[densely dotted,->] (C10) -- node  {} (C9);
	\draw[densely dotted,->] (C9) -- node  {} (C8);
	\draw[densely dotted,->] (C8) -- node  {} (C7);
	\draw[densely dotted,->] (C7) -- node  {} (C6);
	\draw[densely dotted,->] (C6) -- node  {} (C5);
	\draw[->] [green!70!black](C10) -- node [above]  {$\scriptstyle k_{s-1}$} (C7);
	\draw[->, transform canvas={xshift=0.8ex}] [green!70!black] (C5) -- node[right]  {$\scriptstyle k_{s}$} (C7) ;
	\draw [->,bend right] (C1) to node [above]  {$\scriptstyle a_{tt+1}$} (C2);
	\draw [->,bend right] (C2) to node  {} (C3);
	\draw [densely dotted,->,bend right] (C3) to node {} (C4);
	\draw [->,bend right] (C4) to node {} (C5);
	\draw [densely dotted,->,bend right] (C5) to node {} (C6);
	\draw [densely dotted,->,bend right] (C6) to node {} (C7);
	\draw [densely dotted,->,bend right] (C7) to node  {} (C8);
	\draw [densely dotted,->,bend right] (C8) to node  {} (C9);
	\draw [densely dotted,->,bend right] (C9) to node  {} (C10);
	\draw [->,bend right] (C10) to node [right]  {$\scriptstyle 1$} (C11);
	\draw [densely dotted,->,bend right] (C11) to node  {} (C12);
	\draw [->,bend right] (C12) to node [above]  {$\scriptstyle 1$} (C1);
	\end{tikzpicture} \]
	
	First consider the Step $s$ relations. We claim that $k_{s-1},~k_s$ and all the arrows in between $l_{s-1}$ and $l_s$ are determined by $(c_{t+1t},a_{tt+1})$. Since $d_s > 0,$ the relations become
		\begin{align*}
	k_{s} - a_{l_{s}-1l_{s}} &=   \lambda_{s,\upbeta_{s}-1}\\
	a_{l_{s}-1l_{s}} - a_{l_{s}-2l_{s}-1} &=   \lambda_{s,\upbeta_{s}-2}\\
	&\vdots\\
	a_{t+1t+2} - c_{t+1t}a_{tt+1} &=   \lambda_{s,(t+1)-l_{s-1}+1}\\
	a_{tt+1}c_{t+1t} - c_{tt-1} &=   \lambda_{s,(t+1)-l_{s-1}}\\
	&\vdots\\
	c_{l_{s-1}+2l_{s-1}+1} - c_{l_{s-1}+1l_{s-1}} &=   \lambda_{s,2}\\
	c_{l_{s-1}+1l_{s-1}} - k_{s-1} &=   \lambda_{s,1}\\
	k_{s-1} - k_{s} &=   - \sum_{j=1}^{\upbeta_{s}-1}\lambda_{s,j}.
	\end{align*} 
	The last relation follows by summing the other relations. It is furthermore clear that $k_{s},~ k_{s-1}$ with all the anticlockwise and clockwise arrows between vertex $l_s$ and $l_{s-1}$ are determined by $(c_{t+1t},a_{tt+1}),$ and there are no additional relations between $c_{t+1t}$ and $a_{tt+1}.$
	
	\medskip
	\noindent\emph{Anticlockwise.}
	Hence by induction, we can assume that all the anticlockwise arrows between vertex $l_s$ and $l_p$ are determined by $(c_{t+1t},a_{tt+1})$, as are $k_s, \ldots, k_p$ and furthermore the Step $s, \ldots, p$ relations hold with no further relations between $c_{t+1t}$ and $a_{tt+1}$.
	
	\smallskip
	We next establish the induction step, by considering the Step $p+1$ relations. If $d_{p+1} = 0,$ then the relations become
	\begin{align*}
	k_{p+1} - k_{p}A_{0l_p} &= \lambda_{p+1,1}\\
	A_{0l_p}k_{p} - k_{p+1} &=  - \lambda_{p+1,1}.
	\end{align*}
	and therefore $k_{p+1}$ can be determined by $(c_{t+1t},a_{tt+1})$. If $d_{p+1} > 0,$ then
	\begin{align*}
	k_{p+1} - a_{l_{p+1}-1l_{p+1}} &=   \lambda_{p+1,\upbeta_{p+1}-1}\\
	a_{l_{p+1}-1l_{p+1}} - a_{l_{p+1}-2l_{p+1}-1} &=   \lambda_{p+1,\upbeta_{p+1}-2}\\
	&\vdots\\
	a_{l_pl_p+1} - k_{p}A_{0l_p} &=   \lambda_{p+1,1}\\
	A_{0l_p}k_{p} - k_{p+1} &=   - \sum_{j=1}^{\upbeta_{p+1}-1}\lambda_{p+1,j}.
	\end{align*} 
	The last relation follows by summing the other relations. It is furthermore clear that $k_{p+1}$ and all the anticlockwise arrows between vertex $l_p$ and $l_{p+1}$ are determined by $(c_{t+1t},a_{tt+1}),$ and there are no additional relations between $c_{t+1t}$ and $a_{tt+1}.$
	
	\medskip
	\noindent\emph{Clockwise.}
	Similar to the above, we can assume by induction that all the clockwise arrows between vertex $l_{s-1}$ and $l_{q-1}$ are again determined by $(c_{t+1t},a_{tt+1})$, as are $k_{s-1}, \ldots, k_{q-1}$ and furthermore the Step $q, \ldots, s$ relations hold with no further relations between $c_{t+1t}$ and $a_{tt+1}$.
	
	\smallskip
	We then establish the induction step, by considering the Step $q-1$ relations. If $d_{q-1} = 0,$ then the relations become
	\begin{align*}
	k_{q-1}C_{0l_{q-1}} - k_{q-2} &= \lambda_{q-1,1}\\
	k_{q-2} - C_{0l_{q-1}}k_{q-1} &=  - \lambda_{q-1,1}.
	\end{align*}
	and therefore $k_{q-2}$ can be determined by $(c_{t+1t},a_{tt+1})$. If $d_{q-1} > 0,$ then
	\begin{align*}
	k_{q-1}C_{0l_{q-1}} - c_{l_{q-1}l_{q-1}-1} &= \lambda_{q-1,\upbeta_{q-1}-1}\\
	c_{l_{q-1}l_{q-1}-1} - c_{l_{q-1}-1l_{q-1}-2} &=   \lambda_{q-1,\upbeta_{q-1}-2}\\
	&\vdots\\
	c_{l_{q-2}+1l_{q-2}} - k_{q-2} &=   \lambda_{q-1,1}\\
	k_{q-2} - C_{0l_{q-1}}k_{q-1} &=   - \sum_{j=1}^{\upbeta_{q-1}-1}\lambda_{q-1,j}.
	\end{align*} 
	The last relation follows by summing the other relations. It is furthermore clear that $k_{q-2 }$ and all the clockwise arrows between vertex $l_{q-1}$ and $l_{q-2}$ are determined by $(c_{t+1t},a_{tt+1}).$ 
	
	(3) The proof for $W_n$ is very similar to $W_0$ but instead starts at the Step $e-2$ relations and work backwards to the Step $1$ relations. Thus by induction, all arrows are determined by $(c_{t+1t}, a_{tt+1}) \in \mathbb{C}^2.$ 	\qedhere	
\end{proof}	

\begin{remark} In the degenerate case when $a=1,$ a similar proof of Proposition \ref{Opencover} shows that each representation in $W_1$ is determined by $(c_{1},a_{2}) \in \mathbb{C}^2,$ whilst each representation in $W_2$ is determined by $(c_{1},a_{1}) \in \mathbb{C}^2.$ Again, even in the degenerate case $a=1,$ each open set $W_i$ in the open cover is just affine space $\mathbb{A}^2.$	
\end{remark}

\begin{cor}\label{ResCor}
	For any $A_{r,a,\boldsymbol{\lambda}},$ for the fixed $\upvartheta = (-n,1, \ldots, 1),$ \[\Rep(A_{r, a, \boldsymbol{\lambda}},\updelta) /\!\!\!\!/_\upvartheta \mathrm{GL}  \rightarrow \Rep(A_{r, a, \boldsymbol{\lambda}},\updelta) /\!\!/ \mathrm{GL}\]
	is a resolution of singularities. 
\end{cor}

\begin{proof}
	The morphism is projective birational by construction, and the fact that the variety $\Rep(A_{r, a, \boldsymbol{\lambda}},\updelta) /\!\!\!\!/_\upvartheta \mathrm{GL}$ is regular follows from Proposition \ref{Opencover}, since each chart $W_t$ in the open cover is regular.
\end{proof}

\subsection{Simultaneous Resolution} \label{SimulRes}
Write $(\upalpha_{0,0},\upalpha_{1,0},\ldots,\upalpha_{1,\upbeta_{1}-1},\ldots, \upalpha_{e-1,0})$ for the point in Spec $\left(\frac{\mathbb{C}[\mathsf{z}]}{\mathrm{QDet}(\mathsf{z})}\right)$ corresponding to the maximal ideal $(z_{0,0} - \upalpha_{0,0},\ldots, z_{e-1,0} - \upalpha_{e-1,0}).$ Let $Q$ be the quiver of the reconstruction algebra, and consider the map 
\[ \uppi \colon \Rep(\mathbb{C}Q,\updelta) /\!\!/ \mathrm{GL} = \frac{\mathbb{C}[\mathsf{z}]}{\mathrm{QDet}(\mathsf{z})} \rightarrow \Delta,\]
defined by taking 
\[
\begin{tikzpicture}
\node (A) at (0,0) {$(\upalpha_{0,0},\upalpha_{1,0},\ldots,\upalpha_{1,\upbeta_{1}-1},\ldots, \upalpha_{e-1,0})$};
\node (a) at (0,-2) {$(\upalpha_{i,0} - \upalpha_{i,1}, \upalpha_{i,1} - \upalpha_{i,2}, \ldots, \upalpha_{i,\upbeta_{i}-1} - \upalpha_{i,0})^{e-2}_{i=1}.$};
\draw[|->] (A)--(a);
\end{tikzpicture}
\]

\begin{example} 
For the group $\frac{1}{7}(1,3)$ as in Example \ref{rec1} and \ref{gen 1},  the morphism \[ \Rep(\mathbb{C}Q,\updelta) /\!\!/ \mathrm{GL} \rightarrow \Delta\] is given by
\[
\begin{tikzpicture}
\node (A) at (0,0) {$(\upalpha_{0,0}, \upalpha_{1,0},\upalpha_{1,1},\upalpha_{2,0}, \upalpha_{2,1}, \upalpha_{2,2}, \upalpha_{2,3},  \upalpha_{3,0})$};
\node (a) at (0,-2) {$((\upalpha_{1,0} - \upalpha_{1,1}, \upalpha_{1,1} - \upalpha_{1,0}),(\upalpha_{2,0} - \upalpha_{2,1},\upalpha_{2,1} - \upalpha_{2,2},\upalpha_{2,2} - \upalpha_{2,3},\upalpha_{2,3} - \upalpha_{2,0})).$};
\draw[|->] (A)--(a);
\end{tikzpicture}
\]
The fibre above $((\lambda_{1,1} ,\lambda_{1,0}),  (\lambda_{2,3}, \lambda_{2,2}, \lambda_{2,1}, \lambda_{2,0}  )) \in \Delta$ is the zero locus of 
\begin{align*}
z_{1,0} - z_{1,1} &=  \lambda_{1,1} &  z_{2,0} - z_{2,1} &= \lambda_{23}  \\
z_{1,1} - z_{1,0} &=  \lambda_{1,0} = -\lambda_{1,1}  &  z_{2,1} - z_{2,2} &= \lambda_{22} \\
&                  &   z_{2,2} - z_{2,3} &= \lambda_{21} \\
&                     & z_{2,3} - z_{2,0} &= \lambda_{20} = -\lambda_{21} -\lambda_{22} -\lambda_{23},
\end{align*}
which is $\Rep(A_{7, 3, \boldsymbol{\lambda}},\updelta) /\!\!/ \mathrm{GL}.$
\end{example}

 \begin{remark}\label{fibreabovelamda}
	The fibre above a point $\boldsymbol{\lambda}  \in \Delta$ is precisely $\Rep(A_{r, a, \boldsymbol{\lambda}},\updelta) /\!\!/ \mathrm{GL}.$ Indeed, the fibre above $\boldsymbol{\lambda} \in \Delta$ is the zero locus of 
	\begin{align*}
	z_{i,0} - z_{i,1} &=   \lambda_{i,\upbeta_{i}-1}\\
	z_{i,1} - z_{i,2}  &=  \lambda_{i,\upbeta_{i}-2}\\
	&\vdots\\	
	z_{i,\upbeta_{i}-1} - z_{i,0} &= - \sum_{j=1}^{\upbeta_{i}-1}\lambda_{i,j}
	\end{align*}
for all $i$ such that $1 \leq i \leq e-2$.  By (\ref{eq1}) and Definition \ref{DefRecAlg}, this equals $\Rep(A_{r, a, \boldsymbol{\lambda}},\updelta) /\!\!/ \mathrm{GL}.$
\end{remark}

\begin{theorem}\label{thm: main2}
The diagram 
\[
\begin{tikzpicture}
\node (A) at (0,0) {$\Rep(\mathbb{C}Q,\updelta) /\!\!\!\!/_\upvartheta \mathrm{GL}$};
\node (B) at (4,0) {$\Rep(\mathbb{C}Q,\updelta) /\!\!/ \mathrm{GL}$};
\node (b) at (4,-2) {$\Delta$};
\draw[->] (A)-- node[above]  {} (B);
\draw[densely dotted,->] (A)-- node[below]  {$\upphi$} (b);
\draw[->] (B)-- node[right]  {$\uppi$} (b);
a\end{tikzpicture}
\]
is a simultaneous resolution of singularities in the sense that the morphism $\upphi$ is smooth, and $\uppi$ is flat.
\end{theorem}

\begin{proof}
	Write $\upphi$ for the composition
\[Y= \Rep(\mathbb{C}Q,\updelta) /\!\!\!\!/_\upvartheta \mathrm{GL} \rightarrow \Rep(\mathbb{C}Q,\updelta) /\!\!/ \mathrm{GL} \rightarrow \Delta.\]
 	We first claim that $\upphi$ is flat. Since $(1)~\Delta$ is regular, $(2)~Y$ is regular (so Cohen-Macaulay) since $\mathbb{C}Q$ is free, so the analogue of the open charts $W_t$ in Definition \ref{OpenCharts} are clearly all affine spaces, $(3)~\mathbb{C}$ is algebraically closed so $\upphi$ takes closed points of $Y$ to closed points of $\Delta,$ and $(4)$ for every closed point $\boldsymbol{\lambda}  \in \Delta,$ for the same reason as in Remark \ref{fibreabovelamda} the fibre $\upphi^{-1}(\boldsymbol{\lambda})$ is $\Rep(A_{r, a, \boldsymbol{\lambda}},\updelta) /\!\!\!\!/_\upvartheta \mathrm{GL}$  which is always two-dimensional by Proposition \ref{Opencover}, it follows  from \cite[Corollary to 23.1]{Matsumura} that $\upphi$ is flat.
	
	Now as in {\cite[3.35]{Liu}} to show that $\upphi$ is smooth, we just require smoothness (equivalently regularity, as we are working over $\mathbb{C}$) at closed points of fibres above closed points $\boldsymbol{\lambda} \in \Delta.$ But as above $\upphi^{-1}(\boldsymbol{\lambda})$ is $\Rep(A_{r, a, \boldsymbol{\lambda}},\updelta) /\!\!\!\!/_\upvartheta \mathrm{GL},$ which is regular at all closed points by Proposition \ref{Opencover}. Thus $\upphi$ is a smooth morphism, as required. 
	
	Finally, the above can be adapted to show that $\uppi$ is flat. We have that $\uppi^{-1}(\boldsymbol{\lambda})=\Rep(A_{r, a, \boldsymbol{\lambda}},\updelta) /\!\!/ \mathrm{GL},$ which is always two-dimensional as a consequence of the resolution of its singularities computed in Proposition \ref{Opencover}. Thus we can still appeal to \cite[Corollary to 23.1]{Matsumura}.
\end{proof}
	
\begin{remark}\label{genericstability}
	The choice of $\upvartheta = (-n,1, \ldots, 1)$ is important. For Kleinian singularities, it is possible to use any generic stability \cite{OnKleinianSing}. In the more general setting here, other stability parameters do not give simultaneous resolution on the nose, as the following Example demonstrates.
	
	\begin{example}
		Consider the group $\frac{1}{3}(1,1),$ with the generic stability condition $\upvartheta_2 = (1,-1)$ and the dimension vector $(1,1).$ Then $\Rep(A_{r, a, 0},\updelta) /\!\!\!\!/_{\upvartheta_2} \mathrm{GL}$ is covered by three affine charts, namely $ U_0 = (a_1 \neq 0),~ U_1 = (a_2 \neq 0),$ and $ U_2 = (k_1 \neq 0).$ If we consider the first chart $U_0,$ we can base change such that $a_1 = 1,$ which gives
		
		\begin{figure}[H]
		\centering
		\begin{tikzpicture}[scale=1.9] 
		\node (A) at (-1,0) {$\mathbb{C}$};
		\node (B) at (0,0) {$\mathbb{C}$};  
		\draw[transform canvas={yshift=2.5ex},->] (A) -- node [gap]  {$\scriptstyle c_{1}$} (B);
		\draw[transform canvas={yshift=1.0ex},->] (A) -- node [gap]  {$\scriptstyle c_{2}$} (B);
		\draw[transform canvas={yshift=-1.0ex},<-] (A) -- node [gap]  {$\scriptstyle 1$} (B);    
		\draw[transform canvas={yshift=-2.5ex},<-] (A) -- node [gap]  {$\scriptstyle a_{2}$}(B);
		\draw[transform canvas={yshift=-4.0ex},<-][green!70!black] (A) -- node [gap]  {$\scriptstyle k_{1}$} (B);
		\end{tikzpicture}
	\end{figure}
    \noindent
        subject to relations
        \begin{align*}
              c_1a_2 &= c_2  &a_2c_1 &= c_2\\
              c_1k_1 &= c_2a_2  & k_1c_1 &= a_2c_2 .
        \end{align*}
		
		\noindent
		This chart is parameterised by the variables $c_1,~ a_2,~k_1,$ subject to the relation $c_1k_1 = c_1a_2^2,$ i.e. $c_1(k_1-a_2^2)=0,$ which is singular. Thus the fibre $\Rep(\mathbb{C}Q,\updelta) /\!\!\!\!/_{\upvartheta_2} \mathrm{GL}$ above the origin of the corresponding $\upphi$ is singular, and so is not a simultaneous resolution.

	\end{example}
\end{remark}

\end{document}